\documentclass[11pt]{article}
\usepackage[left=3cm,top=2.5cm,right=3cm]{geometry}             
\usepackage{graphicx,caption,subcaption}

\usepackage{enumerate}
\usepackage{appendix}

\usepackage{amssymb,amsmath,amsthm}

\newcommand{\RR}{\mathbb{R}}
\newcommand{\EE}{\mathbb{E}}
\newcommand{\CC}{\mathbb{C}}

\newcommand{\NN}{\mathbb{N}}

\newcommand{\BO}{\mathcal{O}}
\newcommand{\JJ}{\mathcal{J}}

\newcommand{\MM}{\mathbf{M}}
\newcommand{\YY}{\mathbf{Y}}
\newcommand{\afg}{\mathbf{g}}

\newcommand{\DIV}{\,\text{div}}
\newcommand{\VOL}{\,d\text{Vol}}

\usepackage{color} 
\usepackage[normalem]{ulem}                        

\makeatletter
\newtheorem*{rep@theorem}{\rep@title}
\newcommand{\newreptheorem}[2]{%
\newenvironment{rep#1}[1]{%
 \def\rep@title{#2 \ref{##1}}%
 \begin{rep@theorem}}%
 {\end{rep@theorem}}}
\makeatother

\newtheorem{theorem}{Theorem}[section]
\newreptheorem{theorem}{Theorem}

\newtheorem{lemma}[theorem]{Lemma}
\newreptheorem{lemma}{Lemma}

\newtheorem{prop}[theorem]{Proposition}
\newreptheorem{prop}{Proposition}

\newtheorem{corollary}[theorem]{Corollary}
\newreptheorem{corollary}{Corollary}

\newtheorem{defn}[theorem]{Definition}
\newreptheorem{defn}{Definition}

\theoremstyle{definition}
\newtheorem{remark}[theorem]{Remark}

\newtheorem*{question}{Question}

\newtheorem*{conjecture}{Conjecture}

\author{Spyros Alexakis
 \\University of 
Toronto \and Tracey Balehowsky \\University 
of Toronto \and Adrian Nachman\\University of Toronto }

\title{Determining a Riemannian Metric from Minimal Areas}
\date{\today}                                         

\begin{document}

\maketitle

\begin{abstract}
We prove that if $(M,g)$ is a topological 3-ball with a  $C^4$-smooth Riemannian metric $g$,
 and
 mean-convex boundary $\partial M$ then knowledge of least areas circumscribed by simple 
 closed curves $\gamma \subset \partial M$ uniquely determines the metric $g$,
 under some additional geometric assumptions. 
These are that $g$ is  either a) $C^3$-close to Euclidean or b) satisfies much weaker geometric conditions which hold when the manifold is
to a sufficient degree either \emph{thin}, or \emph{straight}. 

 In fact, the least area data that we require 
 is for  a much more restricted class of curves $\gamma\subset \partial M$. We also prove a 
 corresponding local result: assuming only that $(M,g)$ has strictly 
mean convex boundary at a point $p\in\partial M$, we prove that knowledge of the least areas 
circumscribed by any simple closed curve $\gamma$ in a neighbourhood $U\subset \partial M$ 
of $p$ uniquely determines the metric near $p$. Additionally, we sketch the proof of a global result with no thin/straight or curvature condition, but assuming the metric admits minimal foliations ``from all directions''.  

The proofs rely on finding the metric along a 
continuous sweep-out of $M$ by area-minimizing surfaces; they  
bring together ideas from 
the 2D-Calder\'on inverse problem, minimal surface theory, and the careful analysis of 
a system of pseudo-differential equations.
\end{abstract}

\tableofcontents

\section{Introduction}


The classical boundary rigidity problem in differential geometry asks whether knowledge of the 
distance between any two points on the boundary of a Riemannian manifold is sufficient to 
identify the metric up to isometries that fix the boundary. Manifolds for which this is the case 
are called \emph{boundary rigid}. One motivation for the problem comes from seismology, if one seeks to determine the interior structure of the Earth from measurements of travel times of seismic waves. There are counterexamples to boundary rigidity: intuitively, if the manifold has a region of large positive curvature in the interior, length-minimizing geodesics between boundary points need not pass through this region. A way to rule out such counterexamples is to assume the manifold is simple, meaning that any two points can be joined by a unique minimizing geodesic and that the boundary is strictly convex. Michel \cite{Michel} conjectured that simple manifolds are boundary rigid. Special cases have been proved by Michel \cite{Michel}, Gromov \cite{gromov}, Croke \cite{croke}, Lassas, Sharafutdinov, and Uhlmann \cite{lassas2003semiglobal}, Stefanov and Uhlmann \cite{stefanov2005boundary}, and Burago and Ivanov \cite{BuragoIvanov, BuragoIvanov2}. In two dimensions, the conjecture was settled by Pestov and Uhlmann \cite{PestovUhlmann}. Moving away from the simplicity assumption, important recent work of Stefanov, Uhlmann and Vasy solved a local version of the rigidity problem in a neighbourhood of any strictly convex point of the boundary, and obtained a corresponding global rigidity result for manifolds that admit a foliation satisfying a certain convexity condition. (see \cite{SVU2} and earlier references given there).

In this paper, we consider the following higher dimensional version of the boundary rigidity problem, where in lieu of lengths of geodesics the data consists of areas of minimal surfaces. 
\begin{question}
 Given any simple closed curve $\gamma$ on the boundary of a Riemannian 3-manifold $(M,g)$, suppose the area of any least-area surface $Y_\gamma\subset M$ circumscribed by the curve is known. Does this information determine the metric $g$? 
 \end{question}

 Under certain geometric conditions, we show that the answer is yes. In some cases, we only require the area data for a much smaller subclass of curves on the boundary.
  \begin{figure}[h]
\centering
 \includegraphics[width=0.4\textwidth]{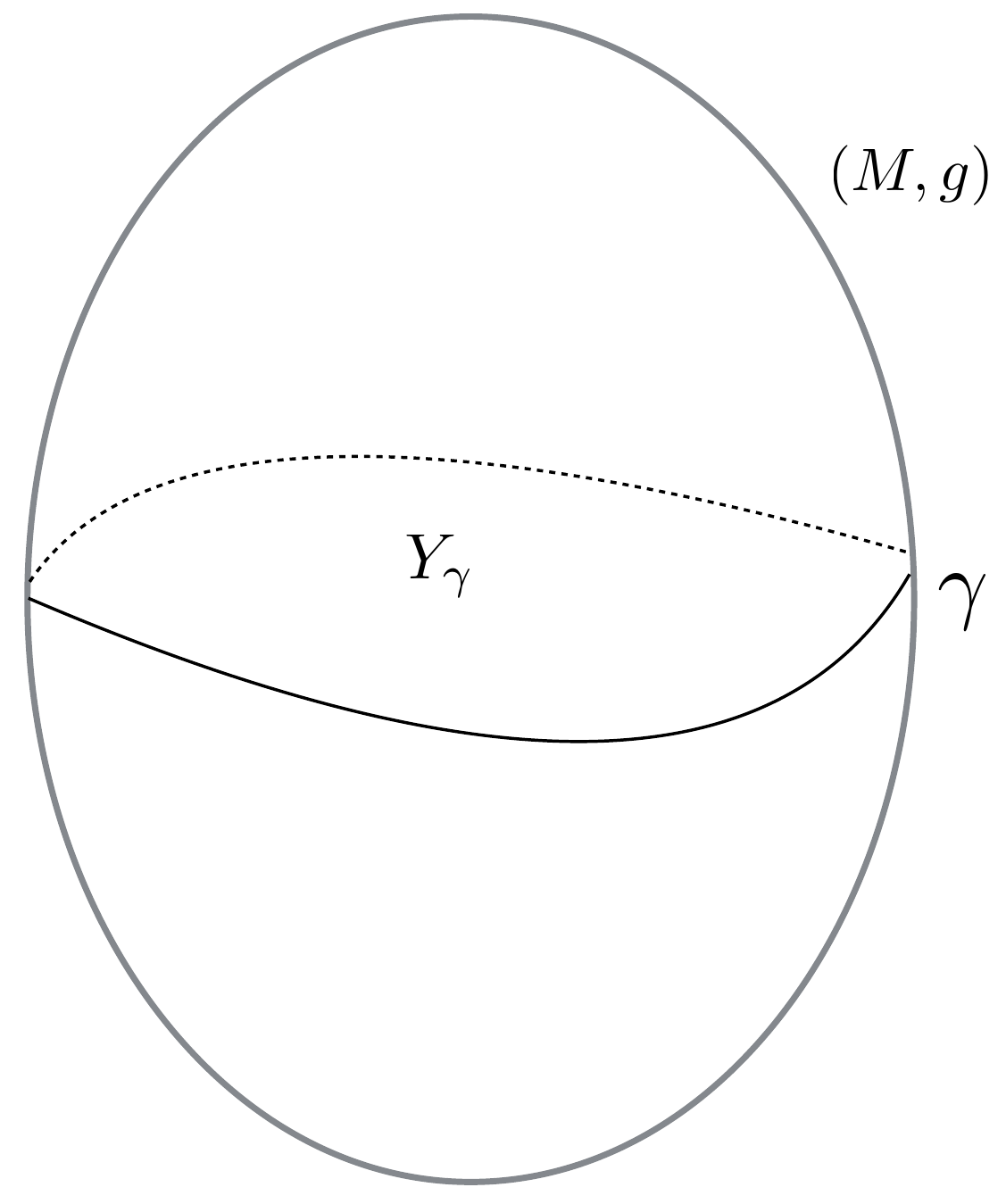}
  \caption{Simple closed curve $\gamma$ on $\partial M$.}
 \end{figure}
 
Theories posited by the AdS/CFT correspondence also provide strong physics motivation to consider the problem of using knowledge of the areas of certain submanifolds to determine the metric. Loosely speaking, the AdS/CFT correspondence states that the dynamics of $(n+2)$-dimensional supergravity theories modelled on an Anti-de Sitter (AdS) space are equivalent to the quantum physics modelled by $(n+1)$-dimensional conformal field theories (CFT) on the boundary of the AdS space (see \cite{Mald1}). This equivalence is often referred to as holographic duality. Analogous to the problem of boundary rigidity, one of the main goals of the AdS/CFT correspondence is to use conformal field theory information on the boundary to determine the metric on the AdS spacetime which encodes information about the corresponding gravity dynamics. 

Towards this goal, Maldecena \cite{Mald2} has proposed that given a curve on the boundary of 
the AdS spacetime, the renormalized area of the minimal surface bounding the curve contains 
information about the expectation value of the Wilson loop associated to the curve. More 
recently, Ryu and Takayanagi \cite{Ryu-T} have conjectured that given a region $A$ on the 
boundary of an $(n+2)$-dimensional AdS spacetime, the entanglement entropy of $A$ is 
equivalent to the renormalized area of the area-minimizing surface $Y_\gamma$ with boundary 
$\gamma=\partial A$ (see Figure \ref{ads-cft-fig}). The AdS/CFT correspondence is often 
studied in Riemannian signature, where one must consider a Riemannian, asymptotically 
hyperbolic manifold $(M,g)$, for which one knows information on the boundary and the 
renormalized area for minimal surfaces bounded by closed loops on the boundary. Hence, 
one is led to the question: Given 
a collection of simple closed curves on the boundary-at-infinity of $(M,g)$, does knowledge of 
the renormalized area of the area-minimizing surfaces bounded by these curves allow us to 
recover the metric? And even locally: Does knowledge of renormalized
areas of loops lying in a given domain $V\subset M$ allow one to reconstruct the
bulk metric, and in how large a (bulk) region is the reconstruction possible? Answers in the affirmative may provide new methods to describe the \
relationship between gravity theories and conformal field theories. 
  \begin{figure}[h]
\centering
 \includegraphics[width=0.38\textwidth]{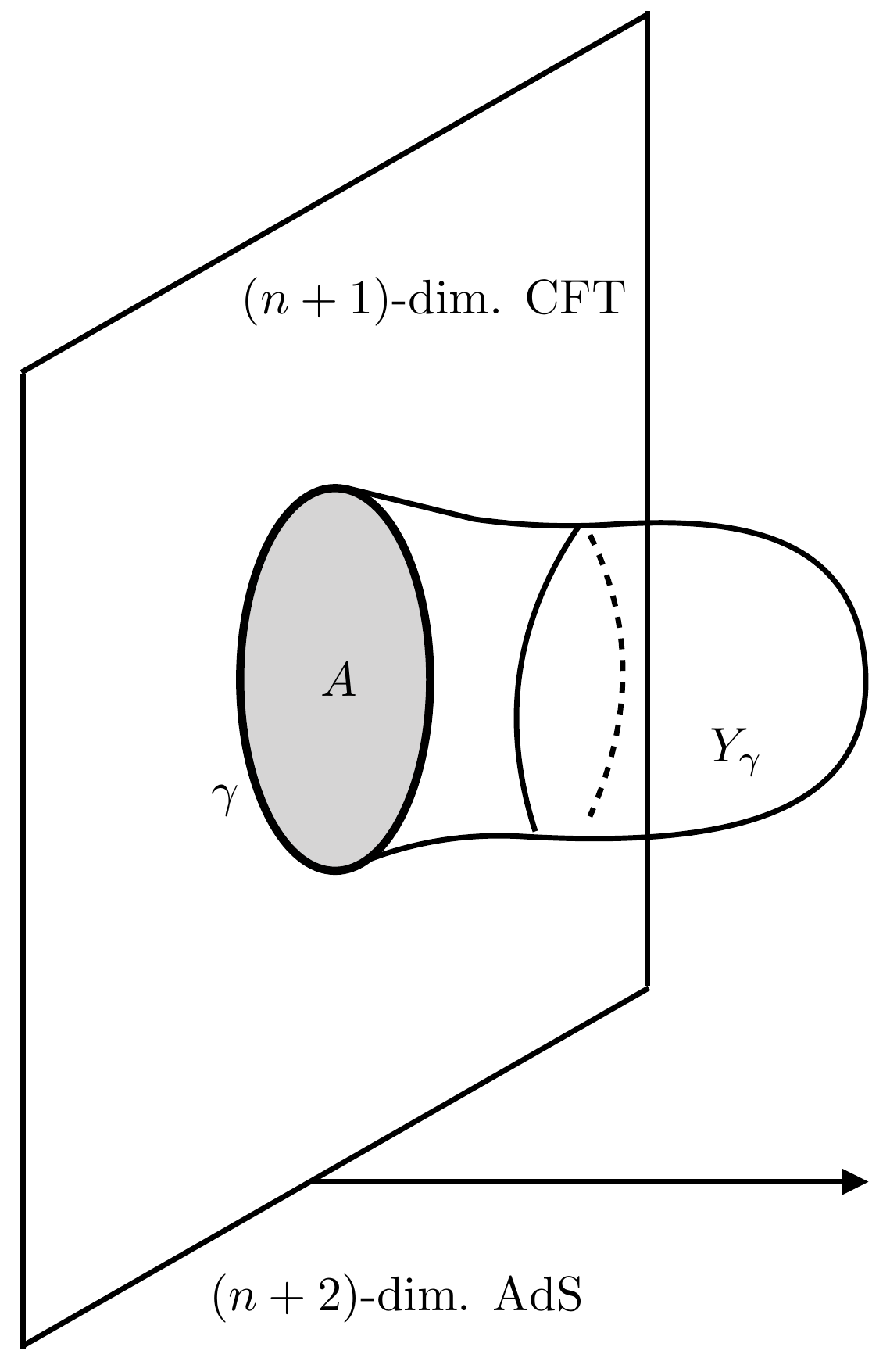}
  \caption{Depiction of a region $A$ in an $(n+1)$-dimensional CFT and a area-minimizing  
  surface $Y_\gamma$ in $(n+2)$-dimensional AdS.}
  \label{ads-cft-fig}
 \end{figure}

Here we do not consider boundaries at infinity, but rather the finite-boundary 
problem. We will also be working with foliations of the boundary $\partial M$
by a family of simple closed loops, and we will need the area-minimizers bounded  
by these loops to yield a foliation of $M$.\footnote{Loosely speaking, a \textbf{foliation} of 
an $n$-manifold $M$ is a continuous, 1-parameter family of $k$-submanifolds, $k<n$, which 
sweep out $M$.}  We use knowledge of the areas 
of the surfaces\footnote{These surfaces will always be homeomorphic to a disc.} of least area 
bounded by such curves to find the metric. We present two such results below, as well as a local 
determination result. All these rely on knowledge of areas for 
one foliation of our manifold by area-minimizers, and the area data 
 for all area-minimizing perturbations of this foliation by (nearby) area-minimizing 
 foliations. It is not clear what the \emph{minimal} knowledge of areas  
required to determine the interior metric is. We suspect that our assumptions on 
what areas are known are  essentially 
optimal. We propose a conjecture (only partially addressed here) which stipulates that 
they are sufficient:

\begin{conjecture}[Boundary rigidity for least area data]
Let $(M,g)$ be a Riemannian 3-ball which admits a foliation by properly embedded, 
area-minimizing surfaces. Suppose that for this foliation and any nearby perturbation, we
 know the areas of the leaves.

Then this information determines $g$ up to boundary-fixing isometries.
\end{conjecture}

We prove this conjecture for particular classes of Riemannian manifolds. To describe in detail 
these classes, we make the following definitions.

\begin{defn}
\label{cl.eucl}
Let $(M,g)$ be a Riemannian 3-manifold. For $k\in \NN$, we say the metric $g$ is 
$\epsilon$-\textbf{\boldmath$C^k$-close} to the Euclidean metric $g_\EE$ on $\RR^3$ if 
there exists a 
  global coordinate chart $(x^\alpha)$, $\alpha=1,2,3$, on 
$M$ for which we have 
$$||g_{\alpha\beta}(x)-(g_\EE)_{\alpha\beta}(x)||_{C^k(M)}<\epsilon$$ for all 
$\alpha,\beta=1,2,3$.
\end{defn}

From this point onwards when we say a metric is $C^k$-close to Euclidean we mean that 
it is $\epsilon$-$C^k$-close, for some sufficiently small $\epsilon>0$. 

\begin{defn}\label{thin-defn}
We say a Riemannian manifold with boundary $(M,g)$ is  \textbf{\boldmath$(K,\epsilon_0,\delta_0)$-thin} if  for some parameters $K,\epsilon_0,\delta_0>0$ the following holds:
\begin{enumerate}
\item there exist global coordinates $(y^\alpha)$, $\alpha=1,2,3$ on $(M,g)$ such that 
the surfaces $Y(t):=\{y^3=\text{constant}=t\}$ are properly embedded and area-minimizing 
discs, 
$\{Y(t)\,:\,t\in(-1,1)\}=M$, and the coordinates $(y^1,y^2,y^3)$ are regular Riemannian 
coordinates in the sense that the Beltrami coefficient $\mu(y^1,y^2,y^3)$ satisfies $|\partial^k_i \mu|\le 10/(\epsilon_0)^i$ for $i=1,2$ (see \eqref{confmap1});
\item $\frac12<||g^{33}||_{L^\infty(M)}<2$, and for $i,j\in\{1,2\}$, and $k,l\in\{0,1,2\}$, with $0<k+l\le2$, $|| \partial_i^k\partial_j^l g^{33}||_{L^\infty(M)}\le K\epsilon_0^{-(k+l)+1}$; 
\item for $i,j,k\in\{1,2\}$, and $\beta,\gamma \in\{0,1,2,3,4\}$ with $\beta+\gamma\le4$, $|| \partial_i^\beta\partial_j^\gamma g^{3k}||_{L^\infty(M)}\le K\epsilon_0^{-(\beta+\gamma)}$; 
\item for each $t$, ${\rm Area}[Y(t)]$ is bounded above by $4\pi\epsilon_0^2$;
\item The Riemann curvature tensor $\text{Rm}_g$ of the metric $g$ 
and the  second fundamental forms $A(t)$ of each $Y(t)$ 
satisfy the bounds:\footnote{$\overline{\nabla}$ is the connection of the metric $g$ on 
$Y_t$ and $\nabla$ is the connection of $g$.}
\[
||A(t)||\le \sqrt{\delta_0}\cdot \epsilon^{-1}_0, ||\overline{\nabla} A(t)||\le 
\sqrt{\delta_0}\cdot \epsilon^{-2}_0, 
||{\rm Rm}_g||\le \delta_0\cdot \epsilon^{-2}_0, ||\nabla {\rm Rm}_g||\le
 \delta_0\cdot \epsilon^{-3}_0
\]
\end{enumerate}
\end{defn}

%

\begin{figure}[h]
    \centering
    \begin{subfigure}[b]{0.4\textwidth}
    \centering
        \includegraphics[width=0.3\textwidth]{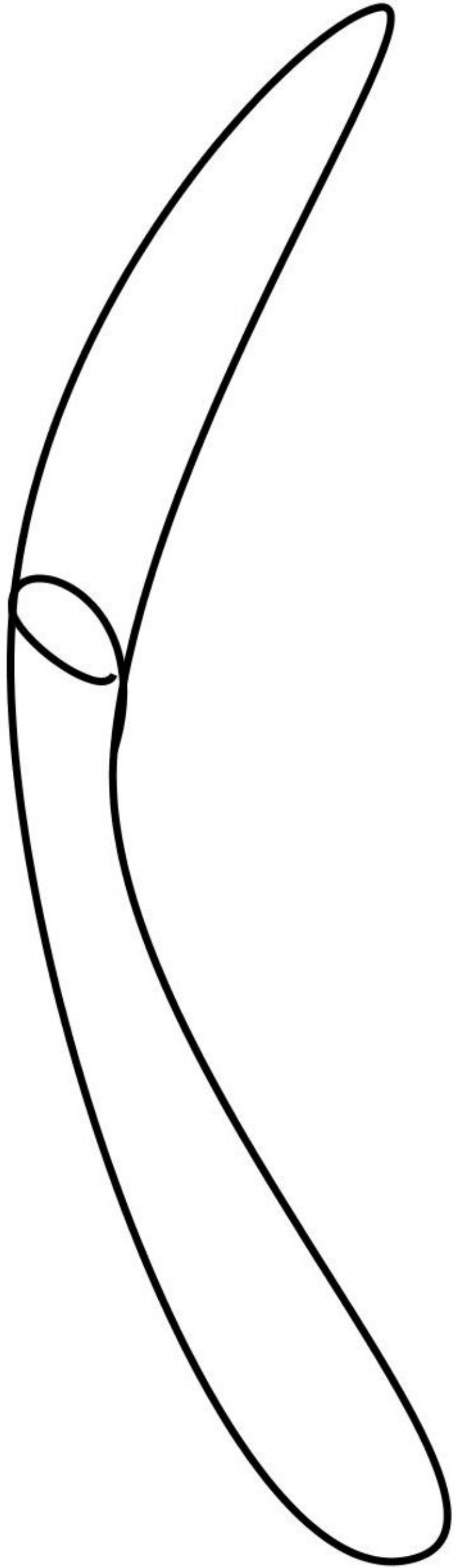}
\caption{The case when $\epsilon_0<<1$. The\\ thin manifold is allowed to be ``bent''.}
    \end{subfigure}
    ~
    \begin{subfigure}[b]{0.3\textwidth}
    \centering
        \includegraphics[width=0.9\textwidth]{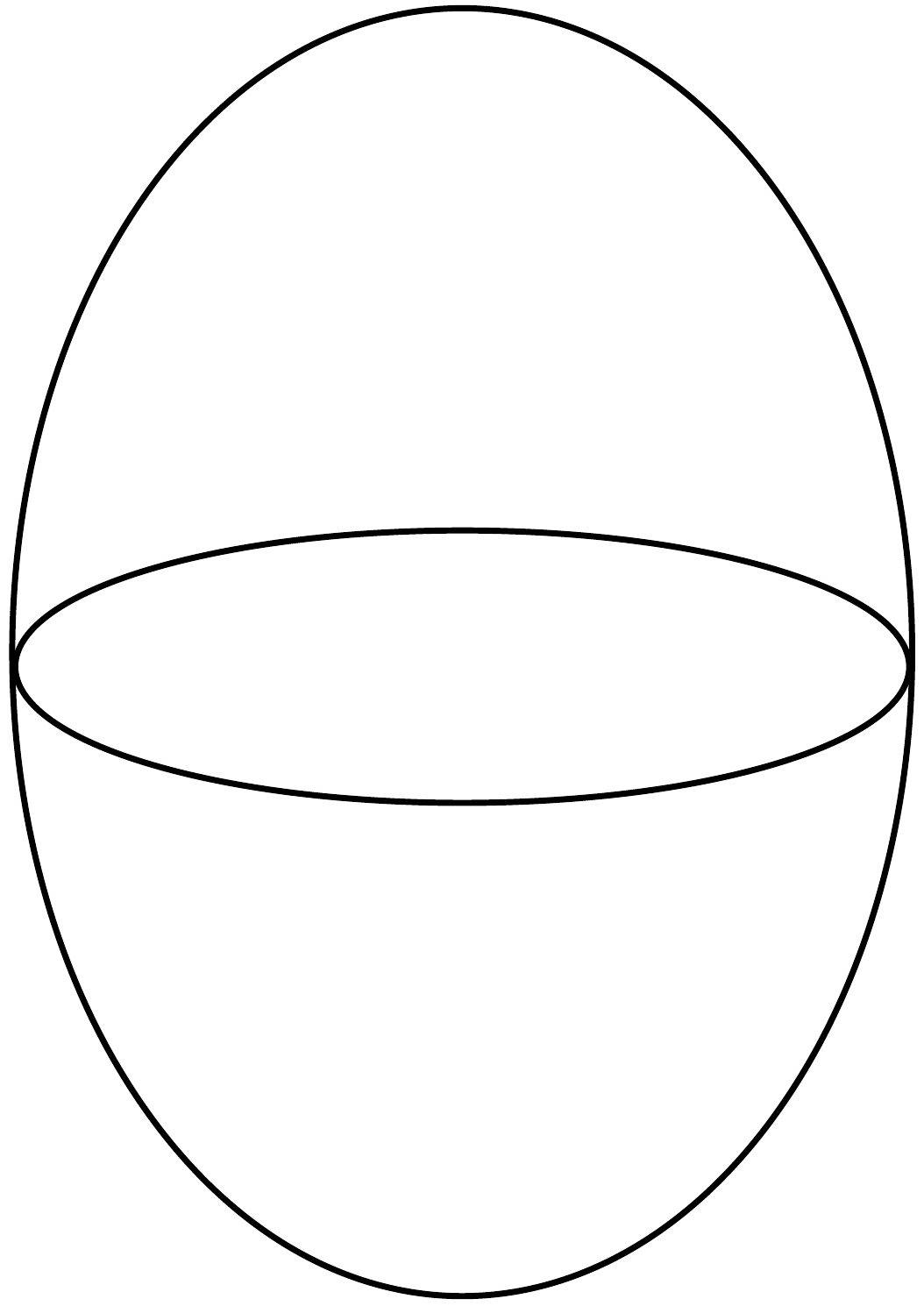}
        \caption{The case when $\epsilon_0>1$. This fatter manifold is required to be ``straight''.}
    \end{subfigure}
    \caption{Depiction of $(K,\epsilon_0,\delta_0)$-thin manifolds.}
\end{figure}

�Let us describe how the parameters $K,\epsilon_0,\delta_0$ correspond to different bounds on the geometry of the minimal foliated $(M,g)$:

The parameter $\epsilon_0$ in requirement 4 corresponds to a weak notion of ``girth'' of the manifold. Note that in conjunction with bounds on the geodesic curvature on the boundary $\gamma(t)=\partial Y(t)$ and the bound in 5, requirement 4 implies bounds on the diameter of each leaf $Y(t)$. Thus a small $\epsilon_0$ implies the manifold has thin girth.

Requirement 5 bounds the curvature of the ambient manifold, as well as the extrinsic geometry of the minimal leaves. These bounds can be large when $\delta_0$ is fixed and $\epsilon_0$ is taken small enough. Requirement 1 (weakly) bounds the intrinsic geometry of the leaves, while the estimates in requirements 2 and 3 bound the ``straightness'' of the metric $g$. In particular, the functions $\partial_i g^{33}$, $\partial_i g^{3k}$ vanish when the vector fields
 $\partial_3$ are hypersurface-orthogonal affine geodesic vector fields, i.e. the metric is expressed in Fermi-type coordinates

�\[
g=(dx^3)^2+\sum_{i,j=1}^2 g_{ij}(x^1, x^2, x^3)dx^idx^j.
\]
Thus, the bounds in 2 and 3 measure the departure from this \emph{straight} picture; smallness of $K$ should be thought of as a nearly \emph{straight} minimal foliated manifold.

With Definitions \ref{cl.eucl} and \ref{thin-defn} in mind, we work with Riemannian manifolds belonging to either of the next two classes: 


\begin{defn}\label{class1and2} Let $(M,g)$ be a $C^4$-smooth, Riemannian manifold which has the properties
\begin{enumerate}
\item $M$ homeomorphic to a 3-dimensional ball in $\RR^3$, 
\item $(M,g)$  has $C^4$-smooth, mean convex boundary $\partial M$,
\item there is a foliation of $\partial M$ by simple closed curves, $\{\gamma(t)\,:\,t\in(-1,1)\}=\partial M$, 
which induces a foliation $\{Y(t)\,:\,t\in(-1,1)\}=M$ by properly embedded, area minimizing 
surfaces\footnote{Note: as $t\to\pm1$, the loops $\gamma(t)$ and the surfaces $Y(t)$ collapse to points on $\partial M$.}, and satisfies a regularity assumption: 
The geodesic curvatures of the curves $\gamma(t)\subset \partial M$
obey\footnote{We also note that the area bound on the leaves $Y(t)=\{x^3=t\}$ together with the geodesic curvature bounds, the bounds on the
ambient curvature, and bounds on $||A||_g$ imply diameter bounds on each $Y(t)$, of the form: ${\rm
diam}(Y(t))\le 4 \sqrt{{\rm Area} [Y(t)]}$.}
\[
||\kappa||\le  
2{\rm Area}[Y(t)]^{-1/2}, ||\overline{\nabla}\kappa||\le  
2{\rm Area}[Y(t)]^{-1}. 
\]
\end{enumerate}
 If additionally the metric $g$ is $\epsilon_0-C^3$-close to Euclidean for some small 
 $\epsilon_0>0$, we say $(M,g)$ is of {\bf Class 1}. If  
$(M,g)$ is $(K,\epsilon_0,\delta_0)$-thin, for some sufficiently small 
$\delta_0>0$, and $K,\epsilon_0>0$ are such that  $K\epsilon_0$ is sufficiently 
small,\footnote{We do not keep track of the constants but careful
 tracking of the proof could yield universal bounds.} we say $(M,g)$ is of {\bf Class 2}.
 \end{defn}
 
In these settings, we prove:
\begin{theorem}\label{global_metric_thm}
Let $(M,g)$ be a manifold of Class 1 or Class 2 above, and $g|_{\partial M}$ be given. Let 
$\{\gamma(t)\,:\,t\in(-1,1)\}=\partial M$ and $\{Y(t)\,:\,t\in(-1,1)\}=M$ be as in Definition \ref{class1and2}. Suppose that for each curve $\gamma(t)$ and any nearby perturbation 
$\gamma(s,t)\subset \partial M$, we know the area of the properly embedded surface $Y(s,t)$
 which solves the least-area problem for $\gamma(s,t)$.  

Then the knowledge of these areas uniquely determines the metric $g$ (up to isometries which 
fix the boundary).
\end{theorem}
We note that for our first result the curvature is required to be small. For the second result,
the curvature can be very large, but this will be compensated by the thinness condition. 
(The requirement on the geodesic curvature is a technical condition imposed to ensure that a 
certain extension of our surfaces to infinity can be performed while preserving the bounds we 
have).

Our $C^3$-close to Euclidean or $(K,\epsilon_0,\delta_0)$-thin assumption\footnote{We 
remark, but do not prove, that under such assumptions we expect that the existence of area 
minimizing foliations for $(M,g)$ can be derived by a perturbation argument.} is technically 
needed for very different reasons than those in \cite{lassas2003semiglobal,stefanov2005boundary, BuragoIvanov}, as we use it to obtain 
unique solvability of the system of  pseudodifferential equations for the metric components 
which we will describe later in this paper.  We impose that $\partial M$ is mean convex to 
ensure the solvability of the Plateau problem for a given simple curve on the the boundary of 
$M$ (see \cite{Meeks-Yau}). However, this is not necessary; it is easy to see that one could 
have foliations on $C^3$-close to Euclidean or $(K,\epsilon_0,\delta_0)$-thin manifolds without 
mean convex boundary. Our hypothesis on foliations of the manifold by area-minimizing surfaces 
has some similarity with that of \cite{SVU2}. However, the proofs are completely different. The 
existence of foliations of $(M,g)$ by properly embedded, area-minimizing surfaces ensures that 
there are no unreachable regions trapped between minimal surfaces bounded by the same 
curve, thus avoiding obstructions to uniqueness analogous to the ones for the boundary 
distance problem. 

\medskip
As a consequence of Theorem \ref{global_metric_thm}, we have the local result:\\

\begin{theorem}\label{local_metric_thm}
Let $(M,g)$ be a $C^4$-smooth, Riemannian manifold with boundary $\partial M$. Assume that $\partial M$ is both $C^4$-smooth and mean convex at $p\in \partial M$. Let $U\subset \partial M$ be a neighbourhood of $p$, and let $\{\gamma(t)\,:\,t\in(-1,1)\}$ be a foliation of $U$ by simple, closed curves which satisfy the estimates in Definition \ref{class1and2}. Suppose that $g|_{U}$ is known, and for each $\gamma(t)$ and any nearby perturbation $\gamma(s,t)\subset U$, we know the area of the properly embedded surface $Y(s,t)$ which solves the least-area problem for $\gamma(s,t)$. Then, there exists a neighbourhood $V\subset M$ of $p$ such that, up to isometries which fix the boundary, $g$ is uniquely determined on $V$.
\end{theorem}



The methods that are developed in this paper can be used to prove further results. 
We highlight one such further result below and provide the sketch of the proof (for brevity's sake) 
at the end of Section 4.

The result below again proves 
the uniqueness of the metric given knowledge of minimal areas. We do not impose 
any thinness or small curvature restrictions. Instead we assume that the manifold admits
a foliation by strictly mean-convex spheres that shrink down to a point, and
 minimal foliations from ``all directions'' and that the areas of all these minimal surfaces and
 their perturbations are known. More precisely, we define the following class of manifolds. 

\begin{defn}\label{foliations-all-directions}
A Riemannian manifold $(M,g)$ \textbf{admits minimal foliations from all directions} if
\begin{itemize}
\item $(M,g)$ has mean convex boundary $\partial M$,
\item there exists a foliation $\{N(r)\,:\, r\in[0,1)\}=M$ by (strictly) mean convex surfaces 
$N(r)\subset M$ with $N(0):=\partial M$, and $lim_{r\to 1} N(r)\to q$, where $q\in M$ 
is a point and the mean curvature of $N(r)$ tends to $+\infty$ as $r\to 1$.

\item for every $r\in [0,1)$ and $p\in N(r)$, there is a foliation 
$\{\gamma(t,p)\,:\,t\in(-1,1)\}=\partial M$ by simple closed curves, which induces a fill-in 
$\{Y(t,p)\,:\,t\in(-1,1)\}=M$ by properly embedded, area minimizing surfaces, with the
 property that $Y(t_0)$ is tangent to $N(r)$ at $p$ for some $t_0\in(-1,1)$.
 \end{itemize}
\end{defn}

\begin{figure}[h]
    \centering
    \begin{subfigure}[b]{0.4\textwidth}
        \includegraphics[width=\textwidth]{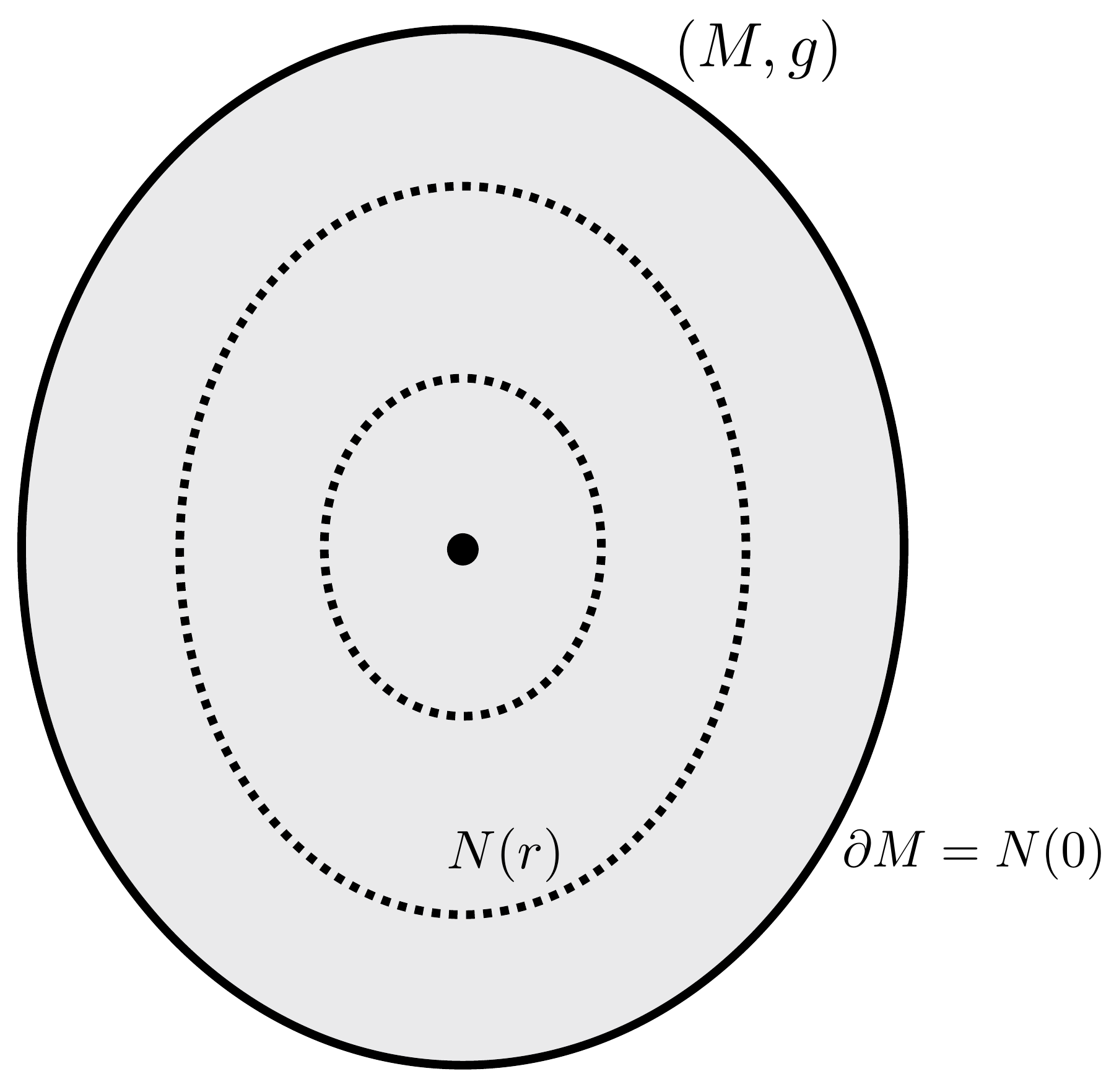}
        \caption{Foliation by strictly mean convex submanifolds $N(r)$.}
    \end{subfigure}
    ~
    \begin{subfigure}[b]{0.4\textwidth}
        \includegraphics[width=\textwidth]{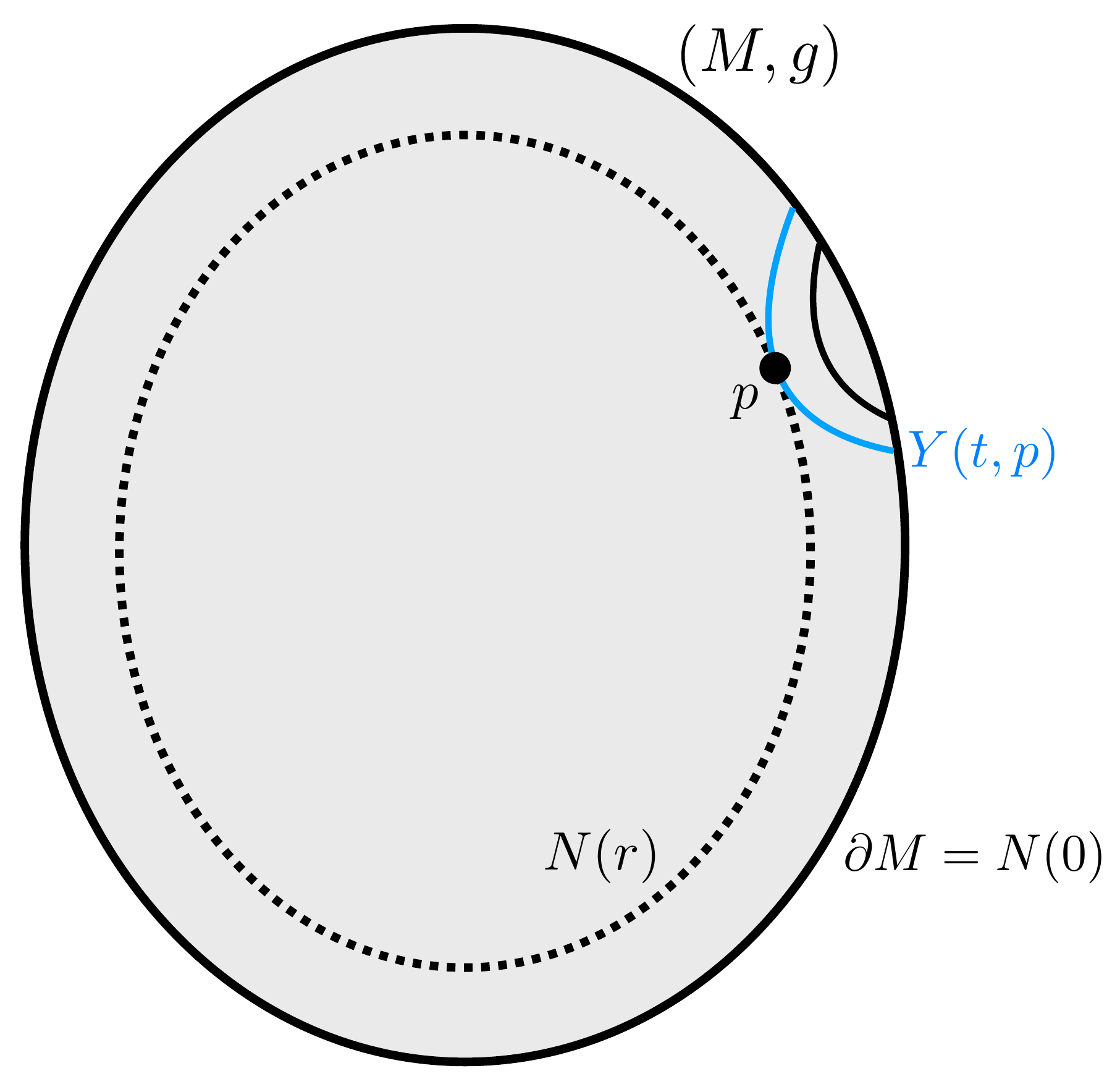}
        \caption{Minimal surface $Y(t,p)$ tangent at $p\in M$.}
    \end{subfigure}
    \caption{Illustration of the property of ``admits minimal foliations from all directions''.}
\end{figure}

\begin{theorem}\label{onion_thm} 
Let $(M,g)$ be a $C^4$-smooth Riemannian manifold which admits minimal foliations from all directions, 
and let $g|_{\partial M}$ be given. Suppose that for all $p\in M$ and for each $\gamma(t,p)$ 
as in Definition \ref{foliations-all-directions}, and any nearby perturbation 
$\gamma(s,t,p)\subset \partial M$, we know the area of the properly embedded surface
 $Y(s,t,p)$ which solves the least-area problem for $\gamma(s,t,p)$.  \\

Then the knowledge of these areas uniquely determines the metric $g$ (up to isometries 
which fix the boundary).
\end{theorem}

\begin{remark}
We note that our theorems (and extensions of these results that one can obtain by adapting 
these methods) use foliations of the unknown manifold $(M,g)$ by a family of area-minimizing 
discs. Not all foliations of $\partial M$ by closed loops $\gamma(t)$ yield fill-ins 
by area minimizing surfaces $Y(t)$ which foliate $M$. For example, if $(M,g)$ admits a 
closed minimal surface in the interior, then the set of area-minimizers $Y(t)$ that fill-in the
loops  $\gamma(t)$ must contains ``gaps''. We note that recently Haslhofer and Ketover \cite{Hasl-Ket} derived that a positive-mean-curvature foliation �as required in Definition \ref{thin-defn} does exist, under the first assumption, and assuming the non-existence of a closed minimal surface inside $(M,g)$. Now, regarding the third requirement of Definition \ref{thin-defn}, it is interesting to note that the area data 
function ${\rm Area}[Y(t)]$, seen as a function of $t$, detects whether the 
area-minimizing fill-ins $\{Y(t)\,:\, t\in(-1,1)\}$ yield a foliation or display gaps. 
\end{remark}

\begin{prop} \label{C1.sees}
Let $(M,g)$ be a Riemannian manifold with mean convex boundary $\partial M$. Let $\{\gamma (t)\,:\, t \in (-1,1)\}=\partial M\}$ be a foliation of $\partial M$. Let $A(t)$  denote the area of area-minimizing surfaces $Y(t)$ that bound $\gamma(t)$\footnote{This function is well-defined
even when there are multiple area-minimizing surfaces bounding $\gamma(t)$.}.

Then there is a foliation of M by fill-ins of area minimizing surfaces that bound $\gamma(t)$ if and only if $A(t)$ is a  $C^1$-smooth function of $t$.
\end{prop}

We provide a proof of Proposition \ref{C1.sees} in the Appendix.

\subsection{Outline of the main ideas.}

We briefly describe below our approach to the proof of Theorem \ref{global_metric_thm}. The main strategy to reconstructing the metric is to use the background foliation 
of $(M,g)$ by area-minimizing, topological discs, 

\[
M=\{Y(t)\,:\,t\in(-1,1)\},
\] to progressively solve for the metric by moving
``upwards'' along the foliation. 

The metric is solved for in a normalized gauge: the metric $g$ is expressed in a new
coordinate system $\{x^1, x^2, x^3\}$ such that $x^3$ is constant on each leaf $Y(t)$ 
of the foliation, and $x^1, x^2$ restricted to each leaf $Y(t)$ are isothermal coordinates for 
$g|_{Y(t)}$. The non-uniqueness of isothermal coordinates is fixed by
moving to an auxiliary extension $(\tilde{M}, \tilde{g})$ of our manifold and 
 imposing a normalization at infinity.

 In such a coordinate system, there are four non-zero independent entries of the metric: 
\[ g=\begin{pmatrix} e^{2\phi} & 0 & g_{31}\\
                                  0 & e^{2\phi} & g_{32}\\
                                  g_{13} & g_{23} & g_{33}
                                  \end{pmatrix}. \]
 The proof then proceeds in reconstructing the components of $g^{-1}$.

Our main strategy is to 
not use the area data directly, but rather the second variation of areas, which is also 
known (since it corresponds to the second variation of a known functional). 
We show in Section 2 (Proposition \ref{area_to_DN}) that the second variation of area yields knowledge of the
 Dirichlet-to-Neumann map for the stability operator 
 
\begin{align}
\JJ:=\Delta_{g_{Y_\gamma}}+\left(\text{Ric}_g(\vec{n},\vec{n}) + ||A||_{g}^2\right)\label{stability-op}
\end{align}
 associated to each of the minimal surfaces  $Y(t)$. In fact 
 we can learn the Dirichlet-to-Neumann map for the associated Schr\"{o}dinger operator 
 
  \begin{align}
\Delta_{g_{\EE}}+e^{2\phi}\left(\text{Ric}_g(\vec{n},\vec{n}) + ||A||_{g}^2\right),\label{stability-op-iso}
\end{align}
in isothermal coordinates. The existence of the foliation $\{Y(t)\,:\,t\in(-1,1)\}$ implies that the equation 
 \begin{align}
\Delta_{g_{\EE}}\psi+e^{2\phi}\left(\text{Ric}_g(\vec{n},\vec{n}) + ||A||_{g}^2\right)\psi=0,\label{stability-op-iso-EQ}
\end{align}
 has positive solutions. Thus, the operator \eqref{stability-op-iso} has strictly negative 
 eigenvalues (see for instance \cite{BarrySimon}), and its potential is of the form covered in 
 \cite{nachman_aniso} (see also \cite{IN}). 
  The main result in \cite{nachman_aniso} implies that such a potential can be determined from the Dirichlet-to-Neumann map. We therefore also have knowledge of \emph{all solutions} of (\ref{stability-op-iso})  for given boundary data.
That is, we first prove
\begin{prop} \label{area_to_DN}
Let $(M,g)$ be a $C^4$-smooth, Riemannian manifold which is homeomorphic to a 3-dimensional ball in $\RR^3$, and has mean convex boundary $\partial M$. Let $g$ be given on ${\partial M}$. Let $\gamma$ be a given simple closed curve on ${\partial M}$, and set $Y_\gamma\subset M$ to be a surface of least area bounded by $\gamma$. Suppose that the stability operator on $Y_\gamma$ is non-degenerate, and that for $\gamma$ and any nearby perturbation $\gamma(s)$, the area of the least-area surface $Y_{\gamma(s)}$ enclosed by $\gamma(s)$ is known. 

Equip a neighbourhood of $Y_\gamma$ with coordinates $(x^\alpha)$ such that on $Y_\gamma$, $x^3=0$ and $(x^1,x^2)$ are isothermal coordinates. Then, 
\begin{enumerate}
\item the first and second variations of the area of $Y_\gamma$ determine the Dirichlet-to-Neumann map associated to the boundary value problem 
\begin{align*}
\Delta_{g_\EE}\psi+e^{2\phi}\left(\text{Ric}_g(\vec{n},\vec{n})+||A||^2_g\right)\psi&=0,\\
\psi|_{\partial Y_\gamma}&=g\left(\left.\frac{d}{ds}\gamma(s)\right|_{s=0},\vec{n}\right),
\end{align*} 
on $Y_\gamma$, where $e^{2\phi}g_\EE=e^{2\phi}[(dx^1)^2+(dx^2)^2]$ is the metric on $Y_\gamma$ in the coordinates $(x^1,x^2)$.
\item Knowledge of the first and second variations of the area of $Y_\gamma$ determines any solution $\psi(x)$ to the above boundary value problem, in the above isothermal coordinates.
\end{enumerate}
\end{prop}

This then yields information on the (first variation of) the position of nearby
 minimal surfaces  $Y(s,t)$ relative to any of our $Y(t)$ without having found 
 \emph{any} information yet on the metric. This first variation of position is expressed in the 
 above isothermal coordinates, thus at this point we make full use of the invariance of 
 solutions of the two equations (\ref{stability-op}) and (\ref{stability-op-iso}) above under 
 conformal changes of the underlying 
 2-dimensional metric. 
 
We note that in the chosen isothermal coordinates, the component $\sqrt{g^{33}}$ is the 
lapse function associated to the foliation:
\begin{defn}
Let $\Omega\subset \RR^2$ be a domain with boundary, and 
$f(\cdot,t):\Omega\times[0,1]\hookrightarrow M$ a foliation of $M$ by the surfaces
 $Y(t):=f(\Omega,t)$.
Set $\vec{n}_t$ to be a unit normal to $Y(t)$. Then, the normal component of the variational 
vector 
$$g\left(f_*\left(\frac{\partial }{\partial t}\right),\vec{n}_t\right)=:\psi,$$ is called the 
\emph{lapse function} of the foliation $f$.
\end{defn}
As the lapse function is a Jacobi field on any such area-minimizing surface, by Proposition 
\ref{area_to_DN}, knowledge of the solutions to equation (\ref{stability-op-iso}) directly yields 
the component $g^{33}$ of the 
 inverse metric, everywhere on $M$. To find further components, we will apply the above 
 strategy not only to
  our background foliation, but also to a suitable family of first variations of our background
   foliation. For each point $p\in Y(t)\subset M$ we can identify two 1-parameter families of 
   foliations $Y(s,t)$ for which the first variation  of  $p$ vanishes at $p$, but the first 
   variation 
   of tangent space is either in the direction $\partial_1$ or $\partial_2$. This is done in 
   Subsection \ref{system_section}
below. These variations lead to a nonlocal
system of equations involving $g^{31}$, $g^{32}$, and $\phi$ (in fact for 
the differences $\delta g^{31}$, $\delta^{32}$, $\delta\phi$ of these quantities for two
putative metrics with the same area data). We then solve for $\delta g^{31}$, $\delta g^{32}$
in terms of the last quantity $\delta\phi$. This involves the inversion of a system of pseudodifferential
equations. The assumption of close to Euclidean or thin/straight is used at this point in a most essential way.
   
   Invoking the knowledge of $g^{33}_s$ obtained above for each of the new foliations 
   $Y(s,t)$ at $p$ and linearizing in $s$ yields new equations on $g^{31}, g^{32}$ at 
   the chosen point $p$. These equations involve the (still unknown) conformal factor $\phi$, 
   but also the first variations of the isothermal coordinates $(x^1, x^2)$. 
   
   This results in a non-local system of the equations. Thus so far for each $p\in M$ we have 
   derived three equations on the four unknown components of the metric $g$. The required
    fourth equation utilizes the fact that each $Y(t)$ is minimal.\footnote{So only now we 
   use the minimal surface equation directly, not its first variation.} This results in an evolution 
   equation (in $t=x^3$) on the components $\phi, g^{13}, g^{23}$. We prove the uniqueness 
   for the resulting system of equations in Section 4. We note that it is at this point that the 
   existence of a global  foliation (without ``gaps'') is used in the most essential way. \\

\textbf{Outline of the paper:} In Section 2 we explicitly describe how we asymptotically extend $(M,g_1)$ and $(M,g_2)$ and construct the coordinates systems on $(M,g_1)$ and $(M,g_2)$ we work with. We also prove Proposition \ref{area_to_DN}.  In Section 3, we collect arguments for determining the lapse function in our chosen coordinates and the first variation of the lapse, as well as the pseudodifferential equation governing the evolution from leaf to leaf of the metric components of $g_1$ and $g_2$ which are purely tangent to the leaves. In Section 4, we give the proofs of all theorems.  \\

\textbf{Notation:} We use the Einstein summation notation throughout this paper. Thus an instance of an index in an up and down position indicates a sum over the index; e.g. $T^iT_j:= \Sigma_{i=j}T^iT_j$. Denote by $\nabla$ the Levi-Civita connection associated to $g$, and $\text{Rm}_g$ the Riemann curvature tensor of $(M,g)$. We'll use the following sign convention for the curvature tensor: $$\text{Rm}_g(U,V)W:= \nabla_V\nabla_UW-\nabla_U\nabla_VW+\nabla_{[U,V]}W.$$ The Ricci curvature of $M$ is $\text{Ric}_g(U,V):= \text{tr}_g\text{Rm}_g(U,\cdot,V,\cdot)$. If $u:M\to \RR$, we write $\nabla u:=\text{grad}(u)$.

\subsection*{Acknowledgements}

S.A. was supported by and NSERC discovery grant and an ERA award. 
T.B. was supported by an OGS fellowship. She also is 
grateful to A. Greenleaf 
for many helpful comments on her PhD thesis, where part of these results were 
developed. A.N. was supported by 
an NSERC discovery grant.


\section{Asymptotically Flat Extension and Conformal Maps}


In this section, we define asymptotically flat extensions of the 3-manifolds we work with. For technical reasons, we also define an extended foliation $\bar{Y}(t)$ of $M$ and a coordinate system $(x^\alpha)$ adapted to the foliation $\bar{Y}(t)$ such that $x^3$ is constant on the leaves $Y(t)$ and $(x^1,x^2)$ are isothermal for the metric restricted to the leaves $x^3=\text{constant}$. We then prove that knowledge of the area of any area-minimizing surface near $\bar{Y}(t)$ determines the Dirichlet-to-Neumann map for the stability operator on $\bar{Y}(t)$ in our preferred coordinates $(x^\alpha)$, and from this information, we also determine the image of $\bar{Y}(t)$ under our chosen isothermal coordinate map.

From this point onwards all estimates we write out will involve a constant $C>0$. Unless stated otherwise, the constant $C$ will depend on the
parameters $K,\delta_0$ in the setting of Class 2, or will be a uniformly fixed parameter in the setting of Class 1.

\subsection{Extension to an Asymptotically Flat Manifold}\label{AFcoords}

In this section, we let $(M,g)$ be a $C^4$-smooth, 3-manifold that satisfies the assumptions 
of Theorem \ref{global_metric_thm}. In particular,  we assume that $(M,g)$ admits foliations by 
properly embedded, area minimizing surfaces. 
In this setting, we use such a foliation to equip $(M,g)$ with a preferred coordinate system. The 
preferred coordinate system we construct will be used in several of the proofs of this paper to 
simplify computations and derive relevant equations for the components of the metric.

To construct the desired coordinates, let $\Omega \subset \RR^2$ be a bounded 
$C^{2,\alpha}$ domain for some $\alpha>0$, and let $\gamma:\partial \Omega\times(-1,1)\to \partial M$ be a given foliation of the boundary by embedded, closed curves. Let $f(\cdot,t):\Omega\to M$ solve Plateau's problem for $\gamma(\cdot,t)$, for each $t\in (-1,1)$. In particular, $f(\cdot,\cdot)$ defines a foliation of $M$ by properly embedded, codimension 1, area-minimizing surfaces such that $f(\partial \Omega\times\{t\})=\gamma(\cdot, t)$ for each $t\in(-1,1)$. We denote the leaves of the foliation $f$ in $M$ by $Y(t):=f( \Omega\times\{t\})$.

Our first choice of coordinate is the parameter identifying each minimal surface $Y(t)$: label the 
coordinate $x^3=t$. Now, to obtain two other coordinate functions, we will choose conformal 
coordinates $(x^1,x^2)$ on each leaf $Y(t)$ of the foliation. Then $(x^1,x^2, x^3)$ will be a 
global coordinate system on $(M,g)$. However, there are many choices for conformal 
coordinates $(x^1,x^2)$ on a 2-dimensional surface $Y(t)$. To remove this ambiguity, we 
extend $(M,g)$ to an asymptotically flat manifold and impose decay conditions at infinity on the 
conformal coordinates on the extension of $Y(t)$ which renders these coordinates unique. 

To this end, let $(z^1,z^2,t):M\to \Omega\times(-1,1)$ be the regular
 coordinates on $M$ stipulated by our Theorems for manifold Classes 1 and 2, respectively. 
 Considering a fixed extension 
 operator for metrics (and using our assumed bounds on curvature and second fundamental
 forms), we may smoothly extend the metric $g|_{\partial M}$ to a 
tubular neighbourhood $N$ of $\partial M$, preserving (up to a multiplicative factor) the 
bounds on curvature and on the second fundamental form of $Y(t)$ assumed in our Theorems.

Let $g_{\EE}$ denote the Euclidean metric on $\MM:= \RR^2\times (-1,1)$, and let 
$\chi:\MM\to\RR$ be a smooth cutoff function such that $\chi|_{M}=1$ and $\chi=0$ outside
 $M\cup N$. We then extend $(M,g)$ to an asymptotically flat manifold $(\MM,\afg)$ with 
 metric $\afg$ defined as $$\afg := \chi g+(1-\chi)g_{\EE}.$$
Again via an extension operator, we obtain a smooth extension 
$\mathbf{f}:\RR^2\times(-1,1)\to \MM$ of the foliation $f$. The smooth (but not 
necessarily minimal with respect to $\afg$) extension of $Y(t)$ to $\MM$ is then 
$\mathbf{f}(\RR^2,t)=:\YY(t)\cong \RR^2$.

A well known result of Ahlfors \cite{ahlfors} then gives the unique existence of \textbf{isothermal coordinates} on $\YY(t)$ which are normalized at infinity. That is, for some $p>2$, there exists a conformal map 
\begin{align*}
\Phi(t): \RR^2\to\RR^2,\quad z:=z^1+iz^2\mapsto x^1+ix^2
\end{align*}
in $L^p(\RR^2)$ satisfying
\begin{align}
\bar\partial \Phi(t) &= \mu(t)\partial\Phi(t)\label{confmap1}\\
\Phi(z,t)-z &= L^p(\RR^2),\label{confmap2}
\end{align}
with \textbf{dilation} $\mu(t) := \frac{g_{11}-g_{22}+2ig_{12}}{g_{11}+g_{22}+2\sqrt{\afg|_{\YY(t)}}}$. In these coordinates, $\Phi(t)$ pushes forward $\afg_t:=\afg|_{\YY(t)}$ to a metric conformal to the Euclidean metric on $\RR^2$: $$\afg_t=e^{2\phi(x,t)}[(dx^1)^2+(dx^2)^2]$$ for some conformal factor $\phi(\cdot, t):\RR^2\to\RR$. We call the images $(x^1,x^2)$ \textbf{isothermal coordinates} on $\Omega$, and denote the conformal image of $\Omega$ as $\tilde\Omega(t):= \Phi(\Omega,t)$.  



Since $\afg$ is $C^4$-smooth by construction, $\mu_t$ is $C^4$-smooth in $t$. As 
shown in \cite{AhlforsBers}, $\Phi(z,t)$ is a $C^4$-smooth map in $t$. Therefore, 
\begin{align*}
\Phi(z(\cdot),t(\cdot))&:\MM\to\tilde\Omega(t)\times(-1,1)\\
\Phi(z(p),t(p))&=(x^1,x^2,x^3)
\end{align*}
defines a global coordinate chart on $(\MM,\afg)$ In this chart, the metric takes the form
$$ \afg = \afg_{3\alpha}dx^3dx^\alpha+\afg_t,$$ where $\afg_t$ is conformally flat for each $t$, and additionally $\afg_{3k}=0$ for $k=1,2$ outside a compact set containing $M$. In particular, the metric $\afg$ restricted to $M$ is written as
$$ g = g_{3\alpha}dx^3dx^\alpha+e^{2\phi(t)}[(dx^1)^2+(dx^2)^2].$$
 

Next we prove that for such a coordinate system $\Phi(z,t)=(x^1,x^2,x^3)$ on $(\MM, \afg)$ as described above, knowledge of the areas of properly embedded area minimizing surfaces in $(M,g)$ determines the conformal map $\Phi(\cdot, t)$ on the complement $\RR^2\setminus \Omega$, for every $t\in(-1,1)$. To prove such a statement, we first show that the knowledge of a Dirichlet-to-Neumann map for a non-degenerate Schr\"{o}dinger operator on $\Omega$ determines the conformal map $\Phi$ satisfying (\ref{confmap1}) and (\ref{confmap2}) on the open set $Z:= \RR^2\setminus\Omega$ (see \cite{AstLasPav}, \cite{SU}, for similar results). Then, we prove that the knowledge of the area of any properly embedded area minimizing surface in $(M,g)$ determines the Dirichlet-to-Neumann map associated to the stability operator on $Y(t)$ and its perturbations.

In the proofs below, we will construct solutions to the Dirichlet problem for the Schr\"{o}dinger operator. Precise asymptotic statements will be given in terms of a weighted $L^2$ space. The particular weighted space we require has norm 
\begin{align*}
||f||_{L^2_{-\delta}(Z)}&= \left(\int_Z |f(w)|^2(1+|w|^2)^{-\delta}\,dw\right)^\frac12.
\end{align*}

The following propositions also recall the construction of isothermal coordinates on a domain in $\RR^2$ as well as existence and uniqueness for an exterior problem which will be crucial to the proofs of the theorems in this paper.

\begin{prop}\label{uniquness-up-down}
Let $\Omega\subset \RR^2$ be a bounded, $C^2$-smooth domain. Let $g$ be a $C^2$-smooth Riemannian metric on $\Omega$, and $\afg$ to be a $C^2$-smooth extension of $g$ to $\RR^2$ with  
\begin{align*}
\afg&=g_{\EE} \text{ outside a large compact set containing $\Omega$},\\
\afg&=g \text{ on $\Omega$},
\end{align*}
where $g_\EE$ is the Euclidean metric on $\RR^2$. Write $Z:=\RR^2\setminus \Omega$ and $\nu$ for the outward pointing unit normal vector field to $\partial\Omega$.
Let $(z^1,z^2)$ be coordinates on $\RR^2$.

\begin{enumerate}
\item For some $p>2$, there exists a unique conformal map $\Phi: \RR^2\to\RR^2$ satisfying 
\begin{align}
&\Phi(z)-z\in L^p(\RR^2)\\
 &\Phi_{*}(\afg)=e^{2\phi(x)}[(dx^1)^2+(dx^2)^2].
\end{align}

\item Let $V\in L^\infty(\Omega)$ and suppose $0$ is not a Dirichlet eigenvalue of $\Delta_{g}+V$. Let $\Lambda:H^{\frac12}(\partial \Omega)\to H^{-\frac12}(\partial \Omega)$ be the Dirichlet-to-Neumann map associated to $\Delta_{g}\psi+V$.

Then there exists and $R_1>0$ such that for any $\xi\in \CC$ with $|\xi|>R_1$ and $1-\frac1p<\delta<1$, there exists a unique solution $\psi(\cdot,\xi)$ to the exterior problem 
\begin{align}
e^{-iz\xi}\psi(\cdot,\xi)-1&\in L^2_{-\delta}(Z), \label{outside-bound1}\\ 
\Delta_{g}\psi(\cdot,\xi)&=0 \text{ on $Z$,}\label{outside-pde1}\\ 
g(\nabla\psi(\cdot,\xi),\nu)&=\Lambda(\psi)(\cdot,\xi)  \text{ on $\partial\Omega$},\label{outside-condt1}.
\end{align}
Moreover,
\begin{align}
||e^{-i\Phi(z)\xi}\psi(\cdot,\xi)-1||_{L^2_{-\delta}(Z)}&< \frac{C}{|\xi|},\label{outside-est1}
\end{align}
for some constant $C>0$ and all $\xi\in \CC$ with $|\xi|>R_1$.
 \end{enumerate}

 \end{prop}
 
 \begin{proof}
1. This statement is proved in \cite{ahlfors}.\\

2. We will prove existence and uniqueness to (\ref{outside-bound1}), (\ref{outside-pde1}), (\ref{outside-condt1}) by transforming the problem into a Euclidean one via the map $\Phi$. 

First, note that for $z=z^1+iz^2\in \partial \Omega$, in the coordinates $\Phi(z)=x^1+ix^2=:x \in \partial\Phi(\Omega)$ the Dirichlet-to-Neumann map $\Lambda$ is given by the bilinear form
 
\begin{align*}
(\chi,\Lambda(\psi))&=\int_{\partial \Omega} \chi(z) g(\nabla\psi(z),\nu(z))\,dS\\
&= \int_{\partial\Phi(\Omega)}\tilde\chi(x)g_\EE(\tilde\nabla\psi\circ\Phi^{-1}(z),e^{-\phi(x)}\tilde\nu(x))e^{\phi(x)}\, d\tilde{S}\\
&= \int_{\partial\Phi(\Omega)}\tilde\chi(x)g_\EE(\tilde\nabla\tilde\psi(x),\tilde\nu(x))\, d\tilde{S}\\
&=:(\tilde\chi,\tilde\Lambda(\tilde\psi)).
\end{align*}
Here $\chi,\psi\in H^{\frac12}(\partial \Omega)$, $\tilde\chi(x):=\chi\circ\Phi^{-1}(x^1+ix^2)$, $\tilde\psi(x):=\psi\circ\Phi^{-1}(x^1+ix^2)$, and $\tilde\nu$ is the outward pointing unit normal vector field to $\partial \Phi(\Omega)$, with respect to the metric $g_\EE$.
 
The boundary value problem (\ref{outside-pde1}), (\ref{outside-condt1}) expressed in the conformal coordinates given by $\Phi$ becomes
\begin{align}
\Delta_{g_\EE}\tilde\psi(\cdot,\xi) &=0 \text{ on $\Phi(Z)$},\label{conf-outside-pde1} \\
g_\EE(\tilde\nabla\tilde\psi(\cdot,\xi),\tilde\nu)&=\tilde\Lambda(\tilde\psi(\cdot,\xi))  \text{ on $\Phi(\partial\Omega)$}.\label{conf-outside-condt1} 
\end{align}
We claim condition (\ref{outside-bound1}) is equivalent to
\begin{align}
e^{-ix\xi}\tilde\psi(x,\xi)-1\in L^2_{-\delta}(\Phi(Z)). \label{conf-outside-bound1}
\end{align}
To show this assertion, we need the following lemma:

\begin{lemma}\label{norm-estimates1} For $f\in {L^2_{-\delta}(\Phi(Z))}$,
$||f||_{L^2_{-\delta}(\Phi(Z))} \le C ||f\circ\Phi||_{L^2_{-\delta}(Z)}$, for some constant $C>0$.
\end{lemma}

\begin{proof}
By a simple change of variable,
\begin{align*}
||f||_{L^2_{-\delta}(\Phi(Z))}^2&= \int_{\Phi(Z)}\left|f(x)\right|^2(1+|x|^2)^{-\delta}\,dx\\
&=\int_{Z}\left|f(\Phi(z))\right|^2(1+|\Phi(z)|^2)^{-\delta}|\Phi'(z)|^2\,dz.
\end{align*}


Choose $R>0$ large so that $\Omega\subset B_R(0)$. Consider $z\in Z$ satisfying $|z|>2R$. We have $\Phi(z)-z\in L^p(\RR)$ and $|\Phi(z)-z|^p$ is subharmonic. Then, the Mean Value Property for $|\Phi(z)-z|^p$ over the ball $B_z(|z|-R)$ gives 
\begin{align}
|\Phi(z)-z|^p&\le \left(\frac{1}{\text{Vol}(B_z(|z|-R))}\right)\int_{B_z(|z|-R)}|\Phi(w)-w|^p\,dw\notag\\
&\le C\frac{||\Phi(w)-w||_{L^p(Z)}^p}{(|z|-R)^2}\notag\\
&\le C\frac{||\Phi(w)-w||_{L^p(Z)}^p}{|z|^2}.\label{phi-minus-z-bound}
\end{align}

On the other hand, for $z\in Z$ and $|z|>3R$,
\begin{align}
|\Phi'(z) -1|&\le \frac{1}{2\pi}\int_{|w-z|=R}\frac{|\Phi(w)-w|}{|(w-z)^2|}\,|dw|\\
&\le C\int_{|w-z|=R}\frac{||\Phi(v)-v||_{L^p(Z)}}{|w|^\frac2p}\frac{1}{|w-z|^2}\,|dw|\\
&\le C\frac{1}{|R|^\frac2p}\\
&\le \frac{C}{|z|^\frac2p}.\label{dphi-est-1}
\end{align}

Therefore 
\begin{align*}
||f||_{L^2_{-\delta}(\Phi(Z\cap\{|z|>3R\}))}^2 &= \int_{Z\cap\{|z|>3R\}}\left|f(\Phi(z))\right|^2(1+|\Phi(z)|^2)^{-\delta}|\Phi'(z)|^2\,dz\\ 
&\le C\int_{Z\cap\{|z|>3R\}}\left|f(\Phi(z))\right|^2(1+|z|^2)^{-\delta}\,dz.
\end{align*}
Correspondingly, for $\{z\in Z\,:\, |z|<3R\}$,
\begin{align*}
\int_{\Phi(Z\cap\{|z|<3R\})}\left|f(x)\right|^2(1+|x|^2)^{-\delta}\,dx &\le C \int_{Z\cap\{|z|<3R\}}\left|f(\Phi(z))\right|^2(1+|z|^2)^{-\delta}\,dz
\end{align*}
since $\Phi$ is a $C^1$-smooth diffeomorphism on $B_R(0)$.

So indeed, $$||f||_{L^2_{-\delta}(\Phi(Z))} \le C ||f\circ\Phi||_{L^2_{-\delta}(Z)}.$$

\end{proof}

Using the previous lemma,
\begin{align*}
||e^{-ix\xi}\tilde\psi(x,\xi)-1||_{L^2_{-\delta}(\Phi(Z))} &\le C||e^{-i\Phi(z)\xi}\psi(z,\xi)-1 ||_{L^2_{-\delta}(Z)}\\
&\le C||e^{-i[\Phi(z)-z]\xi}-1||_{L^\infty(Z)}||e^{-iz\xi}\psi(z,\xi)-1||_{L^2_{-\delta}(Z)}\\
&\qquad+C||e^{-i[\Phi(z)-z]\xi}-1||_{L^2_{-\delta}(Z)}+C||e^{-iz\xi}\psi(z,\xi) -1 ||_{L^2_{-\delta}(Z)}.
\end{align*}

We assume $||e^{-iz\xi}\psi(z,\xi) -1 ||_{L^2_{-\delta}(Z)}<\infty$; to prove $||e^{-ix\xi}\tilde\psi(\cdot,\xi)-1||_{L^2_{-\delta}(\Phi(Z))}<\infty$, it remains to show $||e^{-i[\Phi(z)-z]\xi}-1||_{L^2_{-\delta}(Z)}<\infty$ and $||e^{-i[\Phi(z)-z]\xi}-1||_{L^\infty(Z)}<\infty$.

Since $\Phi(z)-z\to 0$ as $|z|\to\infty$, $\Phi(z)\in L^\infty(Z)$. Thus $||e^{-i[\Phi(z)-z]\xi}-1||_{L^\infty(Z)}$ is bounded. Now we show a bound for $||e^{-i[\Phi(z)-z]\xi}-1||_{L^2_{-\delta}(Z)}$. To do this, we use the estimate (\ref{phi-minus-z-bound}) from Lemma \ref{norm-estimates1} to obtain for $|z|>2R$
\begin{align*}
\left|e^{-i[\Phi(z)-z]\xi}-1\right|
&\le C(|\xi|) |\Phi(z)-z|\\
&\le C(|\xi|)\frac{||\Phi(w)-w||_{L^p(Z)}}{|z|^\frac2p}.
\end{align*}

Therefore,
\begin{align*}
||e^{-i[\Phi(z)-z]\xi}-1||_{L^2_{-\delta}(Z)}&= \int_{Z}\left|e^{-i[\Phi(z)-z]\xi}-1\right|^2(1+|z|^2)^{-\delta}\,dz\\
&= \int_{Z\cap\{|z|>3R\}}\left|e^{-i[\Phi(z)-z]\xi}-1\right|^2(1+|z|^2)^{-\delta}\,dz\\
&\qquad+ \int_{Z\cap\{|z|<3R\}}\left|e^{-i[\Phi(z)-z]\xi}-1\right|^2(1+|z|^2)^{-\delta}\,dz\\
&\le C(|\xi|)\int_{Z\cap\{|z|>3R\}}\frac{||\Phi(w)-w||^2_{L^p(Z)}}{|z|^\frac4p}(1+|z|^2)^{-\delta}|z|\,d|z| +C(|\xi|),
\end{align*}
which is finite since we take $1>\delta>1-\frac1p$.
Hence 
\begin{align*}
e^{-ix\xi}\tilde\psi(\cdot,\xi)-1\in L^2_{-\delta}(\Phi(Z)). 
\end{align*}
By a similar argument as above, if (\ref{conf-outside-bound1}) holds, then (\ref{outside-bound1}) holds.

Now we construct solutions to (\ref{conf-outside-bound1}), (\ref{conf-outside-pde1}), and (\ref{conf-outside-condt1}). Consider $\tilde\psi(\cdot,\xi):\RR^2\to\RR$ defined by the integral equation
\begin{align}
\tilde\psi(x,\xi) &= e^{i(x^1+ix^2)\xi}-G_\xi\ast(\tilde V\tilde\psi_1(\cdot,\xi)),\label{psi-implicit}
\end{align}
where 
\begin{align*}
G_\xi(w) &= \frac{e^{i\xi(w^1+iw^2)}}{4\pi^2}\int_{\RR^2}\frac{e^{iw\cdot\zeta}}{|\zeta|^2+2\xi(\zeta^1+i\zeta^2)}\, d\zeta
\end{align*}
is Faddeev's Green function (see \cite{faddeev1}, \cite{nachman_aniso}, \cite{SylUhl87}), and $\tilde V:= e^{2\phi}(V\circ\Phi^{-1})$ in $\Phi(\Omega)$ and is extended to be zero outside. Equation (\ref{psi-implicit}) is known to have unique solutions $\tilde\psi(x,\xi)$ with $e^{-i(x^1+ix^2)\xi}\tilde\psi(x,\xi)-1\in L^2_{-\delta}(\RR^2)$, for $|\xi|$ sufficiently large. Moreover, for large $|\xi|$, these solutions satisfy  
\begin{align}
||e^{-i(x^1+ix^2)\xi}\tilde\psi_1(x,\xi)-1||_{L^2_{-\delta}(\RR^2)}&\le \frac{C}{|\xi|}. \label{conf-outside-est1}
\end{align}

Consider the pullback $\psi(z,\xi):=\tilde\psi(\Phi_k(z),\xi)$. The functions $\psi(z,\xi)$ then satisfy the exterior problem (\ref{outside-bound1}), (\ref{outside-pde1}), (\ref{outside-condt1}). In addition, from the estimate (\ref{conf-outside-est1}) for $\tilde\psi(x,\xi)$, the estimate (\ref{outside-est1}) holds for $\psi(z,\xi)$. This proves existence. Uniqueness for $\psi$ follows from the uniqueness of $\tilde\psi$.
%

 \end{proof}



\begin{prop}\label{outside-same}
Let $(z^1,z^2)$ be coordinates on $\RR^2$, and $\Omega\subset \RR^2$ be a bounded domain with Lipschitz boundary $\partial \Omega$.  Set $g_1,g_2$ to be two $C^2$-smooth Riemannian metrics on $\Omega$.\\

For $k=1,2$, let $\afg_k$, $\Phi_k$, $V_k\in L^\infty(\Omega)$ be as in Proposition \ref{uniquness-up-down}. Define by $\Lambda_{k}$, $k=1,2$, the Dirichlet-to-Neumann maps associated to $\Delta_{g_k}\psi_k+V_k\psi_k$ in $\Omega$.\\

Then, if $\Lambda_1=\Lambda_2$, the conformal maps $\Phi_1(z)=x^1+ix^2$, $\Phi_2(z)=y^1+iy^2$ are equal on the exterior set $Z:=(\RR^2\setminus \Omega)\cup \partial \Omega$. In particular, the Dirichlet-to-Neumann maps determine the domain $\Phi_1(\Omega)=\Phi_2(\Omega)$.
 \end{prop}
 
\begin{proof}

%
%

Extend $V_k$, $k=1,2$ to all of $\RR^2$ such that $V_k=0$ outside a compact set containing $\Omega$, and $V_1=V_2$ is known on $\RR^2\setminus\Omega$. 

Let $\xi\in\CC\setminus\{0\}$ and $\delta$ sasify the conditions of Proposition \ref{uniquness-up-down} for $g_1$ and $g_2$. For $k=1,2$, consider the exterior problems 
\begin{align}
\psi_k(\cdot,\xi)\in L^2_\text{loc}(Z)\text{ and }&e^{-iz\xi}\psi_k(\cdot,\xi)-1\in L^2_{-\delta}(Z), \label{outside-bound}\\
\Delta_{g_k}\psi_k(\cdot,\xi)+V_k\psi_k(\cdot,\xi)&=0 \text{ on $Z$}\label{outside-pde}\\
g_k(\nabla\psi_k(\cdot,\xi),\nu)&=\Lambda_k\psi_k(\cdot,\xi)  \text{ on $\partial\Omega$},\label{outside-condt}
\end{align}
where $\nu$ is the outward pointing unit normal vector field to $\partial\Omega$. 

By Proposition \ref{uniquness-up-down}, there exists for $k=1,2$ unique families of solutions $\psi_k(\cdot,\xi)$ to (\ref{outside-bound}), (\ref{outside-pde}), (\ref{outside-condt}) which additionally satisfy 
\begin{align}
||e^{-i\Phi_k(z)\xi}\psi_k(\cdot,\xi)-1||_{L^2_{-\delta}(Z)}&\le \frac{C}{|\xi|}.\label{outside-est}
\end{align}

%
%

Since we imposed $V_1=V_2$ on $Z$, and from the assumption that $\Lambda_1=\Lambda_2$, $\psi_1(z.\xi)$ solves the same problem as $\psi_2(z,\xi)$ on $Z$. Proposition \ref{uniquness-up-down} gives uniqueness of the solutions $\psi_k(z,\xi)$ to the exterior problems (\ref{outside-bound}), (\ref{outside-pde}), (\ref{outside-condt}); thus $\psi_1(z,\xi)=\psi_2(z,\xi)$ on $Z$. we now show that this together with (\ref{outside-est}) implies $\Phi_1=\Phi_2$ on $Z$. 

Write $\psi(z,\xi):= \psi_1(z,\xi)=\psi_2(z,\xi)$. From the estimates on $\psi(z,\xi)$, we have
\begin{align}
\frac{2C}{|\xi|}\ge& \left|\left|e^{-i\Phi_2(z)\xi}\psi(\cdot,\xi)-e^{-i\Phi_1(z)\xi}\psi(\cdot, \xi)\right|\right|_{L^2_{-\delta}(Z)}\\
&=\left|\left|\left[e^{i(\Phi_1-\Phi_2)(z)\xi}-1\right]e^{-i\Phi_1(z)\xi}\psi(\cdot, \xi)\right|\right|_{L^2_{-\delta}(Z)}.\label{phi-est}
\end{align}

Using the above estimate and proof by contradiction, we show $\Phi_1(z)=\Phi_2(z)$ for $z\in Z$. 

Suppose $|\Phi_1(z_0)- \Phi_2(z_0)|>0$ for some $z_0\in Z$.  Without loss of generality, we may assume that $\text{Re}(\Phi_1(z_0)- \Phi_2(z_0))>0$. By the continuity of $\Phi_j$, $j=1,2$, there exists $\epsilon>0$ and $c>0$ such that 
\begin{align*}
\text{Re}(\Phi_1(z)- \Phi_2(z))>c
\end{align*}
for $z\in B_{z_0}(\epsilon)$.

Consider $\xi=0+i\xi^2$, for $\xi^2<0$. We find
\begin{align*}
\left|e^{i(\Phi_1-\Phi_2)(z)\xi}-1\right|&\ge \left|\left|e^{i(\Phi_1-\Phi_2)(z)\xi}\right|-1\right|\\
&\ge \left|e^{-\text{Re}(\Phi_1(z_0)- \Phi_2(z_0))\xi^2}-1\right|\\
&\ge \left|e^{-c\xi^2}-1\right|,
\end{align*}
for all $z\in B_{z_0}(\epsilon)$.

Taking $\xi^2\to-\infty$, we have $\left|e^{i(\Phi_1-\Phi_2)(z)\xi}-1\right|\to \infty$. This violates (\ref{phi-est}), since the right hand side goes to zero as $|\xi|\to\infty$.

Therefore, $|\Phi_1(z)- \Phi_2(z)|=0$ on $Z$. This completes the proof.

\end{proof} 

We note some further useful bounds on the conformal factors  $\phi_t$ for future use, 
which stem 
from the above estimates on $\Phi'$, along with the Gauss equation and standard elliptic estimates 
applied to the equation $\Delta_g\phi=-\frac{1}{2}R(g_t)$ (where $R(g_t)$ 
is the Gauss curvature of the metric $g|_{Y_t}$):
%

\begin{lemma}
\label{conf.bounds}
The conformal factors $\phi_t$ have small $\mathcal{C}^2$ norm in the close to
Euclidean 
setting. In the thin/straight setting, the conformal factors $\phi_t$ satisfy the bounds: 

\[
|\phi_t|\le C\delta_0, |\partial \phi_t|\le C\delta_0\epsilon^{-1}_0, 
|\partial^2 \phi_t|\le C\delta_0\epsilon^{-2}_0,
\]
for some universal constant $C>0$; here $\partial$ is the coordinate derivative.  
\end{lemma}

\begin{proof}[Sketch of Proof] The argument is based on standard elliptic estimates. 
It is convenient to consider the re-scaling of the metric $g_t$ by a factor 
$({\rm Area}[Y(t)])^{-1}$. Denote the resulting metric by $\tilde{g}_t$. 
We also rescale the underlying isothermal coordinates by $({\rm Area}[Y(t)])^{-1/2}$, 
 and denote the 
 conformal factor for $\tilde{g}_t$ 
 over the new coordinate system by $e^{2\tilde{\phi}_t}$. The function
$\tilde{\phi}_t$ 
 is just the push-forward of $\phi_t$ under the dilation map (with dilation factor 
 $({\rm Area}[Y(t)])^{-1/2}$). 
 
 In the equation 
 \[
 \Delta_{\tilde{g}_t}\tilde{\phi}_t=-\frac{1}{2}R(\tilde{g}_t)
 \]
 The right hand side now has a $W^{1,p}$ norm bounded by $C\delta_0$,
 for any $p>2$. The $C^2$ norm of the metric $\tilde{g}_t$ is uniformly bounded, via the bounds on the Beltrami coefficient and the assumed curvature bounds; thus we derive:
 
 \[
 ||\tilde{\phi}_t||_{W^{4,p}(\mathbb{R}^2)}\le C\delta_0\implies 
 ||\tilde{\phi}_t||_{\mathcal{C}^2}\le C\delta_0.
 \]
The above estimates imply the claimed bounds in the original metric $g_t$. 
\end{proof}

\begin{remark}\label{est.still.hold.note} We also make note of a consequence of the above bounds, which will be useful further
down: 

Given the formula $$(\Phi')^\top \afg_t\Phi'=:\Phi_*(\afg_t)=e^{2\phi_t}[(dx^1)^2+(dx^2)^2]$$ 
the $C^2$ bounds that we were assuming for the metric $g$ over $M$  in the 
original coordinate system continue to hold when $\afg_t$ is expressed in 
the new isothermal coordinate system $(x^1,x^2)$, up to increasing the constant in those bounds by
a  fixed amount. 
\end{remark}


\subsection{Least Area Data Implies Dirichlet-to-Neumann Data}
Once again, in this section we assume that $(M,g)$ is a $C^4$-smooth, Riemannian manifold which is homeomorphic to a 3-dimensional ball in $\RR^3$. Further, we suppose that $(M,g)$ has $C^4$-smooth, mean convex boundary $\partial  M$.

Recall that for a properly embedded minimal surface $\Omega\hookrightarrow M$, we defined the stability operator as $\Delta_{g_Y}+\left(\text{Ric}_g(\vec{n},\vec{n})+||A||^2_g\right)$ (see (\ref{stability-op})). Then,
\begin{defn}
The \textbf{Dirichlet-to-Neumann map} associated to the stability operator is the map
\begin{align*}
\Lambda_{g_{Y_\gamma}}&:H^{\frac12}(\partial Y_\gamma)\to H^{-\frac12}(\partial Y_\gamma)\\
\Lambda_{g_{Y_\gamma}}(\psi)&:= g_{Y_\gamma}(\nabla \psi,\nu)|_{\partial Y_\gamma},
\end{align*}
where $\psi$ solves (\ref{stability-op}), $\nu$ is the outward pointing normal vector field to the boundary $\partial Y_\gamma$ and ${g_{Y_\gamma}}$ is the metric induced on $Y_\gamma$ by $g$.
\end{defn}

We restate Proposition \ref{outside-same} in a form which connects the Dirichlet-to-Neumann map of the stability operator with knowledge of the isothermal coordinates (normalized at infinity) outside an open set.  
\begin{corollary}\label{DN-outside-same}
Let $g_1,g_2$ be two $C^4$-smooth metrics on $M$, and $\Omega\subset \RR^2$ be a bounded domain. Suppose $g_1=g_2$ on $\partial M$. Let $\Omega\hookrightarrow Y_k\subset M$ for $k=1,2$ be properly embedded, area-minimizing surfaces in $(M, g_k)$ with $\partial Y_1=\partial Y_2$.

As defined earlier, let $\afg_k$ be the extensions of $g_k$ to an asymptotically flat 3-manifold $\MM\supset M$. Set $\Phi_k$ to be the unique conformal maps inducing isothermal coordinates on the extensions of $Y_k$.

Consider the Dirichlet-to-Neumann map $\Lambda_{g_k}$ associated to the stability operator $$\Delta_{g_k|_{Y_k}}+\text{Ric}_{g_k}(\vec{n}_k,\vec{n}_k)+||A_{Y_k}||_{g_k}^2,$$ $k=1,2$. If $\Lambda_{g_1}=\Lambda_{g_2}$, then $\Phi_1=\Phi_2$ on $(\RR^2\setminus \Omega)\cup\partial\Omega$.
\end{corollary}

Key to all subsequent proofs in this paper, we now show that our minimal area data determines the Dirichlet-to-Neumann map associated to the stability operator on an area minimizing surface with boundary.

\begin{prop}\label{area-to-angle} 
 Let $g|_{\partial M}$ be given. Suppose that 
\begin{enumerate}
\item $M$ admits a properly embedded, area-minimizing foliation $Y(t)$
\item the area of each $Y(t)$ and any nearby perturbation of $Y(t)$ by area-minimizing surfaces is known. 
\end{enumerate}

Then,  the first variations of the area of  $Y(t)$ determine the angle at which $Y(t)$ cuts the boundary of $M$.
\end{prop}

\begin{proof}

Write $\gamma(t)=\partial Y(t)\subset\partial M$.  Let $\gamma(s,t):(-\eta,\eta)\times(-1,1)\to \partial M$ be a one parameter variation of $\gamma(t)$ by simple closed curves, chosen so that $\left.\frac{\partial}{\partial s}\right|_{s=0}\gamma(s,t)$ is a vector field tangent to the boundary.  Denote by $Y(s,t)$ the area minimizing surface circumscribed by $\gamma(s,t)$, and $A(s,t)$ the area of $Y(s,t)$. Write $X_0:= \left.\frac{\partial}{\partial s}\right|_{s=0}Y(s,t)$. By standard computations the first variation in the area of $Y(t)$ is
\begin{align*}
\left.\frac{\partial}{\partial s}A(s,t)\right|_{s=0}&=\int_{Y(t)}g(X_0,H) \VOL_{g_{t}} +\int_{\partial Y(t)} g(X_0,\nu)\,dS,
\end{align*}
where $g_{t}$ is the metric restricted to $Y(t)$ and $\nu$ is the unit outward pointing normal vector to $\partial Y(t)$ and tangent to $Y(t)$. Since we have assumed knowledge of the area of any minimal perturbation of $Y(t)$, we know $\left.\frac{\partial}{\partial s}A(s,t)\right|_{s=0}$; further, the fact that $Y(t)$ is minimal implies 
\begin{align*}
\left.\frac{\partial}{\partial s}A(s,t)\right|_{s=0}&= \int_{\partial Y(t)}g(X_0,\nu)\,dS.
\end{align*} 
Since $X_0$ is known, and is allowed to be any vector field tangent to the boundary, we have determined the angle at which $Y(t)$ cuts the boundary of $M$.

Note: the above argument holds for arbitrary dimension.

\end{proof}


\begin{prop}\label{n-dim_area_to_DN}
 Let $g|_{\partial M}$ be given. 

Let $\gamma(t)$, $t\in(-1,1)$ be a 1-parameter family of simple closed curves which foliate $\partial M$. For each $t$, let $Y(t)$ be the area-minimizing surface which is bounded by $\gamma(t)$; we assume that $\{Y(t)\,:\,t\in(-1,1)\}$ defines a foliation of $M$. Suppose that for each $t$, the area of $Y(t)$ and any nearby perturbation of $Y(t)$ by area-minimizing surfaces is known.   

Then, this data determines the Dirichlet-to-Neumann map associated to the stability operator on $Y(t)$.


\end{prop}

\begin{proof}


It is convenient to consider variations of $Y(t)$ that are normal to it at the boundary curve $\gamma$. From such variations, we discover information about the Dirichlet-to-Neumann map associated to the stability operator on each $Y(t)$. Such variations need not arise as variations of $\gamma(t)$ on the boundary of $M$; hence we smoothly extend $M$ and work with this extension. Let $N$ to be a tubular neighbourhood of $\partial M$. Let $\tilde{M}:=N\cup M$ and extend $g$ to a $C^4(\tilde{M})$-smooth metric $\tilde{g}$ on $\tilde{M}$. We further impose that $N$ was chosen so that the Riemannian manifold $(\tilde M,\tilde{g})$ has mean convex boundary $\partial\tilde M$. 

Next, we construct an auxiliary family of unique, area-minimizing surfaces in $\tilde M$ which we will vary normally to obtain information about a 1-parameter family of Dirichlet-to-Neumann maps which, loosely speaking, are close in some sense to the Dirichlet-to-Neumann map we seek to identify on $Y(t)$. Towards this end, for each fixed  $t$, we select a 1-parameter family simple closed curves in $\tilde M\setminus M$ which are $C^4(\tilde{M})$-close to $\gamma(t)$. Denote the curves in this family by $\tilde\gamma(t,\epsilon)$; here $\epsilon\in[0,1]$ and $\tilde\gamma(t,0):= \gamma(t)$.

\begin{figure}[h]
\centering
\includegraphics[width=0.8\textwidth]{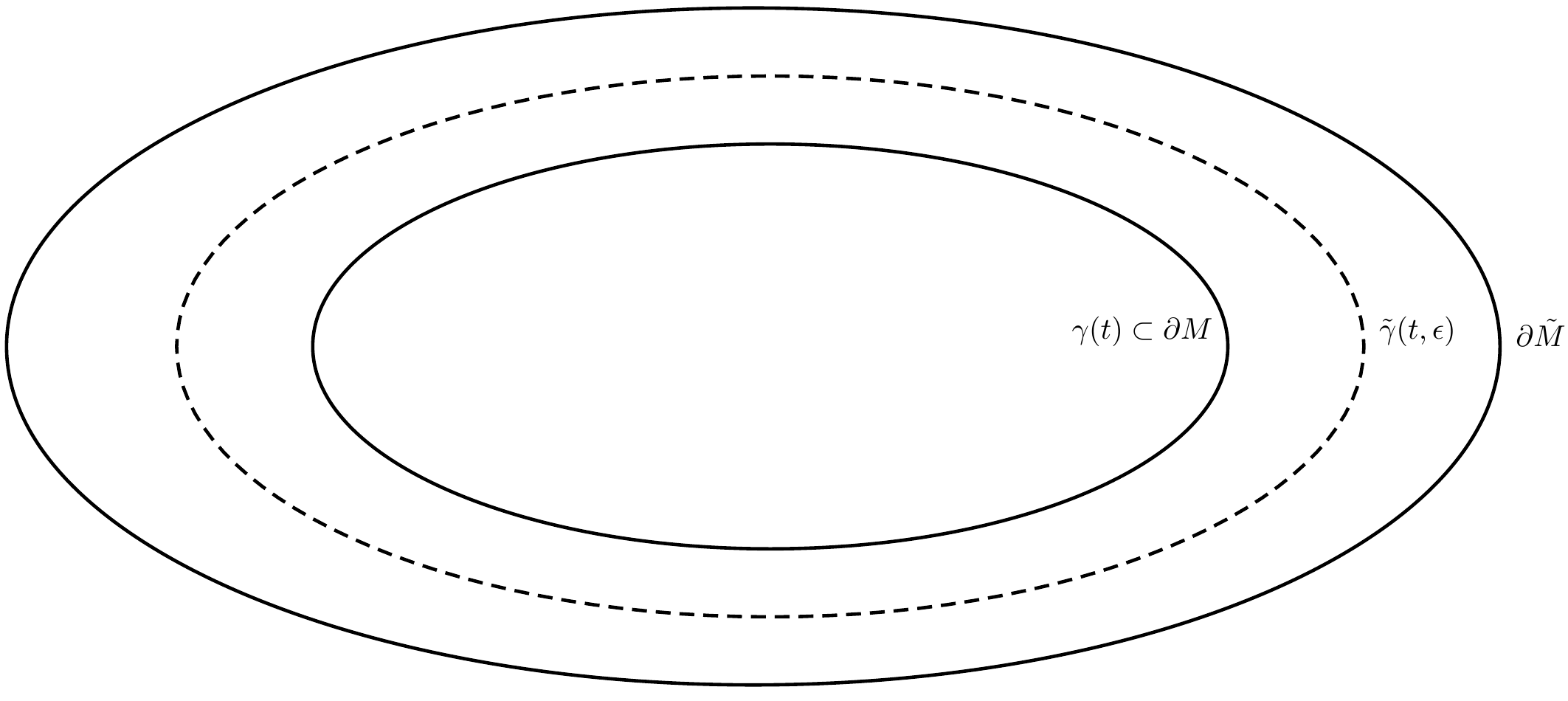}
\caption{Depiction of the curves $\gamma(t)$ and $\tilde{\gamma}(t,\epsilon)$.}
\end{figure}

We have the following two facts: since we have assumed that $\gamma(t)$ bounds a unique, area-minimizing surface, so too for every $t$ and small enough $\epsilon>0$ the curves $\gamma(t, \epsilon)$ bound a unique, area-minimizing surface. Thus, given some bounded domain $\Omega\subset \RR^2$, there exists embeddings
\begin{align*}
&\tilde{f}_{t,\epsilon}:\Omega \to \tilde M,\\
&\left.\tilde{f}_{t,\epsilon}\right|_{\partial \Omega}=\tilde\gamma(t,\epsilon),
\end{align*}
such that $\tilde{f}_{t,\epsilon}(\Omega)=:\tilde{Y}(t,\epsilon)$ is the unique surface which solves the least area problem for $\tilde\gamma(t,\epsilon)$.

Now we describe normal variations of $\gamma(t, \epsilon)$: for every $s\in[0,1]$, define $\tilde{\gamma}(s,t,\epsilon)$ to be a simple closed curve satisfying $\left.\frac{d}{ds}\right|_{s=0}\tilde{\gamma}(s,t,\epsilon)$ is parallel to the unit normal vector field $\vec{n}_{t,\epsilon}$ on the surface $�\tilde {Y}(t,\epsilon)$. Here we write $\tilde{\gamma}(0,t,\epsilon)=\tilde{\gamma}(t,\epsilon)$. Once more, the variations $\tilde{\gamma}(s,t,\epsilon)$ are $C^4(\tilde M)$-close to $\gamma(t)$, and since $\gamma(t)$ bounds a unique, area-minimizing surface, for each $s$, the curves $\tilde{\gamma}(s,t,\epsilon)$ bound a unique, area-minimizing surface. 

In particular, there exist embeddings $\tilde{f}_{s,t,\epsilon}:\Omega \to \tilde M$ satisfying
\begin{align*}
&\left.\tilde{f}_{s,t,\epsilon}\right|_{\partial \Omega}=\tilde\gamma(t,\epsilon),\\
&\tilde{f}_{s,t,\epsilon}(\Omega)=\tilde{Y}(s,t,\epsilon) \text{ solves the least area problem for } \tilde\gamma(s,t,\epsilon),\text{ and}\\
&\left.\frac{d}{ds}\tilde{f}_{s,t,\epsilon}\right|_{s=0}= \psi_{t,\epsilon}\vec{n}_{t\epsilon}.
\end{align*}
Moreover, the $C^4(\tilde{M})$-smooth function $\psi_{t,\epsilon}:\tilde{Y}(t,\epsilon)\to \RR$  
solves the boundary value problem
\begin{align}
\Delta_{\tilde{g}_{t,\epsilon}}\psi_{t,\epsilon}+\left(\text{Ric}_{\tilde{g}}(\vec{n}_{t,\epsilon},\vec{n}_{t,\epsilon})+||A_{t,\epsilon}||^2_{\tilde{g}}\right)\psi_{t,\epsilon}&=0,&&\text{on }\tilde{Y}(t,\epsilon)\\
\psi_{t,\epsilon}|_{\partial \tilde{Y}(t,\epsilon)}&=\psi_{t,\epsilon}^\sharp ,&& \text{on }\partial \tilde{Y}(t,\epsilon), \label{JEQ1}
\end{align} 
for prescribed boundary data $\psi_{t,\epsilon}^\sharp :=g\left(V,\vec{n}_{t,\epsilon}\right)$, $V:= \left.\frac{d}{ds}\tilde\gamma_{s,t,\epsilon}\right|_{s=0}$. Here $\tilde{g}_{t,\epsilon}$ is the metric $\tilde{g}$ restricted to $\tilde{Y}(t,\epsilon)$ and $A_{t,\epsilon}$ is the second fundamental form of $\tilde{Y}_{t,\epsilon}$.

We know the metric on $(\tilde{M}\setminus M)\cup \partial M$. Therefore, by the following lemma (Lemma \ref{interior-curve}), we determine the intersection of $\tilde{Y}(s,t,\epsilon)\cap\partial M$. From this, the area of $\tilde{Y}(s,t,\epsilon)$, denoted by $\text{Area}(s,t,\epsilon)$, is found. An easy computation shows that for each $(s,t,\epsilon)$, the second variation in area is the number given by
\begin{align*}
\left.\frac{\partial^2}{\partial s^2}\text{Area}(s,t,\epsilon)\right|_{s=0}&=\int_{\partial \tilde{Y}(t,\epsilon)} \psi_{t,\epsilon}\tilde{g}(\nabla\psi_{t,\epsilon},\nu_{t,\epsilon})+\tilde{g}(\nabla_{V}V,\nu_{t,\epsilon})\,dS\\
&\qquad -\int_{\tilde{Y}(t,\epsilon)}\psi_{t,\epsilon}\Delta_{\tilde{g}_{t,\epsilon}}\psi_{t,\epsilon}+\psi_{t,\epsilon}^2\left(\text{Ric}_{\tilde g}(\vec{n},\vec{n})+||A||^2_{\tilde g}\right) \VOL_{\tilde{g}_{\bar{Y}(t)}}\\
&= \int_{\partial \tilde{Y}(t,\epsilon)} \psi_{t,\epsilon}\tilde{g}(\nabla\psi_{t,\epsilon},\nu_{t,\epsilon})\,dS+\int_{\partial \tilde{Y}(t,\epsilon)}\tilde{g}(\nabla_{V}V,\nu_{t,\epsilon})\,dS,
\end{align*} 
where $\nu_{t,\epsilon}$ is the outward pointing normal to $\partial \tilde{Y}(t,\epsilon)$ and tangent to $\tilde{Y}(t,\epsilon)$.

The quantity $\int_{\partial \tilde{Y}(t,\epsilon)}\tilde{g}(\nabla_{V}V,\nu_{t,\epsilon})\,dS$ appearing in the above second variation of area is known, since the metric $\tilde{g}_{t,\epsilon}$ is given on $\tilde M\setminus M$, and both $\nu_{t,\epsilon}$  and $V$ are known on $\tilde\gamma(t,\epsilon)$. 
Therefore, 
\begin{align}
\int_{\partial \tilde{Y}(t,\epsilon)} \psi_{t,\epsilon} \tilde{g}(\nabla\psi_{t,\epsilon},\nu_{t,\epsilon})\,dS&=\left.\frac{\partial^2}{\partial s^2}\text{Area}(s,t,\epsilon)\right|_{s=0}-\int_{\partial \tilde{Y}(t,\epsilon)}\tilde{g}(\nabla_{V}V,\nu_{t,\epsilon})\,dS\notag\\
&= \text{ known quantity.} \label{DN1int}
\end{align}
Then, by polarizing, our area data has determined the functional $$L(\phi^\sharp  ,\psi^\sharp  ) :=\int_{\partial \tilde{Y}(t,\epsilon)} \phi^\sharp g(\nabla\psi,\nu_{t,\epsilon})\,dS$$ for any functions $\phi^\sharp  , \psi^\sharp  :\partial \tilde{Y}(t,\epsilon)\to\RR\in C^2(\tilde{Y}(t,\epsilon)).$

In particular, we have learned the Dirichlet-to-Neumann map $$\Lambda_{t,\epsilon}(\psi^\sharp  ):= \tilde{g}(\nabla\psi,\nu_{t,\epsilon})|_{\partial \tilde{Y}(t,\epsilon)}$$ associated to equation (\ref{JEQ1}). We remark that the operators $$\JJ_{t,\epsilon}:= \Delta_{\tilde{g}_{t,\epsilon}}+\text{Ric}_{\tilde{g}}(\vec{n}_{t,\epsilon},\vec{n}_{t,\epsilon})+||A_{t,\epsilon}||^2_{\tilde{g}}$$ are non-degenerate for $\epsilon>0$ small enough, since the eigenvalues of $\JJ_{t,\epsilon}$ depend continuously on the strictly negative eigenvalues of $\JJ_{t,0}$. Hence the Dirichlet-to-Neumann map $\Lambda_{t,\epsilon}$ is 
well-defined for each $t,\epsilon$.

Now since the surfaces $\tilde{Y}(t,\epsilon)$ are $C^4(\tilde{M})$-close to $\tilde{Y}(t)$, as $\epsilon\to0$ the component functions of the metrics $\tilde{g}_{t, \epsilon}$ tend to those of $\tilde{g}_{t,0}$ in the $C^4(\tilde{M})$-norm. Also, for each $t$, the potentials $\left(\text{Ric}_{\tilde{g}}(\vec{n}_{t,\epsilon},\vec{n}_{t,\epsilon})+||A_{t,\epsilon}||^2_{\tilde{g}}\right)$ converge to $\left(\text{Ric}_{\tilde{g}}(\vec{n}_{t,0},\vec{n}_{t,0})+||A_{t,0}||^2_{\tilde{g}}\right)$ in $C^1(\tilde{M})$ as $\epsilon\to 0$. Finally, since each $\psi_{t,\epsilon}$ depends continuously on $g_{t,\epsilon}$ and $\left(\text{Ric}_{\tilde{g}}(\vec{n}_{t,\epsilon},\vec{n}_{t,\epsilon})+||A_{t,\epsilon}||^2_{\tilde{g}}\right)$, the functions $\psi_{t,\epsilon}$ converge to $\psi_{t,0}$ in $C^4(\tilde{M})$. Take the limit as $\epsilon\to0$ of (\ref{DN1int}). Since $Y(t)=: Y(t,0)$, on the original leaf $Y(t)$ we learn
\begin{align*}
\int_{\partial \tilde{Y}(t)} \psi_{t,0} \tilde{g}(\nabla\psi_{t},\nu_t)\,dS&=\left.\frac{\partial^2}{\partial s^2}\text{Area}(s,t,0)\right|_{s=0}+\text{known quantity},
\end{align*}
and thus determine the Dirichlet-to-Neumann map  $$\Lambda_{t}(\psi^\sharp  ):= \tilde{g}(\nabla\psi,\nu_{t})|_{\partial \tilde{Y}(t)}$$ associated to the stability operator on $Y(t)$.

Now we prove our assumption about the boundaries and areas of $\tilde{Y}(t,\epsilon)$. 

\begin{lemma}\label{interior-curve}
Suppose that $(M,g)$ admits a foliation by properly embedded, area-minimizing surfaces. Let $(\tilde{M}, \tilde{g})$ be a smooth extension of $(M,g)$ such that $\tilde{g}|_{M}=g$, $\tilde{g}$ is known on $(\tilde{M}\setminus M)\cup \partial M$, and $\partial \tilde M$ is mean convex. 

Let $\gamma(t)$ be a given 1-parameter family of simple closed curves which foliate $\partial M$, and let $Y(t)$ be unique, area-minimizing leaves of the foliation induced on $M$ by solving the least-area problem for $\gamma(t)$. Suppose that for each $t$, the area of $Y(t)$ and any nearby perturbation of $Y(t)$ by area-minimizing surfaces is known. 

For each fixed $t$, choose $\tilde\gamma(t,\epsilon)$, $\epsilon\in[0,1]$ to be a family of simple closed curves which lie in $\tilde{M}\setminus (M\cup \partial M)$, and satisfy $\tilde\gamma(t,0)=\gamma(t)$ and $\tilde\gamma(t,\epsilon)$ is $C^4(\tilde{M})$-close to $\gamma(t)$. Define $Y(t)$ and $\tilde{Y}(t,\epsilon)$ be the surface of least area which bound $\gamma(t)$ and $\tilde\gamma(t,\epsilon)$, respectively.

Then,
 
 \begin{enumerate} 
 \item[a.]  We know the closed curve given by the intersection $c(t,\epsilon):=\tilde{Y}(t,\epsilon)\cap\partial M$.
 \item[b.] We know the area of $\tilde{Y}(t,\epsilon)$, with respect to $\tilde g$.
\end{enumerate}
\end{lemma}

\begin{proof}



a.  Let $c(t,\epsilon)$ the curve given by $\tilde{Y}(t,\epsilon)\cap\partial M$. Consider the set $\Sigma$ of all simple closed curves on $\partial M$ which are $C^4(\tilde{M})$-close to $c(t,\epsilon)$. For any curve $\sigma\in \Sigma$, denote by $Y_\sigma\subset M$ the surface which minimizes the area enclosed by $\sigma$. 

For $\sigma\in \Sigma$, let $A_{\sigma}$ be the area-minimizing annulus which lies between $\sigma$ and $\tilde{\gamma}(t,\epsilon)$. The metric $\tilde g$ is known on $(\tilde{M}(r)\setminus M)\cup \partial M$, so for any such annulus $A_{\sigma}$, we can determine the inward pointing (with respect to $A_{\sigma}$) unit normal vector field $\tilde\nu_\sigma$ tangent to $A_\sigma$ and normal to the curve $\sigma$. 

Now, given any $\sigma\in \Sigma$ the first variations in the area of $Y_\sigma$ determine the angle at which $Y_\sigma$ cut the boundary of $M$ (see Proposition \ref{area-to-angle}). Thus, we may determine the outward pointing (with respect to $Y_{\sigma}$) unit normal vector fields $\nu_\sigma$ which are tangent to $Y_\sigma$, and normal to the curve $\sigma$.

Consider the annulus $A_{c(t,\epsilon)}$. Notice that the vectors $\nu_{c(t,\epsilon)}$ and $\tilde{\nu}_{c(t,\epsilon)}$ are collinear, since the surface $\tilde{Y}(t,\epsilon)$ is smooth at the curve $c(t,\epsilon)$. We claim that $c(t,\epsilon)$ is the only curve in $\Sigma$ with this property. 

To show the uniqueness of $c(t,\epsilon)$, suppose that $\sigma^\sharp\in\Sigma$ is a curve such that $A_{\sigma^\sharp}$ is the minimal annulus for which $\nu_{\sigma^\sharp}$ and $\tilde\nu_{\sigma^\sharp}$ are collinear on $\partial M$. Note for any $p\in \sigma^\sharp$, the tangent space $T_pA_{\sigma^\sharp}$ coincides with the tangent space $T_pY_{\sigma^\sharp}$, since they are both spanned by $\nu_{\sigma^\sharp}$ and any vector tangent to $\sigma^\sharp$.  Hence, $A_{\sigma^\sharp}\cup Y_{\sigma^\sharp}$ is a $C^1(\tilde{M})$ surface which minimizes area bounded by $\tilde \gamma(t, \epsilon)$. This fact follows by general minimal surface theory, but we include a brief proof. We claim that $A_{\sigma^\sharp}\cup Y_{\sigma^\sharp}$ is in fact a smooth minimal surface and further $A_{\sigma^\sharp}\cup Y_{\sigma^\sharp}\equiv \tilde{Y}(t,\epsilon)$.

To prove that $A_{\sigma^\sharp}\cup Y_{\sigma^\sharp}$ is smooth, we express it as a graph of a function $z$ and show that the derivatives of $z$ exist and are continuous. To this end, let $T_{\sigma^\sharp}\subset \tilde M$ be the surface obtained by following geodesics $c_\theta(\rho)$ with $\theta\in \sigma^\sharp$ and initial direction $\frac{\partial }{\partial r}c_\theta(0)=\nu_{\sigma^\sharp}(\theta)$; that is  $T_{\sigma^\sharp}:=\{p\in \tilde M\,:\,p=c_\theta(\rho), \text{for some }\rho\in[-1,1],\,\theta\in\sigma^\sharp\}.\,$

Express $T_{\sigma^\sharp}$ in the natural coordinate system $(\rho,\theta)$. View $A_{\sigma^\sharp}$ as a graph of a function $z=z(\rho,\theta)$ over $T_{\sigma^\sharp}$. Again, we repeat that since $A_{\sigma^\sharp}\cup Y_{\sigma^\sharp}$ is smooth away from $\rho=0$, to show $A_{\sigma^\sharp}\cup Y_{\sigma^\sharp}$ is smooth, it suffices to show that  the second order derivatives of $z(\rho,\theta)$ at $\rho=0$ are continuous. This follows from standard elliptic regularity; the surfaces $A_{\sigma^\sharp}$ and $Y_{\sigma^\sharp}$ are minimal, hence $z=z(\rho,\theta)$ solves the minimal surface equation $$\DIV_g\left(\frac{\nabla z}{||\nabla z||_g}\right)=0.$$ 

Let $(\rho,\theta,z)$ form a local coordinate system near $T_{\sigma^\sharp}$. Since $A_{\sigma^\sharp}$ agrees with $T_{\sigma^\sharp}$ on $\sigma^\sharp$ to first order, $z(0,\theta)=0$ and $\partial_\rho z(0,\theta)=\partial_\theta z(0,\theta)=0$. The minimal surface equation for $A_{\sigma^\sharp}$ written in our chosen coordinates is
\begin{align*}
0&=\DIV_g\left(\frac{\nabla z}{||\nabla z||_g}\right)\\
&= \frac{||\nabla z ||_g\DIV_g(\nabla z)-g(\nabla z,g(\nabla \nabla z,\nabla z))}{||\nabla z ||_g^3}
\end{align*}
Substituting $\rho=0$ into the above equation and using $z(0,\theta)=0$, $\partial_\rho z(0,\theta)=\partial_\theta z(0,\theta)=0$, we determine $\nabla_{\tilde N}\nabla_{\tilde N} z$, where $\tilde N:=\frac{\nabla z}{||\nabla z||_g}$ is the unit normal to $A_{\sigma^\sharp}$ at $\rho=0$. 
Likewise $Y_{\sigma^\sharp}$ agrees with $T_{\sigma^\sharp}$ on $\sigma^\sharp$ to first order and is also a minimal surface, so by similar analysis we determine $\nabla_N\nabla_N z$, where $N:=\frac{\nabla z}{||\nabla z||_g}$ is the unit normal to $Y_{\sigma^\sharp}$ at $\rho=0$. In particular, at $\rho=0$ we have $\nabla_{\tilde N}\nabla_{\tilde N} z=\nabla_N\nabla_N z$.

So $z=z(\rho,\theta)$ is $C^2(\tilde M(r))$ on $A_{\sigma^\sharp}\cup Y_{\sigma^\sharp}$. Since $A_{\sigma^\sharp}$ agrees with $Y_{\sigma^\sharp}$ at $\sigma^\sharp$ up to second order, we have by elliptic regularity that $A_{\sigma^\sharp}\cup Y_{\sigma^\sharp}$ is smooth.

Now we show that $A_{\sigma^\sharp}\cup Y_{\sigma^\sharp}$ is unique. Since we have $Y(t)$ is a unique area minimizer for each $t$, the perturbed minimal surfaces $\tilde{Y}(t,\epsilon)$ are unique for $\epsilon$ small enough. Now, by construction $A_{\sigma^\sharp}\cup Y_{\sigma^\sharp}$ is $C^4$-close to $\tilde{Y}(t,\epsilon)$, and hence the surface $A_{\sigma^\sharp}\cup Y_{\sigma^\sharp}$ is unique.

In particular, the uniqueness of both $A_{\sigma^\sharp}\cup Y_{\sigma^\sharp}$ and $\tilde{Y}(t,\epsilon)$ implies that $A_{\sigma^\sharp}\cup Y_{\sigma^\sharp}\equiv \tilde{Y}(t,\epsilon)$. Therefore, $\sigma^\sharp\equiv c(t,\epsilon)$.

\vskip 1 em

b.  From part a, for any $t$ we may determine the curves $\gamma(t,\epsilon)$ cut by the intersection $\tilde Y(t,\epsilon)\cap \partial M$. In particular, we can find the area of the annulus $\tilde{Y}(t,\epsilon)\setminus(\tilde{Y}(t,\epsilon)\cap M)$.

We have 
\begin{align*}
\text{Area}(\tilde{Y}(t,\epsilon))&= \text{Area}(\tilde{Y}(t,\epsilon)\setminus(\tilde{Y}(t,\epsilon)\cap M))+\text{Area}(\tilde{Y}(t,\epsilon)\cap M).
\end{align*}
Since the metric $\tilde g$ is known on $\tilde{Y}(t,\epsilon)\setminus(\tilde{Y}(t,\epsilon)\cap M)$, we may compute this area. Since we assumed the knowledge of any minimal surface $M$, the area of $\tilde{Y}(t,\epsilon)\cap M$ in known. Therefore, $\text{Area}(\tilde{Y}(t,\epsilon))$ is known.

\end{proof}

\end{proof}



\begin{repprop}{area_to_DN} 
Let $(M,g)$ be a $C^4$-smooth, Riemannian manifold which is homeomorphic to a 3-dimensional ball in $\RR^3$, and has mean convex boundary $\partial M$. Let $g$ be given on ${\partial M}$. Let $\gamma$ be a given simple closed curve on ${\partial M}$, and set $Y_\gamma\subset M$ to be a surface of least area bounded by $\gamma$. Suppose that the stability operator on $Y_\gamma$ is non-degenerate, and that for $\gamma$ and any nearby perturbation $\gamma(s)$, the area of the least-area surface $Y_{\gamma(s)}$ enclosed by $\gamma(s)$ is known. 

Equip a neighbourhood of $Y_\gamma$ with coordinates $(x^\alpha)$ such that on $Y_\gamma$, $x^3=0$ and $(x^1,x^2)$ are isothermal coordinates. Then, 
\begin{enumerate}
\item the first and second variations of the area of $Y_\gamma$ determine the Dirichlet-to-Neumann map associated to the boundary value problem 
\begin{align*}
\Delta_{g_\EE}\psi+e^{2\phi}\left(\text{Ric}_g(\vec{n},\vec{n})+||A||^2_g\right)\psi&=0,\\
\psi|_{\partial Y_\gamma}&=g\left(\left.\frac{d}{ds}\gamma(s)\right|_{s=0},\vec{n}\right),
\end{align*} 
on $Y_\gamma$, where $e^{2\phi}g_\EE=e^{2\phi}[(dx^1)^2+(dx^2)^2]$ is the metric on $Y_\gamma$ in the coordinates $(x^1,x^2)$.
\item Knowledge of the first and second variations of the area of $Y_\gamma$ determines any solution $\psi(x)$ to the above boundary value problem with  given boundary data $\psi|_{\partial Y_\gamma}$, in the above isothermal coordinates.
\end{enumerate}
\end{repprop} 

\begin{proof}

Without loss of generality, set $t=0$. Let $g_0$ be the metric $g$ restricted to $Y(0)$. From Lemma \ref{n-dim_area_to_DN}, the minimal area data enables us to find the Dirichlet-to-Neumann map $ \Lambda_{g_0}(\psi):= g(\nabla\psi,\nu)|_{\partial \tilde\Omega}$ associated to the boundary value problem for the stability operator 
\begin{align}\label{JEQ2}
\Delta_{g_0}\psi+\left(\text{Ric}_g(\vec{n},\vec{n})+||A||^2_g\right)\psi&=0&&\text{on } \tilde\Omega,\\
\psi|_{\partial \Omega}&=\psi_0 &&\text{on }\partial \tilde\Omega.
\end{align}
In our chosen coordinate system, the metric $g_0$ pulled back to $\tilde\Omega$ takes the form $g_0=e^{2\phi}g_\EE$. In these coordinates, the problem (\ref{JEQ2}) is transformed to
\begin{align}\label{JEQ3}
\Delta_{g_\EE}\psi+e^{2\phi}\left(\text{Ric}_g(\vec{n},\vec{n})+||A||^2_g\right)\psi&=0&&\text{on } \tilde\Omega,\\
\psi|_{\partial \Omega}&=\psi_0 &&\text{on }\partial \tilde\Omega,
\end{align}
and the solutions $\psi(x)$ of (\ref{JEQ2}) are the same as the solutions $\psi(x)$ of (\ref{JEQ3}).

Using the isothermal coordinates $(x^1,x^2)$,
\begin{align*}
\int_{\partial\tilde\Omega}\psi\Lambda_{g_0}(\psi)\, dS&:=\int_{\partial\tilde\Omega} \psi \nu(\psi)\, dS\\
&= \int_{\partial\tilde\Omega} \psi e^{-\phi}\tilde\nu(\psi)\, e^{\phi}d\tilde{S}\\
&= \int_{\partial\tilde\Omega} \psi \Lambda_{g_\EE}(\psi)\, d\tilde{S}
\end{align*}
where $\tilde \nu=e^{-\phi}\nu$ is the unit outward pointing normal with respect to the Euclidean metric $g_\mathbb{E}$, and $\tilde\nabla$ denotes the gradient of $\psi:= \psi(x^1,x^2)$ with respect to the metric $g_\mathbb{E}$.
 
By polarizing, the knowledge of the area of any minimal surface in $M$ has determined the Dirichlet-to-Neumann map $\Lambda_{g_\mathbb{E}}$ associated to the Schr\"{o}dinger equation in (\ref{JEQ3}), with respect to the Euclidean metric.

Employing the result in \cite{IN} for linear Schr\"{o}dinger equations, the Dirichlet-to-Neumann map $\Lambda_{g_\mathbb{E}}$ determines the potential $$e^{2\phi}\left(\text{Ric}_g(\vec{n},\vec{n})+||A||^2_g\right)$$ on $\tilde\Omega$. Now that we know this potential in coordinates $(x^1,x^2)$,  all solutions $\psi(x^1,x^2)$ to the Dirichlet problem (\ref{JEQ3}) are known.

\end{proof} 


\section{Equations for the Components of the Inverse Metric}




\begin{prop}\label{inf_sep_prop} Let $(M,g)$ be a $C^4$-smooth, Riemannian manifold. Let $\gamma(t)$ be a foliation of $\partial M$ by simple closed curves. Suppose that $(M,g)$ admits a non-degenerate foliation by properly embedded, area-minimizing surfaces $Y(t)$ with $\partial Y(t)=\gamma(t)$.  Further, suppose that for $\gamma(t)$ and any nearby perturbation $\gamma(s,t)\subset \partial M$, the area of the least-area surface $Y(s,t)$ enclosed by $\gamma(s,t)$ is known. 

As in section \ref{AFcoords}, extend $(M,g)$ to an asymptotically flat manifold $(\MM,\afg)$ and extend each $Y(t)$ smoothly to $\YY(t)$ in $\MM$. Equip $\MM$ with coordinates $(x^\alpha)$, $\alpha =1,2,3$ such that $x^3=t$ and for each $t$ fixed $x^1,x^2:\YY(t)\to\RR$ are the unique conformal coordinates given by Proposition \ref{uniquness-up-down}. 

In these coordinates, $g^{33}:=||\nabla x^3||_{g}^2$, may be recovered on $M$ from the area data. \end{prop}

\begin{proof} 
Recall $\nabla x^3:= \text{grad}(x^3)$. Set $\vec{n}:=\frac{\nabla x^3}{||\nabla x^3||_{g}}$ to be a unit normal vector field on $Y(t)$ for $t\in (-1,1)$, and write $g_t$ for the restriction of the metric $g$ to the surface $Y(t)$.

For each fixed $t$, we may view the nearby leaves of the foliation $Y(t+\delta t)$ as a variation of $Y(t)$ by area-minimizing surfaces. From this viewpoint, the variation is captured by the vector field $$\frac{\partial }{\partial x^3}=:\partial_3.$$

The associated lapse function is 
\begin{align*}
g(\partial_3,\vec{n}) &= g\left(\partial_3, \frac{\nabla x^3}{||\nabla x^3||_{g}}\right)\\
&=||\nabla x^3||_{g}.
\end{align*}

Recall $x^k:Y(t)\to\RR$, $k=1,2$, are conformal on $Y(t)$. 
Since $||\nabla x^3||_g:Y(t)\to \RR$ is a nontrivial solution of the Jacobi equation
\begin{align}
\Delta_{g_t}\omega+\left(\text{Ric}_g(\vec{n},\vec{n})+||A||^2_g\right)\omega&=0\label{jac1}
\end{align}
on $Y(t)$ (see appendix A.3), the stability operator $\Delta_{g_t}+\left(\text{Ric}_g(\vec{n},\vec{n})+||A||^2_g\right)$ is non-degenerate for each $t$.

Therefore, written in the coordinates $(x^\alpha)$,  for $x^3=t$ fixed, the function $||\nabla x^3||_g$ solves  
\begin{align}
\Delta_{g_\EE}\psi+e^{2\phi(t)}\left(\text{Ric}_g(\vec{n},\vec{n})+||A||^2_g\right)\psi&=0\label{jacobieq}
\end{align}
on $Y(t)$, where the metric on $Y(t)$ is expressed as $g_t=e^{2\phi(t)}[(dx^1)^2+(dx^2)^2]=:e^{2\phi(t)}g_\EE$.

Now, we know $g|_{\partial M}$ and the curves $\partial Y(t)=(x^3)^{-1}(t)\cap\partial M$. Thus, the function $||\nabla x^3||_{g}$ on $\partial Y(t)$ is known. By Proposition \ref{area_to_DN}, we determine the lapse function $||\nabla x^3||_g$ on $Y(t)$, in the conformal coordinate system given by $(x^1,x^2, x^3=t)$. Since we now know $||\nabla x^3||_g$ on $Y(t)$ for any $t\in(-1,1)$, we have determined $||\nabla x^3||_g$ on $M$.

We have $$g^{33}:= g(dx^3,dx^3)=||\nabla x^3||_g^2,$$ in the chosen coordinates $(x^\alpha)$, $\alpha =1,2,3$. Hence, in these coordinates, the metric component $g^{33}$ is known on $M$.

\end{proof}



%

\begin{lemma} \label{sep_lin_lem}
Let $(M,g)$ be a $C^4$-smooth, Riemannian manifold. Let $\gamma(t)$ be a foliation of $\partial M$ by simple closed curves. Suppose that $(M,g)$ admits a foliation by properly embedded, area-minimizing surfaces $Y(t)$ with $\partial Y(t)=\gamma(t)$.  Further, suppose that for $\gamma(t)$ and any nearby perturbation $\gamma(s,t)\subset \partial M$, the area of the least-area surface $Y(s,t)$ enclosed by $\gamma(s,t)$ is known. 

Extend $M$ and each $Y(t)$ to asymptotically flat manifolds $\MM$ and $\YY(t)$ as defined in section \ref{AFcoords}. Further, set $(x^1,x^2)=\Phi(\cdot,t):\YY(t)\to\RR^2$ to be unique isothermal coordinates on $\YY(t)$ given by Proposition \ref{uniquness-up-down}; write $\tilde\Omega(t):=\Phi(Y(t))$. Set $x^3=t$.

Consider a point $p\in Y(t)$. Let $h:[0,S]\times\tilde\Omega(t)\to\MM$, $h(s,x^1,x^2,t)=:Y(s,t)$ be a variation of $Y(t)\subset \MM$ by properly embedded, area-minimizing surfaces which has the property that the component of $h_*\left(\left.\frac{\partial }{\partial s}\right|_{s=0}\right)$ projected onto the normal vector field to $Y(t)$, denoted by $\psi_{p}=\psi_p(x^1,x^2)$, vanishes at $p$.\footnote{It is not always the case that such a variation exits. We do not prove the existence here, but  later in Section 4.}  Set $(x^1_s,x^2_s)=\Phi(\cdot,s,t):\YY(s,t)\to\RR^2$ to be the unique isothermal coordinates on the extended, new foliation $\YY(s,t)$.\\

Then, the linearization of $||\nabla x^3_s||_g$ at the point $p$ is 
\begin{align}
\left.\frac{d}{ds}||\nabla x^3_s||_g(p)\right|_{s=0}
&=g^{3\alpha}(p)\partial_\alpha\psi_{p}(x^1(p),x^2(p))+\partial_\alpha||\nabla x^3||_g(p)\dot{x}^\alpha(p),\label{equation-lin-sep}
\end{align}
where $\dot{x}^\alpha:=\left. \frac{d}{ds}x_s^k\right|_{s=0}$ is the first variation in the coordinate functions $x^\alpha_s$ at $p$, for $\alpha=1,2,3$.\\

Moreover, the quantity $\left.\frac{d}{ds}||\nabla x^3_s||_g(p)\right|_{s=0}$ is known in the coordinates $(x^\alpha)$, $\alpha =1,2,3$.
\end{lemma}

\begin{figure}[h]
\centering
 \includegraphics[width=0.8\textwidth]{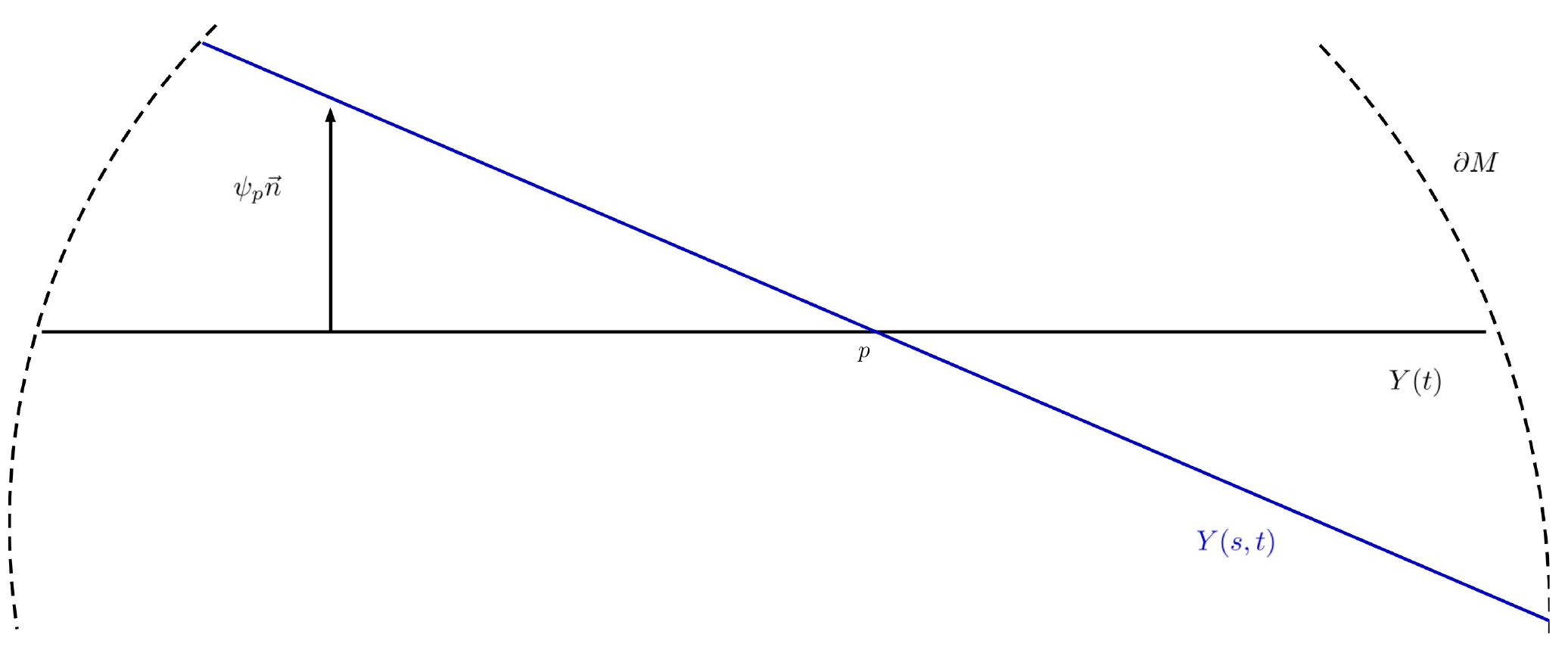}
  \caption{Depiction of the leaves $Y(s,t)$.}
 \end{figure}

\begin{proof}
 Let $\vec{n}:=\frac{\nabla x^3}{||\nabla x^3||_g}$ denote the unit normal vector field to $Y(t)$.  

Via Taylor expansion, the new coordinate functions $(x^\alpha_s)$, $\alpha=1,2,3$ on $Y(s,t)$ in terms of the ``original'' coordinate functions $(x^\alpha)$ on $Y(t)$ are expressed as 
\begin{align}
x^\alpha_s&=x^\alpha+s\dot{x}^\alpha+\BO(s^2).\label{xseq}
\end{align}

Then, linearizing $||\nabla x^3_s||_g^2(p)$ about $s=0$,  
\begin{align*}
\left.\frac{d}{ds}\right|_{s=0}||\nabla x^3_s||_g^2(p)&=\left.\frac{d}{ds}\right|_{s=0}\left[\left(||\nabla x^3||^2_g+2sg(\nabla x^3,\nabla \dot{x}^3)\right)\circ(p)+\BO(s^2)\right]\notag\\
&= 2g(\nabla x^3,\nabla \dot{x}^3)(p)+\partial_\alpha||\nabla x^3||^2_g(p)\dot{x}^\alpha(p)\\
&= 2g^{3\alpha}(p)\partial_\alpha[||\nabla x^3||_g\psi_{p}](x^1(p),x^2(p))+\partial_\alpha||\nabla x^3||^2_g(p)\dot{x}^\alpha(p)\\
&=2||\nabla x^3||_g(p)g^{3\alpha}(p)\partial_\alpha\psi_{p}(x^1(p),x^2(p))+\partial_\alpha||\nabla x^3||^2_g(p)\dot{x}^\alpha(p), 
\end{align*}
at the chosen point $p$.

Now, by Proposition \ref{inf_sep_prop}, the function $||\nabla x^3_s||_g^2(p)$ is known on $M$ for $s\ge0$. Hence $\left.\frac{d}{ds}\right|_{s=0}||\nabla x^3_s||_g^2(p)$ is known.

\end{proof}
\begin{figure}[h]
\centering
 \includegraphics[width=0.6\textwidth]{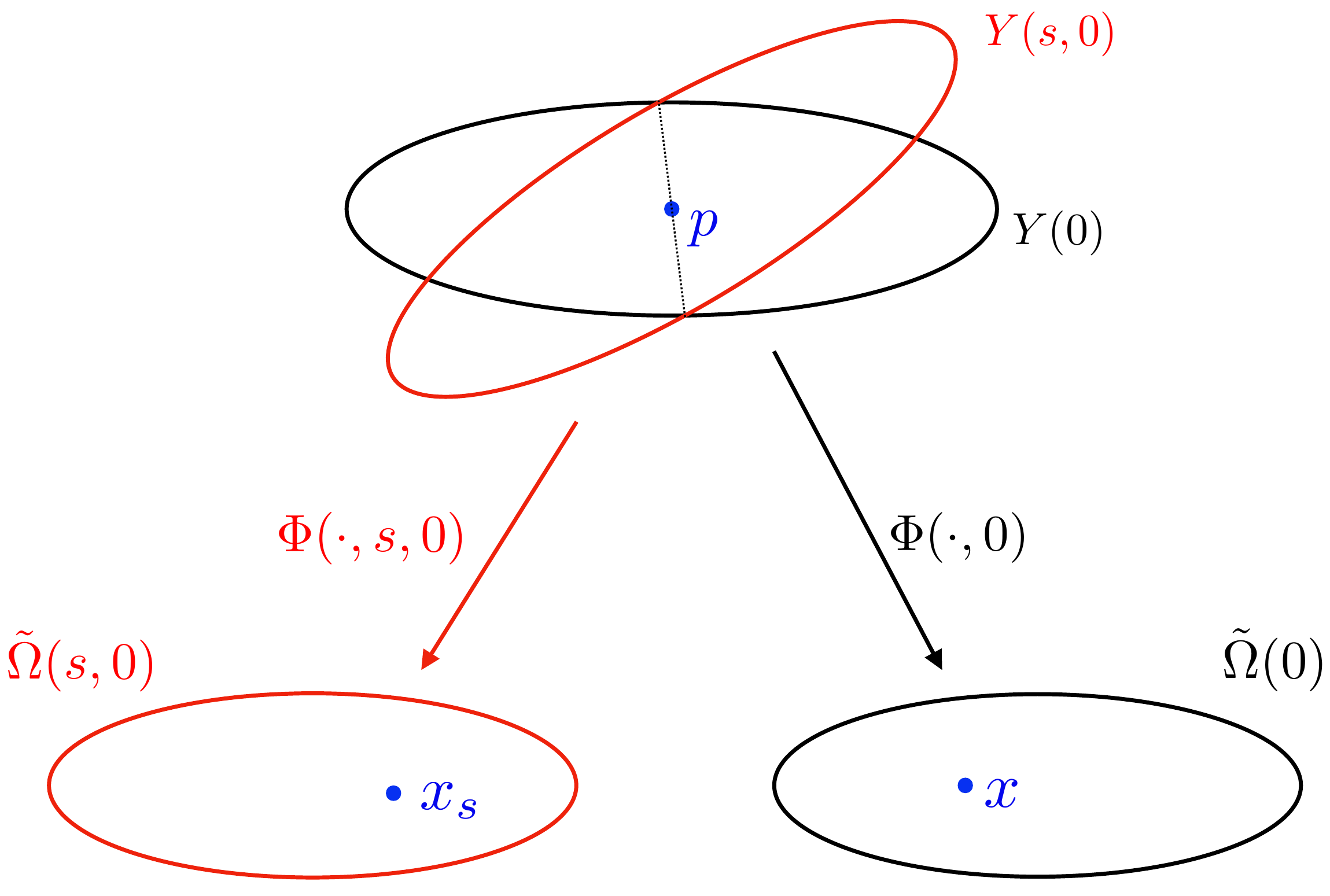}
  \caption{Depiction of the coordinate maps $\Phi(\cdot,0)$ and $\Phi(\cdot,s,0)$.}
 \end{figure}


Since the left hand side equation (\ref{equation-lin-sep}) is known, and the function $\psi_p$ as defined above is known from Proposition \ref{area_to_DN}, we would like to use this equation to solve for $g^{3k}$, $k=1,2$. However, solving equation (\ref{equation-lin-sep}) is complicated due to the presence of the terms containing $\dot{x}^k$, $k=1,2$. It will thus serve our purposes to find an expression for $\dot{x}^k$ in terms of $g^{13},g^{23}$ and $\phi$. The calculations for such an expression are carried out below.

\begin{lemma} \label{sep_lin_dotx}
 Let $(M,g)$, $\tilde\Omega(t)$, $Y(t)$, $Y(s,t)$, $p\in Y(t)$, and $\psi_{p}:\tilde\Omega(t)\to\RR$ be as defined in Lemma \ref{sep_lin_lem}. For $\alpha=1,2,3$, let $x^\alpha:Y(t)\to\RR$ and $x^\alpha_s:Y(s,t)\to\RR$  be the  coordinate systems on $Y(t)$ and $Y(s,t)$, as defined in Lemma \ref{sep_lin_lem}, and write $\dot{x}^k:Y(t)\to\RR$, $k=1,2$ for the first order change in the isothermal coordinates $x^k_s$. 

Then on $Y(t)$, for given variations of $\partial Y(t)$ the functions $\dot{x}^k$, $k=1,2$ are determined via a Poisson equation 
\begin{align*}
\Delta_{g_\mathbb{E}}\dot{x}^k &= \mathcal{F}^k(g^{13},g^{23},\phi,\psi_{p},d\psi_{p},p),
\end{align*}
where $\mathcal{F}^k$ us given explicitly below in (\ref{dotx2}) and (\ref{dotxEQfinal}), $\phi=\phi(x^1,x^2,t)$ is the conformal factor on $Y(t)$, and $\mathcal{F}^k$ is a second order differential operator acting on $g^{13}$, $g^{23}$, and $\phi$.
\end{lemma}

\begin{proof} 
 Without loss of generality, fix $t=0$ and consider $Y(0)$. Write $g_0$ for the metric induced by $g$ on $Y(0)$, and $g_{s,0}:=g|_{Y(s,0)}$ for the metric induced on the leaves $Y(s,0)$. Recall from Lemma \ref{sep_lin_lem}, we express the foliation $Y(s,0)$ as embeddings $h:[0,S]\times\tilde\Omega(0)\to\MM$ into the extension $\MM$ of $M$; that is, $h(s,x^1,x^2,0)=Y(s,0)$.

The equation (\ref{xseq}) which expresses $x^k_s$ in terms of $x^k$ is 
\begin{align*}
x^k_s&=x^k+s\dot{x}^k+\BO(s^2).
\end{align*}

Now, to compute how $\dot{x}^k$ depends on the components of the metric $g$, we linearize $x^k_s$ in $s$. 

The conformal coordinates $(x^k_s)$ on the leaves $Y(s,0)$ are harmonic functions, and thus satisfy $$\Delta_{g_{s,0}}x_s^1=0=\Delta_{g_{s,0}}x_s^2.$$ Linearizing about $s=0$ and noting $\partial_jx_s^k=\delta_j^k$, we derive
\begin{align}
0&=\frac{d}{ds}[ \Delta_{g_{s,0}}x_s^k]\notag\\\notag
&= \left[\frac{d g_{s,0}^{ij}}{ds}\partial_i\partial_jx_s^k-\frac{d}{ds}[g_{s,0}^{ij}\Gamma^k_{ij}(g_{s,0})]+\Delta_{g_{s,0}}\frac{d}{ds}x_s^k+\mathcal{O}(s)\right]_{s=0}\\\notag
&= 0-(\dot{g}_{0})^{ij}\Gamma^k_{ij}(g_{0,0})-g_{0,0}^{ij}\dot{\Gamma}^k_{ij}(g_{0,0})+\Delta_{g_{0,0}}\frac{d}{ds}x_s^k|_{s=0}\\\notag
&= -(\dot{g}_{0})^{ij}\Gamma^k_{ij}(g_{0,0})-\frac{1}{2}g_{0,0}^{ij}g_{0,0}^{kl}[\nabla_j(\dot{g}_{0})_{il}+\nabla_i(\dot{g}_{0})_{jl}-\nabla_l(\dot{g}_{0})_{ij}]+\Delta_{g_{0,0}}\dot{x}^k\\\notag
&=  -\dot g^{ij}\Gamma^k_{ij}(g_{0,0})-g_{0,0}^{ij}\nabla_j(g_{0,0}^{kl}(\dot{g}_{0})_{il})+\frac{1}{2}g_{0,0}^{kl}\nabla_l(g_{0,0}^{ij}(\dot{g}_{0})_{ij})+\Delta_{g_{0,0}}\dot{x}^k,\label{firstxeq}
\end{align}
on $Y(0)$, for $k=1,2$.

To perform further analysis, we require and expression for the linearization $\dot{g}_0$ of the induced metric on the leaves $Y(s,0)$, as well as the Christofffel symbols associated to the metric $g_{0,0}:=g_0$ on $Y(0)$.

In all computations which follow, let $i,j,k,l,m,p$ sum over $1,2$ and $\alpha,\beta,\gamma$ sum over $1,2,3$. In the coordinates $(x^\alpha)$, $g_0=e^{2\phi}g_{\EE}$. Hence, the Christoffel symbols of $g_0$ are 
\begin{align*}
\Gamma^k_{ij}(g_0)&=\Gamma^k_{ij}(e^{2\phi}{g_\mathbb{E}})\\
&= \Gamma^k_{ij}({g_\mathbb{E}})+{g_\mathbb{E}}^k_i\partial_j\phi+{g_\mathbb{E}}^k_j\partial_i\phi-(g_\mathbb{E})_{ij}{g_\mathbb{E}}^{kl}\partial_l\phi\\
&= {g_\mathbb{E}}^k_i\partial_j\phi+{g_\mathbb{E}}^k_j\partial_i\phi-(g_\mathbb{E})_{ij}{g_\mathbb{E}}^{kl}\partial_l\phi.
\end{align*}

For ease of computation of the linearization $\dot{g}_0$, we employ Gaussian coordinates adapted to $Y(0)$: for $i=1,2$, define the coordinate vector fields $X_i:= h(\cdot,s,0)_{*}\left(\frac{\partial}{\partial x^i}\right)$ and $X_s:= h(\cdot,s,0)_{*}\left(\frac{\partial}{\partial s}\right)$. Then in these coordinates, the components of the metric $g_{s,0}$ induced on the leaves $Y(s,0)$ are given by $(g_{s,0})_{ij}:=g(X_i,X_j)$. 

Now Taylor expand $g_{s,0}$ in terms of $s$:  $g_{s,0}= g_0+s\dot{g}_{0}+\mathcal{O}(s^2)$.  Then,
\begin{align*}
(\dot{g}_{0})_{ij}:=\left.\frac{d}{ds}(g_{s,0})_{ij}\right|_{s=0} &=\left. \frac{d}{ds}g(X_i,X_j)\right|_{s=0}\\
&= [g(\nabla_{X_s}X_i,X_j)+g(X_i,\nabla_{X_s}X_j)]|_{s=0}\\
&= g(\nabla_{X_i}(\psi_{p}\vec{n}),X_j)+g(X_i,\nabla_{X_j}(\psi_{p}\vec{n}))\\
&= -2\psi_{p} A_{ij},
\end{align*}
where $A_{ij}$ are the components of the second fundamental form of $Y(0)$. Thus, the first order change in $g_{s,0}$ is given by the coordinate free expression
\begin{align}
\dot{g}_{0}&= -2\psi_{p} A.
\end{align}

Recall $g_{0,0}=g_0$, and substitute $\dot{g}_{0}=-2\psi_{p} A$ into equation (\ref{firstxeq}); the resulting PDE describes the first variation in $s$ of the coordinates $(x^i_s)$:
\begin{align}
\Delta_{g_0}\dot{x}^k&=  -2\psi_{p} A^{ij}\Gamma^k_{ij}(g_0)-2g_{0}^{ij}\nabla_j(g_0^{kl}\psi_{p} A_{il})+g_0^{kl}\nabla_l(g_0^{ij}\psi_{p} A_{ij})\notag\\
&=-2\psi_{p} A^{ij}\Gamma^k_{ij}(g_0)-2g_0^{ij}\nabla_j(\psi_{p} A_i^k).\label{dotx1}
\end{align}
Note the term $g_0^{kl}\nabla_l(g_0^{ij}\psi_{p} A_{ij})$ is zero since the surface $Y(0)$ is minimal.

We now expand each term in equation (\ref{dotx1}) in terms of the components of $g$ and $g^{-1}$ which we aim to uniquely determine. To this end, a quick calculation gives that in the coordinates $(x^\alpha)$, the normal vector field to the leaves $Y(t)$ is
\begin{align*}
\vec{n}&:= \frac{\nabla x^3}{||\nabla x^3||_g}=  \frac{1}{||\nabla x^3||_g}g^{\alpha\beta}\partial_\beta x^3\partial_\alpha=  \frac{1}{||\nabla x^3||_g}g^{\alpha3}\partial_\alpha.
\end{align*}
Hence the components of the second fundamental form are
\begin{align*}
A_{ij}
&=-\frac{1}{2||\nabla x^3||_g}g^{3\alpha}(\partial_ig_{\alpha j}+\partial_jg_{i\alpha}-\partial_\alpha g_{ij}).
\end{align*}  
Raising an index and noting $-g_{\alpha j}\partial_ig^{3\alpha}=\partial_ig_{\alpha j}g^{3\alpha}$ then gives
\begin{align}
A_i^k
&=-\frac{e^{-2\phi}(g_\EE)^{jk}}{2||\nabla x^3||_g}(g_{\alpha j}\partial_ig^{3\alpha}+g_{i\alpha}\partial_jg^{3\alpha}+g^{3\alpha}\partial_\alpha g_{ij}).\label{2ndForm}
\end{align}

So we calculate a factor of the first term in (\ref{dotx1}) to be
\begin{align*}
-2||\nabla x^3||_gg^{im}_0A_m^j\Gamma^k_{ij}(g_0)&=e^{-2\phi}\cdot(g_\EE)^{im}\cdot e^{-2\phi}(g_\EE)^{jl}\cdot\left(g_{l \alpha}\partial_mg^{3\alpha}+g_{m \alpha}\partial_lg^{3\alpha}+g^{3\alpha}\partial_\alpha g_{ml}\right)\cdot\\
&\qquad\left({g_\mathbb{E}}^k_i\partial_j\phi+{g_\mathbb{E}}^k_j\partial_i\phi-(g_\mathbb{E})_{ij}{g_\mathbb{E}}^{kp}\partial_p\phi\right)\\
%
%
%
&= 2e^{-4\phi}\left\{g_\EE^{km}g_\EE^{jl}g_{3l}\partial_mg^{33}\partial_j\phi+g_\EE^{im}g_\EE^{kl}g_{3l}\partial_mg^{33}\partial_i\phi\right.\\
&\qquad-g_\EE^{kj}g_\EE^{ml}g_{3l}\partial_mg^{33}e^{2\phi}\partial_j\phi+g_\EE^{km}e^{2\phi}\partial_mg^{3j}\partial_j\phi\\
&\qquad+g_\EE^{im}e^{2\phi}\partial_mg^{3k}\partial_i\phi-g_\EE^{kj}\partial_mg^{3m}e^{4\phi}\partial_j\phi+\left.0\right\}.
\end{align*}

For the second term in (\ref{dotx1}), using (\ref{2ndForm}), observe the partial coordinate derivatives of the components of the second fundamental form are 
\begin{align*}
2\partial_j(A_i^k)&=-2\partial_j\phi A^k_i-\frac{1}{||\nabla x^3||_g}\partial_j||\nabla x^3||_gA^k_i+\frac{e^{-2\phi}(g_\EE)^{km}}{||\nabla x^3||_g}\left\{\partial_jg_{\alpha m}\partial_ig^{3\alpha}+g_{\alpha m}\partial_j\partial_ig^{3\alpha}\right.\\
&\qquad +\partial_jg_{\alpha i}\partial_mg^{3\alpha}+ g_{\alpha i}\partial_j\partial_mg^{3\alpha}+2e^{2\phi}(g_\EE)_{im}\partial_jg^{3\alpha}\partial_\alpha\phi\\
&\qquad+\left.2e^{2\phi}(g_\EE)_{im}g^{3\alpha}\partial_j\partial_\alpha\phi-4e^{2\phi}(g_\EE)_{im}g^{3\alpha}\partial_\alpha\phi\partial_j\phi\right\}.
\end{align*}
Substituting the expressions for $\partial_j(A_i^k)$ and $g^{im}A_m^j\Gamma^k_{ij}(g_0)$ above into equation (\ref{dotx1}), the equation for the first order change in conformal coordinates is given below (indices run over values $\alpha\in\{1,2,3\}$ and $i,j,k,l,m\in\{1,2\}$).
\begin{align}
\Delta_{g_{0}}\dot{x}^k &=-2\psi_{p} A^{ij}\Gamma^k_{ij}(g_0)-2g_0^{ij}\nabla_j(\psi_{p} A_i^k)\notag\\
&=-2\psi_{p} g^{im}A^j_m\Gamma^k_{ij}(g_0)-2g_0^{ij}\nabla_j(\psi_{p})A_i^k-2g_0^{ij}\psi_{p}[\partial_jA_i^k-\Gamma^m_{ij}(g_0)A_m^k+\Gamma^k_{mj}(g_0)A_i^m]\notag\\
&=-2g_0^{ij}\nabla_j(\psi_{p})A_i^k-2\psi_{p} g_0^{ij}\partial_jA_i^k-4\psi_{p} g^{im}A^j_m\Gamma^k_{ij}(g_0)\notag\\ 
&=-g_0^{ij}\nabla_j(\psi_{p})\frac{e^{-2\phi}(g_\EE)^{kl}}{||\nabla x^3||_g}(g_{\alpha l}\partial_ig^{3\alpha}+g_{i\alpha}\partial_lg^{3\alpha}+g^{3\alpha}\partial_\alpha g_{il})\notag\\
&\qquad+g_0^{ij}\psi_{p}\partial_j\phi\frac{e^{-2\phi}(g_\EE)^{kl}}{||\nabla x^3||_g}(g_{\alpha l}\partial_ig^{3\alpha}+g_{i\alpha}\partial_lg^{3\alpha}+g^{3\alpha}\partial_\alpha g_{il})\notag\\
&\qquad+g_0^{ij}\psi_{p}\frac{1}{||\nabla x^3||_g}\partial_j||\nabla x^3||_g\frac{e^{-2\phi}(g_\EE)^{kl}}{2||\nabla x^3||_g}(g_{\alpha l}\partial_ig^{3\alpha}+g_{i\alpha}\partial_lg^{3\alpha}+g^{3\alpha}\partial_\alpha g_{il})\notag\\
&\qquad -g_0^{ij}\psi_{p}\frac{e^{-2\phi}(g_\EE)^{km}}{||\nabla x^3||_g}\left\{\partial_jg_{\alpha m}\partial_ig^{3\alpha}+g_{\alpha m}\partial_j\partial_ig^{3\alpha}+\partial_jg_{\alpha i}\partial_mg^{3\alpha}+g_{\alpha i}\partial_j\partial_mg^{3\alpha}\right.\notag\\
&\qquad+\left.2e^{2\phi}(g_\EE)_{im}\partial_jg^{3\alpha}\partial_\alpha\phi+2e^{2\phi}(g_\EE)_{im}g^{3\alpha}\partial_j\partial_\alpha\phi-4e^{2\phi}(g_\EE)_{im}g^{3\alpha}\partial_\alpha\phi\partial_j\phi\right\}\notag\\
&\qquad -8\psi_{p}e^{-4\phi}\left\{g_\EE^{km}g_\EE^{jl}g_{3l}\partial_mg^{33}\partial_j\phi+g_\EE^{im}g_\EE^{kl}g_{3l}\partial_mg^{33}\partial_i\phi-g_\EE^{kj}g_\EE^{ml}g_{3l}\partial_mg^{33}e^{2\phi}\partial_j\phi\right.\notag\\
&\qquad+\left.g_\EE^{km}e^{2\phi}\partial_mg^{3j}\partial_j\phi+g_\EE^{im}e^{2\phi}\partial_mg^{3k}\partial_i\phi-g_\EE^{kj}\partial_mg^{3m}e^{4\phi}\partial_j\phi\right\}.
\end{align}
We may express this complicated PDE schematically as
\begin{align}
\Delta_{g_0}\dot{x}^k  &=e^{-2\phi}\psi_{p} A^{ijk}_{\alpha}\partial_i\partial_jg^{3\alpha}+e^{-2\phi}\psi_{p} B_{\alpha}^{ik}\partial_i\partial_\alpha\phi+e^{-2\phi}\psi_{p} C^{ijk}_\alpha\partial_i\phi\partial_jg^{3\alpha}\notag\\
&\qquad+e^{-2\phi}\psi_{p} D^{ijk}_{\alpha \beta}\partial_ig^{3\alpha}\partial_jg^{3\beta}+e^{-2\phi}\psi_{p} F^{ik\alpha }\partial_i\phi\partial_\alpha\phi\notag\\
&\qquad+e^{-2\phi}(\psi_{p} H^{ik}_\alpha+\nabla_j\psi_{p}I^{ijk}_{\alpha})\partial_i g^{3\alpha}+(\psi_{p}J^{k\alpha}_1+\nabla_i\psi_{p}J^{ik\alpha}_2)\partial_\alpha \phi,\notag\\
&=:e^{-2\phi}\mathcal{F}^k(g^{13},g^{23},\phi,\psi_{p},p)\label{dotx2}
\end{align}
where $A^{ijk}_{l},B^{ik\alpha},...,J^{ik\alpha}_{j}$ are smooth functions of $\phi,g^{13},g^{23},g^{33}$, and the indices range over $i,j,k,l\in\{1,2\}$, and $\alpha,\beta\in\{1,2,3\}$. Now, since $g_0=e^{2\phi}g_\mathbb{E}$, $\Delta_{g_0}=e^{-2\phi}\Delta_{g_\mathbb{E}}$; so on $Y(0)$ we have the equation 
\begin{align}
\Delta_{g_\mathbb{E}}\dot{x}^k &= \mathcal{F}^k(g^{13},g^{23},\phi,\psi_{p},p),\label{dotxEQfinal}
\end{align}
 where the differential operator $\mathcal{F}^k(g^{13},g^{23},\phi,\psi_{p},p)$ is defined in (\ref{dotx2}).

\end{proof}


Equation (\ref{dotxEQfinal}) together with equations of the form of  (\ref{equation-lin-sep}) will allow us to solve for the metric components, $g^{3k}$, $k=1,2$, in terms of the conformal factor, $\phi$. Thus we require one more equation to ultimately determine all components of the metric. Such an equation will be provided by a transport-type equation for the conformal factor $\phi=\phi(x^1,x^2,t)$ on the leaf $Y(t)$, which we derive in the next proposition.


\begin{prop}\label{phi_lem}
Let $(M,g)$ be a $C^4$-smooth, compact, 3-dimensional Riemannian manifold  and let 
$\gamma(t)$ be a foliation of $\partial M$ by simple closed curves. Suppose that $(M,g)$ admits a foliation by properly embedded, area-minimizing surfaces $Y(t)$ with $\partial Y(t)=\gamma(t)$.  Further, suppose that for $\gamma(t)$ and any nearby perturbation $\gamma(s,t)\subset \partial M$, the area of the least-area surface $Y(s,t)$ enclosed by $\gamma(s,t)$ is known. 

Extend $M$ and each $Y(t)$ to asymptotically flat manifolds $\MM$ and $\YY(t)$ as defined in section \ref{AFcoords}. Further, set $(x^1,x^2)=\Phi(\cdot,t):\YY(t)\to\RR^2$ to be unique isothermal coordinates on $\YY(t)$ given by Proposition \ref{uniquness-up-down}. Set $x^3=t$.


Then, the evolution in $x^3$ of the conformal factor $\phi=\phi(x^1,x^2,t)$ is described by the transport-type equation 
\begin{align}
g^{31}\partial_1 \phi+g^{32}\partial_2 \phi+g^{33}\partial_3 \phi+\frac{1}{2}\partial_kg^{3k}-\frac{1}{2}g^{3 k}\partial_k\log(g^{33})=0,\label{meancurve1}
\end{align}
where $g_t:= e^{2\phi(t)}g_\EE$ is the metric on the leaf $Y(t)$ in the coordinate system $(x^\alpha)$. 
\end{prop}

\begin{proof}
Recall the mean curvature of $Y(t)$ is the trace of the second fundamental form: $H(Y(t)):= A^i_i$ (we do not average over the dimension). 

As demonstrated in the proof of Lemma \ref{sep_lin_dotx}, equation (\ref{2ndForm}), the second fundamental form may be written as
\begin{align*}
A_i^k&=-\frac{e^{-2\phi}(g_\EE)^{jk}}{2||\nabla x^3||_g}(g_{\alpha j}\partial_ig^{3\alpha}+g_{i\alpha}\partial_jg^{3\alpha}+g^{3\alpha}\partial_\alpha g_{ij})
\end{align*}
in the coordinates $(x^\alpha)$, $\alpha=1,2,3$.

Therefore the mean curvature of $Y(t)$ is given by
\begin{align*}
H&:=A^i_i
=-\frac{e^{-2\phi}}{||\nabla x^3||_g}(g_{3 1}\partial_1g^{33}+e^{2\phi}\partial_kg^{3k}+g_{3 2}\partial_2g^{33}+2e^{2\phi}g^{3\alpha}\partial_\alpha \phi),
\end{align*}
where $k$ sums over $1,2$.

Since $Y(t)$ is minimal for each $t\in \RR$, $H(Y(t))=0$ provides the differential equation
\begin{align*}
 0&=e^{-2\phi}(g_{3 1}\partial_1g^{33}+g_{3 2}\partial_2g^{33})+(\partial_kg^{3k}+2g^{3\alpha}\partial_\alpha \phi)
\end{align*}
which we rewrite as
\begin{align}
g^{31}\partial_1 \phi+g^{32}\partial_2 \phi+g^{33}\partial_3 \phi+\frac{1}{2}\partial_kg^{3k}+\frac{e^{-2\phi}}{2}(g_{3 1}\partial_1g^{33}+g_{3 2}\partial_2g^{33})=0.\label{phiindet}
\end{align}
As shown in the appendix, we can express the components $g_{31},g_{32}$ in terms of the components of the inverse metric as 
\begin{align*}
g_{31}&= - \frac{g^{31}}{g^{33}}e^{2\phi},\quad g_{32}= - \frac{g^{32}}{g^{33}}e^{2\phi}.
\end{align*}
 Substituting the above into equation (\ref{phiindet}), we obtain
  \begin{align*}
g^{31}\partial_1 \phi+g^{32}\partial_2 \phi+g^{33}\partial_3 \phi+\frac{1}{2}\partial_kg^{3k}-\frac{1}{2}[g^{3 1}\partial_1\log(g^{33})+g^{3 2}\partial_2\log(g^{33})]=0.
\end{align*}

\end{proof}

\emph{Remark.} Notice that if $g^{13}$, $g^{23}$, and $g^{33}$ were known functions on $Y(t)$, then equation (\ref{meancurve1}) reduces to a simple first order, linear differential equation for the conformal factor $\phi$, which can be easily solved.


\section{Proof of the Main Theorems}



In this section, we prove the main theorems stated in the introduction. We first prove: 

\begin{reptheorem}{global_metric_thm}
Let $(M,g)$ be a manifold of Class 1 or Class 2, and $g|_{\partial M}$ be given. Let 
$\{\gamma(t)\,:\,t\in(-1,1)\}=\partial M$ and $\{Y(t)\,:\,t\in(-1,1)\}=M$ be as in Definition \ref{class1and2}. Suppose that for each curve $\gamma(t)$ and any nearby perturbation 
$\gamma(s,t)\subset \partial M$, we know the area of the properly embedded surface $Y(s,t)$
 which solves the least-area problem for $\gamma(s,t)$.  

Then the knowledge of these areas uniquely determines the metric $g$ (up to isometries which 
fix the boundary).
\end{reptheorem}

\subsection{Proof of Theorem \ref{global_metric_thm}}
\begin{proof}[Proof of Theorem \ref{global_metric_thm}]
To show $g_1$ is isometric to $g_2$ on $M$, we construct coordinate systems on $(M,g_1)$ and $(M,g_2)$, and explicitly construct a diffeomorphism $F:(M,g_1)\to(M,g_2)$ which maps one coordinate system to the other. In this setting, we prove that the components of the inverses of the metrics $g_1$ and $F^*(g_2)$ satisfy  
$$g_1^{\alpha\beta}-F^*(g_2)^{\alpha\beta}=0.$$
This equation implies our uniqueness result.

\bigskip
\subsubsection{Construction of the diffeomorphism $F:(M,g_1)\to(M,g_2)$:} 
As in section \ref{AFcoords}, extend $(M,g_1)$ to an asymptotically flat manifold $(\MM,\afg_1):=(\RR^2\times(-1,1),\afg_1)$. Smoothly extend each leaf $Y_1(t)$ to an asymptotically flat manifold $\YY_1(t)$ as defined in section \ref{AFcoords}. Further, set $(x^1,x^2)=\Phi_1(\cdot,t):\YY_1(t)\to\RR^2$ to be unique isothermal coordinates on $\YY_1(t)$ given by Proposition \ref{uniquness-up-down}; write $\tilde\Omega_1(t):=\Phi_1(Y_1(t))$. Set $x^3=t$.

Let $Y_2(t)$, $t\in(-1,1)$, be a foliation of $(M,g_2)$ by properly embedded, area minimizing surfaces which is found by solving the least-area problem for $\gamma(t)$. As in section \ref{AFcoords}, we also extend $(M,g_2)$ to an asymptotically flat manifold $(\MM,\afg_2):=(\RR^2\times(-1,1),\afg_2)$ and smoothly extend each leaf $Y_2(t)$ to an asymptotically flat manifold $\YY_2(t)$. As above, we set $(y^1,y^2)=\Phi_2(\cdot,t):\YY_2(t)\to\RR^2$ to be unique isothermal coordinates on $\YY_2(t)$ given by Proposition \ref{uniquness-up-down}; write $\tilde\Omega_2(t):=\Phi_1(Y_2(t))$. Set $y^3=t$.

Then, define 
\begin{align*}
F&:(\MM,\afg_1)\to(\MM,\afg_2)\\
F&(p)=\Phi_2^{-1}\circ \Phi_1(p).
\end{align*}
From Proposition \ref{outside-same}, in section 2, we know $\tilde\Omega_1(t)=\tilde\Omega_2(t)$ for all $t\in(-1,1)$ and $F=\text{Id}$ on $\MM\setminus M\cup \partial M$. The restriction of $F$ to $(M,g_1)$ is then our desired diffeomorphism. 

\textbf{Notation:} Abusing notation, we write $g_2$ for the pulled-back metric $F^*(g_2)$ on 
$M$ in all that follows. Note that in the $(x^\alpha)$ coordinates, the metrics $g_1$ and 
$g_2$ take the form 
\begin{align*}
g_1 &= (g_1)_{3\alpha}dx^3dx^\alpha+e^{2\phi_1(t)}\left[(dx^1)^2+(dx^2)^2\right]\\
g_2 &= (g_2)_{3\alpha}dx^3dx^\alpha+e^{2\phi_2(t)}\left[(dx^1)^2+(dx^2)^2\right].
\end{align*}
Further, since Proposition \ref{outside-same} shows our area data determines 
$\tilde\Omega_1(t)=\tilde\Omega_2(t)$ for all $t\in(-1,1)$, by choosing a new family of
 conformal maps, we may assume that $\tilde\Omega_1(t)=\tilde\Omega_2(t)=D(r(t))$ for some 
 disc of radius $r(t)>0$, where $r(t)$ is chosen via the requirement: 
\[
4\pi r^2(t)={\rm Area}[Y(t)].
 \]
 (We write $r$ instead of $r(t)$ below for simplicity of notation).  We note that given the regularity assumptions on the boundaries of 
 $Y(t)$, the conformal factors corresponding to the two new conformal maps satisfy the same 
 bounds as the conformal factor over the domain $\tilde\Omega_1(t)=\tilde\Omega_2(t)$,
 in Lemma \ref{conf.bounds}
 up to a uniformly bounded multiplicative factor.

We make this choice for simplicity in the proofs to follow. Again abusing notation, we still write $(x^1,x^2)=\Phi_1$ and $(y^1,y^2)=\Phi_2$ for the resulting maps to the discs $D(r(t))$.

\bigskip


\begin{lemma}\label{g33-equiv}
In the coordinate system described above, $g_1^{\alpha\beta}-g_2^{\alpha\beta}=0$ on $\MM\setminus M$, and $g_1^{33}-g_2^{33}=0$ on $\MM$.
\end{lemma}
\begin{proof}
Since $F=\text{Id}$ on $\MM\setminus M\cup \partial M$, $g_1^{\alpha\beta}-g_2^{\alpha\beta}=0$ on $\MM\setminus M$.

By the hypothesis on $M$, we know the area of $Y_k(t)$ for all $t\in(-1,1)$,, $k=1,2$ as measured by $g_1$ and $g_2$ respectively. Proposition \ref{inf_sep_prop} tells us this area information determines the lapse functions $||\nabla x^3||_{g_1}^2$ and $||\nabla y^3||_{g_2}^2$. Moreover, since we have assumed the area of $Y_1(t)$ as measured by $g_1$ equals the area of $Y_2(t)$ as measured by $g_2$ for each $t\in(-1,1)$, by Proposition \ref{area_to_DN} the Dirichlet-to-Neumann maps associated to the stability operators on the leaves are equal. Hence, in the coordinates $(x^\alpha)$,
\begin{align*}
g^{33}_1=||\nabla x^3||_{g_1}^2 &= ||\nabla y^3||_{g_2}^2\circ F = g_2^{33}.
\end{align*}

\end{proof}

Next, we prove uniqueness for all the remaining metric components by showing the differences $g^{\alpha\beta}_1-g^{\alpha\beta}_2$ vanish on $M$. First, in the coordinates $(x^\alpha)$, we derive a system of equations for the differences 
\begin{align*}
\delta g^{3j}:=g_1^{3j}-g^{3j}_2, \text{ and }\mathbf{\delta\phi}:=\phi_1-\phi_2
\end{align*}
on $\MM$, $j=1,2$. Then, using Lemma \ref{sep_lin_lem} and Lemma \ref{sep_lin_dotx}, we will express $\delta g^{3j}$ as a linear combination of pseudodifferential operators acting on $\delta\phi$ or $\partial_3\delta\phi$. The requirement that $(M,g_i)$, $i=1,2$, are either $C^3$-close to Euclidean or $(K,\epsilon_0,\delta_0)$-thin will play a crucial role here. After this has been achieved, it will suffice to show that $\delta\phi=\phi_1-\phi_2=0$.

We use Proposition \ref{phi_lem} and the pseudodifferential expressions for $\delta g^{3j}$ to obtain a hyperbolic Cauchy problem for $\delta\phi$. Here too, our assumption that $(M,g_i)$, $i=1,2$, are either $C^3$-close to Euclidean or $(K,\epsilon_0,\delta_0)$-thin will be key to obtaining the hyperbolic problem for $\delta\phi$. The desired result $\delta\phi=\phi_1-\phi_2=0$ will then follow from a very standard energy argument.


\subsubsection{Derivation of a system of equations for the metric components:}\label{system_section}
First consider the foliation $\{Y_1(t)\,:\, t\in(-1,1)\}$ of $(M,g_1)$. Notice that in the coordinates $(x^\alpha)$, the gradient of the function $x^3:M\to\RR$  is
\begin{align*}
\nabla x^3&:= g^{\alpha\beta}_1\partial_\beta x^3\partial_\alpha= g^{\alpha3}_1\partial_\alpha.
\end{align*}
Geometrically, the components of the inverse metric $g_1^{3\alpha}$, for $\alpha=1,2,3$, correspond to rescalings of the components of the normal vector field $\vec{n}_1:=\frac{\nabla x^3}{||\nabla x^3||_{g_1}}$ to a leaf $Y_1(t):= \{x^3=t\}$. Further, we saw above that if we view $Y_1(t+\delta t)$ as a variation of $Y_1(t)$ by area-minimizing surfaces, the lapse function associated to the foliation $\{Y_1(t)\,:\,t\in(-1,1)\}$, is $$g_1(\partial_3,\vec{n}_1):= ||\nabla x^3||_{g_1}=\sqrt{g_1^{33}}.$$ Below we will consider two chosen variations $\{Y_{i,1}(s,t)\,:\,t\in(-1,1),s\in[0,S]\}$, $i=1,2$, of the foliation $Y_1(t)$, $t\in(-1,1)$. Knowledge of the areas of the leaves $Y_1(t)$ and $Y_{i,1}(s,t)$ $i=1,2$ allows us to recover information about the \emph{new} lapse function for the new foliation, and we use this to find equations which describe the differences $\delta g^{3k}$.

We construct two variations of $Y_1(t)$ as follows: Consider a point $p\in Y_1(t)$. For $i=1,2$, define the maps $h_i:[0,S]\times D(r(t))\to\MM$, $h_i(s,x^1,x^2,t)=:Y_{i,1}(s,t)$ to be a variation of $Y_1(t)\subset \MM$ by properly embedded, area-minimizing surfaces which has the property that the component of $(h_i)_*\left(\left.\frac{\partial }{\partial s}\right|_{s=0}\right)$ projected onto the normal to $Y(t)$, denoted by 
$$\psi_{p,i,1}=g\left((h_i)_*\left(\left.\frac{\partial }{\partial s}\right|_{s=0}\right),\vec{n}_1\right),$$  
vanishes at $p$. We further require that $\nabla\psi_{p,i,1}(x(p)) =\frac{\partial}{\partial x^i}$.  We write $(x^1_s,x^2_s,s,t)=\Phi_1(\cdot,s,t):\YY_1(s,t)\to\RR^2$ for the unique isothermal coordinates on a smooth, asymptotically flat extension $\YY_1(s,t)$ of the new foliation $Y_1(s,t)$ (see Proposition \ref{uniquness-up-down}).

The existence of the desired foliations $Y_{i,1}(s,t):=h_i(s,x^1_t,x^2_t)$, $s\in [0,S]$, $t\in(-1,1)$, is equivalent to the existence of the desired functions $\psi_{p,i,1}$. In particular, the existence of the new foliations $Y_{i,1}(s,t):=h_i(s,x^1,x^2,t)$, $s\in [0,S]$, $t\in(-1,1)$, follows from the (uniform) non-degeneracy of the stability operator on each leaf $Y_1(t)$.  In the coordinate system $(x^\alpha)$, $\alpha=1,2,3$, we claim that from Proposition \ref{area_to_DN}, for each $t\in(-1,1)$ and for each $p\in Y_1(t)\subset M$, we may construct two distinct, nontrivial solutions $\psi_{p,i,1}\in C^2(\RR^2)$, $i=1,2$, of the Jacobi equation 
$$\Delta_{g_\EE}\psi_{p,i,1}+e^{2\phi_1}\left(\text{Ric}_{g_1}(\vec{n},\vec{n})+||A||^2_{g_1}\right)\psi_{p,i,1}=0$$
on $\RR^2$, which additionally satisfy  
\begin{align*}
\psi_{p,i,1}(x(p))&=0, \\
\nabla\psi_{p,i,1}(x(p)) &=\frac{\partial}{\partial x^i}.
\end{align*}
Moreover, from the area data for each $Y_1(t)$ and nearby perturbation, Proposition \ref{area_to_DN} implies that these functions $\psi_{p,i,1}$ are known in our chosen coordinates system $(x^\alpha)$, $\alpha=1,2,3$.

%


In the setting of the first case of Theorem \ref{global_metric_thm}, we have imposed that all the metric components are $C^3$-close to Euclidean, and we will show this is enough to obtain the functions $\psi_{p,i,1}$. In the setting of the second case of Theorem \ref{global_metric_thm}, the condition on the size of each of the the radii of $D(r(t))$ will play the crucial role in place of the $C^3$-close assumption, and we will show that we are able to construct the functions $\psi_{p,i,1}$ in this case too. We also seek such solutions $\psi_{p,i,1}$ which have suitable bounds in the following function spaces:

Let $\Omega \subset \RR^2$. For $k=0,1,2,\dots$, define
\begin{align*}
C^k_x(\Omega)&:=\left\{f(x,p):\Omega\times\Omega\to \RR\,|\, \frac{\partial^{i+j}}{\partial x^i\partial x^j}f\text{ is continuous for all }i+j\le k\right\}\\
C^k_p(\Omega)&:=\left\{f(x,p):\Omega\times\Omega\to \RR\,|\, \frac{\partial^{i+j}}{\partial p^i\partial p^j}f\text{ is continuous for all }i+j\le k\right\}.
\end{align*}



\begin{lemma}\label{psi-lem-thin} Let $(M, g)$, $\epsilon_0$, and $Y(t)$ satisfy the conditions of the first or second case of Theorem \ref{global_metric_thm}. Write $\vec{n}$ for the unit normal vector field to $Y(t)$, and $A$ the second fundamental form of $Y(t)$. For $0< r(t)\le \epsilon_0$, set $(x^1,x^2)\subset D(r(t))$ to be isothermal coordinates on $Y(t)$, so $g|_{Y(t)}:=e^{2\phi(t)}g_\EE$. Suppose $V:=e^{2\phi(t)}\left(\text{Ric}_{g}(\vec{n},\vec{n})+||A||^2_{g}\right) \in W^{1,q}(D(r(t)))$ with $q>2$. 

Then, for any $p\in D(r(t))$, and fixed $i=1,2$, there exists a function $\psi_{p,i}\in C^{3}_x(D(r(t)))$ which satisfies 
\begin{enumerate}
\item $(\Delta_{g_\EE} +V)\psi_{p,i}=0$ on $D(r(t))$,
\item $\psi_{p,i}(p)=0$,
\item $\frac{\partial}{\partial x^j}\psi_{p,i}(p)=\delta_{ij}$.
\end{enumerate}
Moreover, for $j,k,l=1,2$, there is a $C>0$ independent of $p$ (and depending  on 
$\delta_0$ in Case 2) such that
\begin{enumerate}
\item[4.]  $||\psi_{p,i}-(x^j-p^j)\delta_{ij}||_{C^0_x(D(r(t)))} \le Cr^3$, 
\item[5.]  $||\partial_{x^j}\psi_{p,i}(x)-\delta_{ij}||_{C^0_x(D(r(t)))}\le Cr^2$,
\item[6.]  $||\partial_{x^k}\partial_{x^j}\psi_{p,i}(x)||_{C^0_x(D(r(t)))}\le Cr$,
\item[7.]  $||\partial_{p^j}\psi_{p,i}(x)+\delta_{ij}||_{C^0_x(D(r(t)))}\le Cr^2$,
\item[8.]  $||\partial_{p^k}\partial_{p^j}\psi_{p,i}(x)||_{C^0_x(D(r(t)))}\le Cr$,
\item[9.] $||\partial_{p^l}\partial_{p^k}\partial_{p^j}\psi_{p,i}(x)||_{C^0_x(D(r(t)))}\le C$,
\item[10.]  $||\partial_{p^k}\partial_{x^j}\psi_{p,i}(x)||_{C^0_x(D(r(t)))}\le Cr$,
\item[11.]  $||\partial_{p^l}\partial_{p^k}\partial_{x^j}\psi_{p,i}(x)||_{C^0_x(D(r(t)))}\le C$
\item[12.]  $||\partial_{p^l}\partial_{x^k}\partial_{x^j}\psi_{p,i}(x)||_{C^0_x(D(r(t)))}\le C$,
\end{enumerate}
where $\partial_{p^j}\psi_{p,i}$ denotes differentiation with respect to the $j$-th component of $p$. Lastly, we have
\begin{enumerate}
\item[13.] the area data for $(M,g)$ determines the function $\psi_{p,i}(x)$ on $D(r(t))$.
\end{enumerate}
\end{lemma}

\begin{proof}

Part 1. In the first case of Theorem \ref{global_metric_thm} we are assuming that the metric $g$ is $C^3$-close to Euclidean, hence the norm $||V||_\infty$ is small. Then, by setting $r(t)=1$ in the proof of Part 2 below, we obtain the desired claims.

Part 2. Now, let $(M, g)$, $0< r\le \epsilon_0$, $D(r)$, and $Y(t)$ be as in the second case of Theorem \ref{global_metric_thm}. We seek a function $\psi_{p,i}: D(r)\to \RR$ with the properties 1--13. Without loss of generality, set $i=1$ and write $\psi_p:= \psi_{p,i}$ for simplicity. We also simply write $Y:=Y(t)$, $r:=r(t)$, and $\phi:=\phi(t)$ below.

We will construct $\psi_{p}$ from linear combinations of solutions which are close to either $x^1$, $x^2$, or 1 in $C^2(D(r))$. 

To this end, let $\chi^1$ be a function which solves the boundary value problem 
\begin{align*}
(\Delta_{g_\mathbb{E}} +V)\chi^1&=0&& \text{on }D(r),\\
\chi^1&=x^1 && \text{on }\partial D(r).
\end{align*}

Rescale the coordinates $(x^i)$, $i=1,2$, to $(\tilde x^i)$, $i=1,2$, so that we work over the unit disk $D(1)$: define $f:D(1)\to D(r)$ to be the change of coordinates map $f(\tilde x)=r\tilde x=x$. Set $\tilde\chi^1:= \frac1r\chi^1\circ f$, $\tilde V:= V\circ f$, and $\tilde p=f^{-1}(p)$. Then 
\begin{align}
f^*(\Delta_{g_\mathbb{E}} +V)&= r^{-2}\Delta_{g_\mathbb{E}} +\tilde V,
\end{align}
and we seek solutions $\tilde\chi^1$ to 
\begin{align}
(\Delta_{g_\mathbb{E}} +r^2\tilde V)\tilde\chi^1&=0,&& \text{on }D(1)\label{tildepsi1}\\
\tilde\chi^1&= \tilde x^1 && \text{on }\partial D(1).\label{tildepsi1b}
 \end{align}

Since $(M,g)$ is $(K,\epsilon_0,\delta_0)$-thin and using the definition of $V$ and the bounds
 in Lemma \ref{conf.bounds}, the potential satisfies the estimate 
 $||V||_{W^{1,q}(D(r))}<\frac{\delta_0}{\epsilon_0^2}$ for some small $\delta_0>0$; 
 this combined with the fact $r<\epsilon_0$ gives 
 $||r^2\tilde V||_{W^{1,q}(D(1))}\le \delta_0$. Now, as $\delta_0$ is small, the operator 
 $(\Delta_{g_\mathbb{E}} +r^2\tilde V) : W^{3,q}(D(1))\to W^{1,q}(D(1))$ is invertible, 
 and the norm of the inverse is 
 bounded by a constant depending on $\delta_0$ for the Class 2 case. 
Therefore there exists a unique solution $\tilde\chi^1$ to (\ref{tildepsi1}), (\ref{tildepsi1b}).

Now consider
\begin{align*}
(\Delta_{g_\mathbb{E}} +r^2\tilde V)\left[\tilde\chi^1-\tilde x^1\right]&= -(r^2\tilde V) \tilde x^1&& \text{on }D(1),\\
\tilde\chi^1-\tilde x^1&=0 && \text{on }\partial D(1).
 \end{align*}
Since $r$ is chosen to be small, for $q>2$ we have via Sobolev embedding
\begin{align}
\left|\left|\tilde\chi^1-\tilde x^1\right|\right|_{C^{2}(D(1))}&\le C\left|\left|\tilde\chi^1-\tilde x^1\right|\right|_{W^{3,q}(D(1))}\notag\\
&\le C\left[||(r^2\tilde V) \tilde x^1||_{W^{1,q}(D(1))}\right]\notag\\
&\le C\left[r^2||\tilde V \tilde x^1||_{L^{q}(D(1))}+r^2||\tilde\nabla(\tilde V\tilde x^1)||_{L^{q}(D(1))}\right]\notag\\
&\le Cr^2 \label{good-est0}
 \end{align}
 for $C>0$ depending on $\delta_0$, $q$, and $D(1)$. 
 
Therefore, the function $\tilde\chi^1$ satisfies the estimate
\begin{align}
\left|\left|\tilde\chi^1- \tilde x^1\right|\right|_{C^{2}(D(1))}&\le Cr^2. \label{good-est1}
 \end{align}
 In particular, the estimate (\ref{good-est1}) implies
  \begin{align}
 ||\partial_{\tilde x^1}\tilde\chi^1 - 1||_{C^1(D(1))}&\le Cr^2,\label{1closeto1}\\
  ||\partial_{\tilde x^2}\tilde\chi^1 - 0||_{C^1(D(1))}&\le Cr^2.\label{1closeto0}
 \end{align}

Similar to above, we construct a function $\tilde\chi^2$, which solves the boundary value problem 
\begin{align*}
(\Delta_{g_\mathbb{E}} +r^2\tilde V)\tilde\chi^2&=0&& \text{on }D(1),\\
\tilde\chi^2&= \tilde x^2 && \text{on }\partial D(1),
 \end{align*}
 and satisfies
 \begin{align}
\left|\left|\tilde\chi^2- \tilde x^2\right|\right|_{C^{2}(D(1))}&\le Cr^2. \label{good-est2}
 \end{align}
Further, let $\omega$ solve
\begin{align*}
(\Delta_{g_\mathbb{E}} +r^2\tilde V)\omega&=0&& \text{on }D(1),\\
\omega&=1 && \text{on }\partial D(1).
\end{align*} 
As argued above, the function $\omega$ satisfies
\begin{align}
\left|\left|\omega- 1\right|\right|_{C^{2}(D(1))}&\le Cr^2. \label{good-est3}
 \end{align} 

Now, for $j=1,2$, and $\tilde p\in D(1)$ fixed, the function 
\begin{align*}
\gamma^j(\tilde{p},\tilde{x}) &:= \omega(\tilde p)\tilde\chi^j(\tilde x)-\tilde\chi^j(\tilde p)\omega(\tilde x)
\end{align*}
 solves
\begin{align*}
(\Delta_{g_\mathbb{E}} +r^2\tilde V)\gamma^j &=0&& \text{for }D(1),\\
\gamma^j(\tilde{p},\tilde{x}) &=\omega(\tilde p)\tilde x^j-\tilde \chi^j(\tilde p) && \text{on }\partial D(1),
\end{align*} 
and has the property $\gamma^j(\tilde p)=0$. The function $\gamma^j$ also obeys the following estimates. 
%
From (\ref{good-est1}), (\ref{good-est2}), and (\ref{good-est3}),
\begin{align}
||\gamma^j(\tilde{p},\tilde{x})-(\tilde{x}^j-\tilde{p}^j)||_{C^0_x(D(1))} &= ||\omega(\tilde p)\tilde\chi^j(\tilde x)-\tilde\chi^j(\tilde p)\omega(\tilde x)-(\tilde{x}^j-\tilde{p}^j)||_{C^0_x(D(1))}\notag\\
&\le ||\omega(\tilde p)\tilde\chi^j(\tilde x)-\omega(\tilde p)\tilde{x}^j||_{C^0_x(D(1))} +||\tilde x^j\omega(\tilde p)-\tilde{x}^j||_{C^0_x(D(1))}\notag \\
&\qquad+||\tilde\chi^j(\tilde p)\omega(\tilde x)-\tilde{p}^j\omega(\tilde x)||_{C^0_x(D(1))}+||\tilde p^j\omega(\tilde x)-\tilde{p}^j||_{C^0_x(D(1))}\notag\\
&\le Cr^2,\label{good-est4}
\end{align}
where $C>0$ is a constant independent of $\tilde p$. Similarly, we find for $i,j=1,2$,
\begin{align}
||\partial_{\tilde x^i}\gamma^j(\tilde{p},\tilde{x})-\delta^j_i||_{C^{0}_x(D(1))}&\le Cr^2,\label{gamma-est1}\\
||\partial_{\tilde p^i}\gamma^j(\tilde{p},\tilde{x})+\delta^j_i||_{C^{0}_x(D(1))}&\le Cr^2,\label{gamma-est2}\\
||\partial_{\tilde p^k}\partial_{\tilde p^i}\gamma^j(\tilde{p},\tilde{x})||_{C^{0}_x(D(1))}&\le Cr^2,\label{gamma-est3}\\
||\partial_{\tilde p^k}\partial_{\tilde x^i}\gamma^j(\tilde{p},\tilde{x})||_{C^{0}_x(D(1))}&\le Cr^2,\label{gamma-est4}
 \end{align}
where the constant is independent of $\tilde p$.

Consider $\gamma^1$ and $\gamma^2$ as above. From (\ref{good-est1}), (\ref{good-est2}), and (\ref{good-est3}), 
\begin{align}
||\partial_{\tilde x^1}\gamma^1(\tilde p,\tilde p)\partial_{\tilde x^2}\gamma^2(\tilde p,\tilde p)-\partial_{\tilde x^2}\gamma^1(\tilde p,\tilde p)\partial_{\tilde x^1}\gamma^2(\tilde p,\tilde p)||_{C^0_x(D(1))}\sim 1+Cr^2.
\end{align}
Then, the function
\begin{align*}
\tilde\psi_{\tilde p}(\tilde x) &:= \frac{\partial_{\tilde x^2}\gamma^2(\tilde p,\tilde p)\gamma^1(\tilde p,\tilde x)-\partial_{\tilde x^2}\gamma^1(\tilde p,\tilde p)\gamma^2(\tilde p,\tilde x)}{\partial_{\tilde x^1}\gamma^1(\tilde p,\tilde p)\partial_{\tilde x^2}\gamma^2(\tilde p,\tilde p)-\partial_{\tilde x^2}\gamma^1(\tilde p,\tilde p)\partial_{\tilde x^1}\gamma^2(\tilde p,\tilde p)}
\end{align*}
is well defined and satisfies conditions 2-3. By linearity of the operator $(\Delta_{g_\mathbb{E}} +r^2\tilde V)$, $\tilde\psi_{\tilde p}$ satisfies condition 1:
\begin{align*}
(\Delta_{g_\mathbb{E}} +r^2\tilde V)\tilde\psi_{\tilde p} &=0&& \text{on }D(1),\\
\tilde\psi_{\tilde p} &=\tilde f_{\tilde p} && \text{on }\partial D(1),
\end{align*} 
where the boundary data is explicitly given by
\begin{align*}
\tilde f_{\tilde p}(\tilde x) &:=  \frac{\partial_{\tilde x^2}\gamma^2(\tilde p,\tilde p)[\omega(\tilde p)\tilde x^1 -\tilde\chi^1(\tilde p)]-\partial_{\tilde x^2}\gamma^1(\tilde p,\tilde p)[\omega(\tilde p)\tilde x^2-\tilde\chi^2(\tilde p)]}{\partial_{\tilde x^1}\gamma^1(\tilde p,\tilde p)\partial_{\tilde x^2}\gamma^2(\tilde p,\tilde p)-\partial_{\tilde x^2}\gamma^1(\tilde p,\tilde p)\partial_{\tilde x^1}\gamma^2(\tilde p,\tilde p)}.
\end{align*}
%

Now, we prove preliminary estimates for $\tilde\psi_{\tilde p}$ and its derivatives which will enable us to obtain the desired function $\psi_{p}$ and estimates claimed for $\psi_{p}$. First, we show 
\begin{align}
||\tilde\psi_{\tilde p}-(\tilde{x}^j-\tilde{p}^j)||_{C^0_x(D(1))}&\le Cr^2,\label{c0-est-psi-x}
\end{align}
 for some constant $C>0$ independent of $\tilde p$. Indeed, from the inequalities (\ref{good-est1})--(\ref{good-est4}) and (\ref{gamma-est1})--(\ref{gamma-est4}),

\begin{align}
||\tilde\psi_{\tilde p}(\tilde x)-(\tilde{x}^1-\tilde{p}^1)||_{C^0_x(D(1))} &= \left|\left|\frac{\partial_{\tilde x^2}\gamma^2(\tilde p,\tilde p)\gamma^1(\tilde p,\tilde x)+\partial_{\tilde x^2}\gamma^1(\tilde p,\tilde p)\gamma^2(\tilde p,\tilde x)}{\partial_{\tilde x^1}\gamma^1(\tilde p,\tilde p)\partial_{\tilde x^2}\gamma^2(\tilde p,\tilde p)-\partial_{\tilde x^2}\gamma^1(\tilde p,\tilde p)\partial_{\tilde x^2}\gamma^2(\tilde p,\tilde p)}-(\tilde{x}^j-\tilde{p}^j)\right|\right|_{C^0_x(D(1))}\notag\\[2pt]
&\le \frac{||\partial_{\tilde x^2}\gamma^2(\tilde p,\tilde p)\left[\gamma^1(\tilde p,\tilde x)-(\tilde{x}^1-\tilde{p}^1)\partial_{\tilde x^1}\gamma^1(\tilde p,\tilde p)\right]||_{C^0_x(D(1))}}{|\partial_{\tilde x^1}\gamma^1(\tilde p,\tilde p)\partial_{\tilde x^2}\gamma^2(\tilde p,\tilde p)-\partial_{\tilde x^2}\gamma^1(\tilde p,\tilde p)\partial_{\tilde x^2}\gamma^2(\tilde p,\tilde p)|}\notag\\[2pt]
&\qquad+\frac{||\partial_{\tilde x^2}\gamma^1(\tilde p,\tilde p)\left[\gamma^2(\tilde p,\tilde x)-(\tilde{x}^1-\tilde{p}^1)\partial_{\tilde x^1}\gamma^2(\tilde p,\tilde p)\right]||_{C^0_x(D(1))}}{|\partial_{\tilde x^1}\gamma^1(\tilde p,\tilde p)\partial_{\tilde x^2}\gamma^2(\tilde p,\tilde p)-\partial_{\tilde x^2}\gamma^1(\tilde p,\tilde p)\partial_{\tilde x^2}\gamma^2(\tilde p,\tilde p)|}\notag\\[2pt]
%
%
&\le Cr^2.
\end{align}

We perform similar computations and use the inequalities (\ref{good-est1})--(\ref{good-est4}) and (\ref{gamma-est1})--(\ref{gamma-est4}) to derive an estimate for $\partial_{\tilde x^j}\tilde\psi_{\tilde p}(\tilde x)$, $j=1,2$.
\begin{align}
||\partial_{\tilde x^j}\tilde\psi_{\tilde p}(\tilde x)-\delta^1_j||_{C^0_x(D(1))} &= \left|\left|\frac{\partial_{\tilde x^2}\gamma^2(\tilde p,\tilde p)\partial_{\tilde x^j}\gamma^1(\tilde p,\tilde x)+\partial_{\tilde x^2}\gamma^1(\tilde p,\tilde p)\partial_{\tilde x^j}\gamma^2(\tilde p,\tilde x)}{\partial_{\tilde x^1}\gamma^1(\tilde p,\tilde p)\partial_{\tilde x^2}\gamma^2(\tilde p,\tilde p)-\partial_{\tilde x^2}\gamma^1(\tilde p,\tilde p)\partial_{\tilde x^2}\gamma^2(\tilde p,\tilde p)}-\delta^1_j\right|\right|_{C^0_x(D(1))}\notag\\[2pt]
&\le ||\partial_{\tilde x^2}\gamma^2(\tilde p, \tilde p)||_{C^0_p(D(1))}||\partial_{\tilde x^j}\gamma^1(\tilde p,\tilde x)-\delta^1_j\partial_{\tilde x^1}\gamma^1(\tilde p, \tilde p)||_{C^0_x(D(1))}\notag\\[2pt]
&\qquad + ||\partial_{\tilde x^2}\gamma^1(\tilde p, \tilde p)||_{C^0_p(D(1))}||\partial_{\tilde x^j}\gamma^2(\tilde p,\tilde x)-\delta^1_j\partial_{\tilde x^1}\gamma^2(\tilde p, \tilde p)||_{C^0_x(D(1))}\notag\\[2pt]
%
%
%
%
&\le Cr^2.
\end{align}

Additionally, for $j,k=1,2,$ and using (\ref{gamma-est1})--(\ref{gamma-est4}), we obtain the inequality
\begin{align}
||\partial_{\tilde x^k}\partial_{\tilde x^j}\tilde\psi_{\tilde p}(\tilde x)||_{C^0_x(D(1))} &\le \frac{||\partial_{\tilde x^2}\gamma^2(\tilde p,\tilde p)\partial_{\tilde x^k}\partial_{\tilde x^j}\gamma^1(\tilde p,\tilde x)||_{C^0_x(D(1))}}{|\partial_{\tilde x^1}\gamma^1(\tilde p,\tilde p)\partial_{\tilde x^2}\gamma^2(\tilde p,\tilde p)-\partial_{\tilde x^2}\gamma^1(\tilde p,\tilde p)\partial_{\tilde x^1}\gamma^2(\tilde p,\tilde p)|}\notag\\[2pt]
&\qquad +\frac{||\partial_{\tilde x^2}\gamma^1(\tilde p,\tilde p)\partial_{\tilde x^k}\partial_{\tilde x^j}\gamma^2(\tilde p,\tilde x)||_{C^0_x(D(1))}}{|\partial_{\tilde x^1}\gamma^1(\tilde p,\tilde p)\partial_{\tilde x^2}\gamma^2(\tilde p,\tilde p)-\partial_{\tilde x^2}\gamma^1(\tilde p,\tilde p)\partial_{\tilde x^1}\gamma^2(\tilde p,\tilde p)|}\notag\\[2pt]
&\le Cr^2.
\end{align}
By an analogous computation we find the claimed estimate 
\begin{align}
||\partial_{\tilde p^l}\partial_{\tilde x^k}\partial_{\tilde x^j}\tilde\psi_{\tilde p}(\tilde x)||_{C^0_x(D(1))} &\le Cr.
\end{align}
%

Now we show estimates for the derivatives $\partial_{\tilde p^i}\tilde\psi_{\tilde p}$, $\partial_{\tilde p^j}\partial_{\tilde p^i}\tilde\psi_{\tilde p}$, $\partial_{\tilde p^j}\partial_{\tilde x^i}\tilde\psi_{\tilde p}$, and $\partial_{\tilde p^k}\partial_{\tilde p^j}\partial_{\tilde x^i}\tilde\psi_{\tilde p}$, where $i,j,k=1,2$.

First we show $||\partial_{\tilde p^i}\tilde\psi_{\tilde p}+\delta^1_i||_{C^1_x(D(1))}$. By definition,
\begin{align*}
\partial_{\tilde p^i}\tilde \psi_{\tilde p}(\tilde x) &:=  \frac{\partial_{\tilde p^i}\partial_{\tilde x^2}\gamma^2(\tilde p,\tilde p)\gamma^1(\tilde p,\tilde x)[\partial_{\tilde x^1}\gamma^1(\tilde p,\tilde p)\partial_{\tilde x^2}\gamma^2(\tilde p,\tilde p)-\partial_{\tilde x^2}\gamma^1(\tilde p,\tilde p)\partial_{\tilde x^1}\gamma^2(\tilde p,\tilde p)]}{[\partial_{\tilde x^1}\gamma^1(\tilde p,\tilde p)\partial_{\tilde x^2}\gamma^2(\tilde p,\tilde p)-\partial_{\tilde x^2}\gamma^1(\tilde p,\tilde p)\partial_{\tilde x^1}\gamma^2(\tilde p,\tilde p)]^2}\\[2pt]
&\qquad+\frac{\partial_{\tilde x^2}\gamma^2(\tilde p,\tilde p)\partial_{\tilde p^i}\gamma^1(\tilde p,\tilde x)[\partial_{\tilde x^1}\gamma^1(\tilde p,\tilde p)\partial_{\tilde x^2}\gamma^2(\tilde p,\tilde p)-\partial_{\tilde x^2}\gamma^1(\tilde p,\tilde p)\partial_{\tilde x^1}\gamma^2(\tilde p,\tilde p)]}{[\partial_{\tilde x^1}\gamma^1(\tilde p,\tilde p)\partial_{\tilde x^2}\gamma^2(\tilde p,\tilde p)-\partial_{\tilde x^2}\gamma^1(\tilde p,\tilde p)\partial_{\tilde x^1}\gamma^2(\tilde p,\tilde p)]^2}\\[2pt]
&\qquad-\frac{\partial_{\tilde p^i}\partial_{\tilde x^2}\gamma^1(\tilde p,\tilde p)\gamma^2(\tilde p,\tilde x)[\partial_{\tilde x^1}\gamma^1(\tilde p,\tilde p)\partial_{\tilde x^2}\gamma^2(\tilde p,\tilde p)-\partial_{\tilde x^2}\gamma^1(\tilde p,\tilde p)\partial_{\tilde x^1}\gamma^2(\tilde p,\tilde p)]}{[\partial_{\tilde x^1}\gamma^1(\tilde p,\tilde p)\partial_{\tilde x^2}\gamma^2(\tilde p,\tilde p)-\partial_{\tilde x^2}\gamma^1(\tilde p,\tilde p)\partial_{\tilde x^1}\gamma^2(\tilde p,\tilde p)]^2}\\[2pt]
&\qquad-\frac{\partial_{\tilde x^2}\gamma^1(\tilde p,\tilde p)\partial_{\tilde p^i}\gamma^2(\tilde p,\tilde x)[\partial_{\tilde x^1}\gamma^1(\tilde p,\tilde p)\partial_{\tilde x^2}\gamma^2(\tilde p,\tilde p)-\partial_{\tilde x^2}\gamma^1(\tilde p,\tilde p)\partial_{\tilde x^1}\gamma^2(\tilde p,\tilde p)]}{[\partial_{\tilde x^1}\gamma^1(\tilde p,\tilde p)\partial_{\tilde x^2}\gamma^2(\tilde p,\tilde p)-\partial_{\tilde x^2}\gamma^1(\tilde p,\tilde p)\partial_{\tilde x^1}\gamma^2(\tilde p,\tilde p)]^2}\\[2pt]
&\qquad+\frac{\partial_{\tilde x^2}\gamma^2(\tilde p,\tilde p)\gamma^1(\tilde p,\tilde x)[\partial_{\tilde p^i}\partial_{\tilde x^1}\gamma^1(\tilde p,\tilde p)\partial_{\tilde x^2}\gamma^2(\tilde p,\tilde p)-\partial_{\tilde p^i}\partial_{\tilde x^2}\gamma^1(\tilde p,\tilde p)\partial_{\tilde x^1}\gamma^2(\tilde p,\tilde p)]}{[\partial_{\tilde x^1}\gamma^1(\tilde p,\tilde p)\partial_{\tilde x^2}\gamma^2(\tilde p,\tilde p)-\partial_{\tilde x^2}\gamma^1(\tilde p,\tilde p)\partial_{\tilde x^1}\gamma^2(\tilde p,\tilde p)]^2}\\[2pt]
&\qquad-\frac{\partial_{\tilde x^2}\gamma^1(\tilde p,\tilde p)\gamma^2(\tilde p,\tilde x)[\partial_{\tilde p^i}\partial_{\tilde x^1}\gamma^1(\tilde p,\tilde p)\partial_{\tilde x^2}\gamma^2(\tilde p,\tilde p)-\partial_{\tilde p^i}\partial_{\tilde x^2}\gamma^1(\tilde p,\tilde p)\partial_{\tilde x^1}\gamma^2(\tilde p,\tilde p)]}{[\partial_{\tilde x^1}\gamma^1(\tilde p,\tilde p)\partial_{\tilde x^2}\gamma^2(\tilde p,\tilde p)-\partial_{\tilde x^2}\gamma^1(\tilde p,\tilde p)\partial_{\tilde x^1}\gamma^2(\tilde p,\tilde p)]^2}.
\end{align*}
Thus,
\begin{align*}
||\partial_{\tilde p^i}\tilde \psi_{\tilde p}(\tilde x) +\delta^1_i||_{C^0_x(D(1))}&\le ||\partial_{\tilde p^i}\partial_{\tilde x^2}\gamma^2(\tilde p,\tilde p)\gamma^1(\tilde p,\tilde x)\partial_{\tilde x^1}\gamma^1(\tilde p,\tilde p)\partial_{\tilde x^2}\gamma^2(\tilde p,\tilde p)||_{C^0_x(D(1))}\\
&\qquad+||\partial_{\tilde p^i}\partial_{\tilde x^2}\gamma^2(\tilde p,\tilde p)\gamma^1(\tilde p,\tilde x)\partial_{\tilde x^2}\gamma^1(\tilde p,\tilde p)\partial_{\tilde x^1}\gamma^2(\tilde p,\tilde p)||_{C^0_x(D(1))}\\
&\qquad+||[\partial_{\tilde p^i}\gamma^1(\tilde p,\tilde x)+\delta^1_i\partial_{\tilde x^1}\gamma^1(\tilde p,\tilde p)]\partial_{\tilde x^2}\gamma^2(\tilde p,\tilde p)||_{C^0_x(D(1))}\\
&\qquad+||[\partial_{\tilde p^i}\gamma^1(\tilde p,\tilde x)+\delta^1_i\partial_{\tilde x^2}\gamma^1(\tilde p,\tilde p)]\partial_{\tilde x^1}\gamma^2(\tilde p,\tilde p)||_{C^0_x(D(1))}\\
&\qquad+||\partial_{\tilde p^i}\partial_{\tilde x^2}\gamma^1(\tilde p,\tilde p)\gamma^2(\tilde p,\tilde x)\partial_{\tilde x^1}\gamma^1(\tilde p,\tilde p)\partial_{\tilde x^2}\gamma^2(\tilde p,\tilde p)||_{C^0_x(D(1))}\\
&\qquad+||\partial_{\tilde p^i}\partial_{\tilde x^2}\gamma^1(\tilde p,\tilde p)\gamma^2(\tilde p,\tilde x)\partial_{\tilde x^2}\gamma^1(\tilde p,\tilde p)\partial_{\tilde x^1}\gamma^2(\tilde p,\tilde p)||_{C^0_x(D(1))}\\
&\qquad+||\partial_{\tilde x^2}\gamma^1(\tilde p,\tilde p)\partial_{\tilde p^i}\gamma^2(\tilde p,\tilde x)\partial_{\tilde x^1}\gamma^1(\tilde p,\tilde p)\partial_{\tilde x^2}\gamma^2(\tilde p,\tilde p)||_{C^0_x(D(1))}\\
&\qquad+||\partial_{\tilde x^2}\gamma^1(\tilde p,\tilde p)\partial_{\tilde p^i}\gamma^2(\tilde p,\tilde x)\partial_{\tilde x^2}\gamma^1(\tilde p,\tilde p)\partial_{\tilde x^1}\gamma^2(\tilde p,\tilde p)||_{C^0_x(D(1))}\\
&\qquad+||\partial_{\tilde x^2}\gamma^2(\tilde p,\tilde p)\gamma^1(\tilde p,\tilde x)[\partial_{\tilde p^i}\partial_{\tilde x^1}\gamma^1(\tilde p,\tilde p)\partial_{\tilde x^2}\gamma^2(\tilde p,\tilde p)||_{C^0_x(D(1))}\\
&\qquad+||\partial_{\tilde x^2}\gamma^2(\tilde p,\tilde p)\gamma^1(\tilde p,\tilde x)\partial_{\tilde p^i}\partial_{\tilde x^2}\gamma^1(\tilde p,\tilde p)\partial_{\tilde x^1}\gamma^2(\tilde p,\tilde p)||_{C^0_x(D(1))}\\
&\qquad+||\partial_{\tilde x^2}\gamma^1(\tilde p,\tilde p)\gamma^2(\tilde p,\tilde x)\partial_{\tilde p^i}\partial_{\tilde x^1}\gamma^1(\tilde p,\tilde p)\partial_{\tilde x^2}\gamma^2(\tilde p,\tilde p)||_{C^0_x(D(1))}\\
&\qquad+||\partial_{\tilde x^2}\gamma^1(\tilde p,\tilde p)\gamma^2(\tilde p,\tilde x)\partial_{\tilde p^i}\partial_{\tilde x^2}\gamma^1(\tilde p,\tilde p)\partial_{\tilde x^1}\gamma^2(\tilde p,\tilde p)||_{C^0_x(D(1))}.
\end{align*}
From (\ref{gamma-est1})--(\ref{gamma-est4}), we see
\begin{align}
||\partial_{\tilde p^i}\tilde \psi_{\tilde p}(\tilde x) +\delta^1_i||_{C^0_x(D(1))}&\le Cr^2.
\end{align}

By analogous computations, $\tilde \psi_{\tilde p}(\tilde x)$ satisfies
%
\begin{align}
||\partial_{\tilde p^j}\partial_{\tilde p^i}\tilde \psi_{\tilde p}(\tilde x)||_{C^0_x(D(1))}&\le Cr^2,\quad
||\partial_{\tilde p^j}\partial_{\tilde x^i}\tilde \psi_{\tilde p}(\tilde x)||_{C^0_x(D(1))}\le Cr^2,\\
||\partial_{\tilde x^j}\partial_{\tilde x^i}\tilde \psi_{\tilde p}(\tilde x)||_{C^0_x(D(1))}&\le Cr^2,\quad 
||\partial_{\tilde p^k}\partial_{\tilde x^j}\partial_{\tilde x^i}\tilde \psi_{\tilde p}(\tilde x)||_{C^0_x(D(1))}\le Cr^2, \\
||\partial_{\tilde p^k}\partial_{\tilde p^j}\partial_{\tilde x^i}\tilde \psi_{\tilde p}(\tilde x)||_{C^0_x(D(1))}&\le Cr^2, \quad
||\partial_{\tilde p^k}\partial_{\tilde p^j}\partial_{\tilde p^i}\tilde \psi_{\tilde p}(\tilde x)||_{C^0_x(D(1))}\le Cr^2,\label{tilde-psi-last-est}
\end{align}
for $i.j,k=1,2$.

%

Having constructed $\tilde\psi_{\tilde p}$ which satisfies the estimates, we rescale $\tilde \psi_{\tilde p}$ to achieve the desired function $\psi_p$ on the disk $D(r)$. Define $$\psi_p(x):= r\tilde\psi_{\tilde p}\left(\frac xr\right).$$ We claim that this is the function which satisfies conditions 1-12.

By construction, $\psi_p$ satisfies conditions 1, 2, and 3. We now prove the estimates 4--11. Note that the coordinates change as 
$$x=r\tilde x,\quad p=r\tilde p.$$
Hence, the derivatives change as
$$\partial_{\tilde x^j}=r\partial_{x^j},\quad \partial_{\tilde p^j}=r\partial_{p^j}.$$
where $j=1,2$. In particular we have
\begin{align*}
&\partial_{\tilde x^j}\tilde \psi_{\tilde p}(\tilde x)=\partial_{x^j}\psi_p(x),\quad
&&\partial_{\tilde x^k}\partial_{\tilde x^j}\tilde \psi_{\tilde p}(\tilde x) =r \partial_{x^k}\partial_{x^j}\psi_p(x),\\
&\partial_{\tilde p^k}\partial_{\tilde x^j}\tilde \psi_{\tilde p}(\tilde x) = r\partial_{p^k}\partial_{x^j}\psi_p(x),\quad 
&&\partial_{\tilde p^l}\partial_{\tilde x^k}\partial_{\tilde x^j}\tilde \psi_{\tilde p}(\tilde x) = r^2\partial_{p^l}\partial_{x^k}\partial_{x^j}\psi_p(x),\\
&\partial_{\tilde p^l}\partial_{\tilde p^k}\partial_{\tilde x^j}\tilde \psi_{\tilde p}(\tilde x) = r^2\partial_{p^l}\partial_{p^k}\partial_{x^j}\psi_p(x),\quad 
&&\partial_{\tilde p^l}\partial_{\tilde p^k}\partial_{\tilde p^j}\tilde \psi_{\tilde p}(\tilde x) = r^2\partial_{p^l}\partial_{p^k}\partial_{x^j}\psi_p(x),\\
&\partial_{\tilde p^j}\tilde \psi_{\tilde p}(\tilde x)=\partial_{p^j}\psi_p(x),\quad
&&\partial_{\tilde p^k}\partial_{\tilde p^j}\tilde \psi_{\tilde p}(\tilde x) = r\partial_{p^k}\partial_{p^j}\psi_p(x),
\end{align*}
for $j,k,l=1,2$. 
By the change of coordinates and the estimates (\ref{c0-est-psi-x}) -- (\ref{tilde-psi-last-est}) for $\tilde\psi_{\tilde p}$, we have $\psi_{p}$ obeys the estimates 4--12.

Now for the last condition 13: By Proposition \ref{area_to_DN}, the area data for $(M,g)$ determines $\psi_p$. 

Thus, $\psi_p$ is the desired function which satisfies conditions 1--13.

\end{proof}
\bigskip

%
%

Now, from the above lemma, for $i=1,2$, and each fixed $t\in(-1,1)$, and $p\in Y_1(t)$, there exists families of foliations $Y_{i,1}(s,t)$ given by embeddings
\begin{align*}
&h_{i,1}(\cdot,\cdot,t):[0,S]\times D(r(t))\to M\subset \MM\\
&\left.(h_{i,1})_*\left(\frac{\partial}{\partial s}\right)\right|_{s=0} = \psi_{p,i,1}\vec{n_1},
\end{align*}
for which $\psi_{p,i,1}:\RR^2\to\RR$ has the properties defined in Lemma \ref{psi-lem-thin} on the disc $D(r(t))\subset \RR^2$; in particular,
$$\Delta_{g_\EE}\psi_{p,i,1}+e^{2\phi_1}\left(\text{Ric}_{g_1}(\vec{n},\vec{n})+||A||^2_{g_1}\right)\psi_{p,i,1}=0$$
on $D(r(t))$, and  
\begin{align*}
\psi_{p,i,1}(f_1(p))&=0, \\
\nabla\psi_{p,i,1}(f_1(p)) &=\frac{\partial}{\partial x^i}.
\end{align*}

The induced variation of the coordinate $x^3$ is written in Taylor expanded form as $$x^3_i(s)=x^3+s\dot{x}^3_i+\mathcal{O}(s^2).$$
As shown by Proposition \ref{inf_sep_prop}, the knowledge of the areas of $Y_{i,1}(s,t):=h_{i,1}(D(r(t)),s,t)$ determines the functions 
\begin{align*}
||\nabla x^3_i(s)||_{g_1}^2(p)
\end{align*}
on $\RR^2$, in the coordinates $(x^\alpha)$, for all $s\in[0,S]$, $i=1,2$.

Linearizing in $s$, by Lemma \ref{sep_lin_lem} we obtain a nonlinear, non-local, coupled system of equations for $g^{k3}_1$, $k=1,2$ of the form
\begin{align}
\left.\frac{d}{ds}||\nabla x^3_1(s)||_{g_1}(p)\right|_{s=0}
&=g^{3\alpha}_1\partial_\alpha\psi_{p,1,1}(x^1,x^2)+\partial_\alpha||\nabla x^3||_g(p)\dot{x}^\alpha_1,\label{evoleq11}\\
\left.\frac{d}{ds}||\nabla x^3_2(s)||_{g_1}(p)\right|_{s=0}
&=g^{3\alpha}_1\partial_\alpha\psi_{p,2,1}(x^1,x^2)+\partial_\alpha||\nabla x^3||_g(p)\dot{x}^\alpha_2,\label{evoleq12}
\end{align}
where the first order change in conformal coordinates $\dot{x}^\alpha_i$, $\alpha =1,2,3$, depends on $p$, $\psi_{p,i,1}$, $\delta g$ and the first and second derivatives of $\delta g$, as shown in Lemma \ref{sep_lin_dotx}. By Proposition \ref{inf_sep_prop}, our area information determines the functions $||\nabla x^3_1(s)||_{g_1}(p)$ and $||\nabla x^3_2(s)||_{g_1}(p)$ in the coordinates $(x^\alpha)$.

By the very same argument as above, if we consider instead the foliation $Y_2(t)$ of $(M,g_2)$, we may obtain a nonlinear system of equations for $g^{k3}_2$, $k=1,2$ of the form
\begin{align}
\left.\frac{d}{ds}||\nabla y^3_1(s)||_{g_2}(p)\right|_{s=0}
&=g^{3\alpha}_2\partial_\alpha\psi_{p,1,2}(x^1,x^2)+\partial_\alpha||\nabla x^3||_g(p)\dot{y}^\alpha_1\circ F\circ\Phi_1,\label{evoleq_21}\\
\left.\frac{d}{ds}||\nabla y^3_2(s)||_{g_2}(p)\right|_{s=0}
&=g^{3\alpha}_2\partial_\alpha\psi_{p,2,2}(x^1,x^2)+\partial_\alpha||\nabla x^3||_g(p)\dot{y}^\alpha_2\circ F\circ\Phi_1,\label{evoleq22}
\end{align}
where $\dot{y}^\alpha_i$ depends on $p$, $\psi_{p,i,2}$, $\delta g$ and the first and second derivatives of $\delta g$, as as shown in Lemma \ref{sep_lin_dotx}, and the functions $\psi_{p,i,2}\in H^2(\RR^2)$, $i=1,2$ are solutions of the Jacobi equation 
$$\Delta_{g_\EE}\psi_{p,i,2}+e^{2\phi_2}\left[\left(\text{Ric}_{g_2}(\vec{n}_2,\vec{n}_2)+||A||^2_{g_2}\right)\circ F\circ\Phi_1\right]\psi_{p,i,2}=0$$
on $\RR^2$, which additionally satisfy the conditions of Lemma \ref{psi-lem-thin} on $D(r)$. Again, employing Proposition \ref{inf_sep_prop}, our area information determines the functions $||\nabla y^3_1(s)||_{g_2}(p)$ and $||\nabla y^3_2(s)||_{g_2}(p)$ in the coordinates $(x^\alpha)$.

\begin{lemma}
In the coordinates $(x^\alpha)$, $\psi_{p,i,1}=\psi_{p,i,2}$ for $i=1,2$ on $\RR^2$.
\end{lemma}

\begin{proof}
Fix $x^3=y^3=t$, and consider the leaves $\YY_1(t)$ and $\YY_2(t)$. Since the foliations $\YY_1(t)$ and $\YY_2(t)$ agree outside the disc $D(r(t))\times\{t\}\equiv D(r(t))\subset \RR^2$, so do the functions $\psi_{p,i,1}$ and $\psi_{p,i,2}$. As shown by Proposition \ref{n-dim_area_to_DN}, our area data determines the Dirichlet-to-Neumann maps associated to the operators 
\begin{align*}
\JJ_1&:= \Delta_{g_{1,t}}+\left(\text{Ric}_{g_1}(\vec{n},\vec{n})+||A||^2_{g_1}\right)\\
\JJ_2&:= \Delta_{g_{2,t}}+\left(\text{Ric}_{g_2}(\vec{n},\vec{n})+||A||^2_{g_2}\right)
\end{align*}
on $D(r(t))$, in the coordinates $(x^\alpha)$ and $(y^\alpha)$ respectively. Here $g_{k,t}$ denotes the metric induced on $Y_k(t)$ by $g_k$, $k=1,2$. Via the map $F$, we can express both the operators in the coordinate system $(x^\alpha)$ as
\begin{align}
\JJ_1&:= \Delta_{g_\EE}+e^{2\phi_1}\left(\text{Ric}_{g_1}(\vec{n},\vec{n})+||A||^2_{g_1}\right)\label{opj1}\\
\JJ_2&:= \Delta_{g_\EE}+e^{2\phi_2}\left(\text{Ric}_{g_2}(\vec{n},\vec{n})+||A||^2_{g_2}\right)\circ F\circ\Phi_1.\label{opj2}
\end{align}

From Proposition \ref{area_to_DN}, the Dirichlet-to-Neumann maps to these operators is also determined as expressed in the coordinates $(x^\alpha)$. By construction the foliation $Y_1(t)$ agrees with the foliation $Y_2(t)$ on the boundary of $M$. By hypothesis, for each fixed $t$ the area of $Y_1(t)$ as measured by $g_1$ is equal to the area of $Y_2(t)$ as measured by $g_2$. Thus, Proposition \ref{n-dim_area_to_DN}, the Dirichlet-to-Neumann maps associated to the operators (\ref{opj1}) and (\ref{opj2}) agree as maps $H^{\frac12}(\partial D(r(t)))\to H^{-\frac12}(\partial D(r(t)))$. By the linear result in \cite{IN}, the potential functions 
\begin{align*}
&e^{2\phi_1}\left(\text{Ric}_{g_1}(\vec{n},\vec{n})+||A||^2_{g_1}\right)\\
&e^{2\phi_2}\left(\text{Ric}_{g_2}(\vec{n},\vec{n})+||A||^2_{g_2}\right)\circ F\circ\Phi_1
\end{align*}
agree on $D(r(t))$, as written in the coordinates $(x^\alpha)$.

So in the coordinates $(x^\alpha)$, $\psi_{p,i,1}=\psi_{p,i,2}$ on $D(r(t))$ for $i=1,2$.
\end{proof}
With the above lemma in hand, we simplify notation a bit and write $$\psi_{p,i}:=\psi_{p,i,1}=\psi_{p,i,2}$$
for $i=1,2$.

Now in the above setting, we may obtain pseudodifferential equations which relate the differences of the metric components $\delta\phi$, $\delta g^{31}$, and $\delta g^{32}$:

\begin{lemma}
For each $p\in M$, the unknown differences $\delta\phi$, $\delta g^{31}$, $\delta g^{32}$ satisfy the following system of equations:
\begin{align}
0&=\delta g^{3k}(p)\partial_k\psi_{p,1}(x(p))+\partial_k||\nabla x^3||_{g_1}(p)\delta\dot x^k_1(p)\label{dgEQ1}\\
0&=\delta g^{3k}(p)\partial_k\psi_{p,2}(x(p))+\partial_k||\nabla x^3||_{g_1}(p)\delta\dot x^k_2(p)\label{dgEQ2}\\
    0 &= g^{k3}_1\partial_k(\delta\phi)+g^{33}_1\partial_3(\delta\phi)+\delta g^{k3}\left(\partial_k\phi_2-\frac{1}{2}\partial_k\log(g^{33}_1)\right)+\frac{1}{2}\partial_k(\delta g^{3k}).\label{meancurve3}
\end{align}
Furthermore, by construction $\delta\phi=0$, and $\delta g^{3k}=0$ for $k=1,2$ on $\MM\setminus M\cup \partial M$.
\end{lemma}

\begin{proof}
These equations follow from the previous lemma and from taking the difference of the the equations on the leaves $Y_1(t)$ and $Y_2(t)$, $t\in(-1,1)$, given in Lemma \ref{sep_lin_lem} and Proposition \ref{phi_lem}.

\end{proof}


%
\subsubsection{Expressing $\delta g^{3k}$ via pseudodifferential operators:}
Our next goal is to solve equations (\ref{dgEQ1}) and (\ref{dgEQ2}) by expressing $\delta g^{3k}$, $k=1,2$, as a linear combination of pseudodifferential operators acting on $\delta\phi$ and $\partial_3\delta\phi$. To do this, we use the fact that $(M,g_i)$, $i=1,2$ are either $C^3$-close to Euclidean or $(K,\epsilon_0,\delta_0)$-thin.  Furthermore, the assumptions of either $C^3$-close to Euclidean or $(K,\epsilon_0,\delta_0)$-thin allow us to obtain estimates for the pseudodifferential operators acting on $\delta\phi$ and $\partial_3\delta\phi$ which appear in the expressions for $\delta g^{3k}$, $k=1,2$.  From these estimates, we show equation (\ref{meancurve3}) can be written as a hyperbolic pseudodifferential operator acting on $\delta\phi$. A standard energy argument then gives uniqueness for the conformal factors, which implies uniqueness for the metrics $g_1$ and $g_2$.

\begin{prop}[$\delta g^{31}$ and $\delta g^{31}$ are $\Psi$DOs]\label{psido_lem}
At each $p\in M$, the above differences $\delta g^{31}$ and $\delta g^{32}$ can be expressed as
\begin{align}
\delta g^{31}(p) &=: P^1_{-1}(\delta\phi,p)+Q^1_{-2}(\partial_3\delta\phi,p)\label{psido01}\\
\delta g^{32}(p) &=: P^2_{-1}(\delta\phi,p)+Q^2_{-2}(\partial_3\delta\phi,p),\label{psido02}
\end{align}
where $P^k_{-1}:L^2(D(r(t)))\to H^{1}(D(r(t)))$, $Q^k_{-2}:L^2(D(r(t)))\to H^{2}(D(r(t)))$, $k=1,2$, $s>0$, are respectively order $-1$ and $-2$ pseudodifferential operators in the tangential directions $\partial_k$, $k=1,2$. 

Further, we have the estimates
\begin{align} 
|| P^k_{-1}(\delta\phi)||_{H^1}&\le C\epsilon_0||\delta\phi||_{L^2}\label{p-est1}\\
|| Q^k_{-2}(\partial_3\delta\phi)||_{H^2}&\le C\epsilon_0||\partial_3\delta\phi||_{L^2}\label{q-est1}\\
|| \partial_j Q^k_{-2}(\partial_3\delta\phi)||_{L^2}&\le C\epsilon_0||\partial_3\delta\phi||_{L^2}\label{dq-est1}
\end{align}
in the first case of Theorem \ref{global_metric_thm}, and 
\begin{align}
|| P^k_{-1}(\delta\phi)||_{H^1}&\le CK||\delta\phi||_{L^2}\label{p-est}\\
|| Q^k_{-2}(\partial_3\delta\phi)||_{H^2}&\le CK||\partial_3\delta\phi||_{L^2},\label{q-est}\\
|| \partial_j Q^k_{-2}(\partial_3\delta\phi)||_{L^2}&\le CK\epsilon_0||\partial_3\delta\phi||_{L^2}\label{dq-est}
\end{align}
in the second case of Theorem \ref{global_metric_thm}.
%
In both settings, $C>0$ is a uniform constant independent of $r(t)$ and $\epsilon_0$ (but depending on $\delta_0$ in Definition \ref{thin-defn} in the second case of Theorem \ref{global_metric_thm}), and $j,k\in\{1,2\}$.

\end{prop}

\begin{proof}
First we derive the operators $P^k_{-1}$, $Q^k_{-2}$. Let $w\in M$, and set $\Delta_{g_\mathbb{E}} = \frac{\partial^2}{\partial (w^1)^2}+ \frac{\partial^2}{\partial (w^2)^2}$. From Lemma \ref{sep_lin_dotx}, for each of our choices $\psi_{p,i}$, $i=1,2$, and for each $t$, the 
first order change in the conformal coordinates $(x^1,x^2)$ on $Y_1(t)$ is given schematically as 

\begin{align}
\Delta_{g_\EE}\dot{x}^k_i&=\psi_{p,i} A^{ijk}_{\alpha}\partial_i\partial_jg^{3\alpha}_1+\psi_{p,i} B_{\alpha}^{mk}\partial_m\partial_\alpha\phi_1+\psi_{p,i} C^{mjk}_\alpha\partial_m\phi_1\partial_jg^{3\alpha}_1\notag\\
&\qquad+\psi_{p,i} D^{mjk}_{\alpha \beta}\partial_mg^{3\alpha}_1\partial_jg^{3\beta}_1+\psi_{p,i} F^{mk\alpha }\partial_m\phi_1\partial_\alpha\phi_1\notag\\
&\qquad+(\psi_{p,i} H^{mk}_\alpha+\nabla_j\psi_{p,i}I^{mjk}_{\alpha})\partial_m g^{3\alpha}_1+(\psi_{p,i}J^{k\alpha}_1+\nabla_m\psi_{p}J^{mk\alpha}_2)\partial_\alpha \phi_1,\notag\\
&=:\mathcal{F}^k(g^{13}_1,g^{23}_1,\phi_1,\psi_{p,i}),\label{dox-again}
\end{align}
with an analogous equation for the conformal coordinates $(y^1,y^2)$ on $Y_2(t)$, with respect to the metric $g_2$. Here the functions $A^{mjk}_{\alpha},\dots,J^{mk\alpha}_2$ appearing in (\ref{dox-again}) are polynomials in the components of $g_1$ and $g_1^{-1}$ and the function $e^{\phi_1}$ and $e^{-\phi_1}$. Here $\partial_i:=\partial_{w^i}$. The indices take values $\alpha,\beta\in\{1,2,3\}$ and $i,j,k,l,m\in\{1,2\}$.

The differences in the conformal coordinate functions $\delta\dot x^k(w):x^k-y^k$, $k=1,2$, satisfy on the the disc $D(r)$ $(r=r(t))$ an equation of the form
\begin{align*}
\Delta_{g_\mathbb{E}}\delta\dot x^k_i(w) &= \delta\mathcal{F}^k(w,\delta\phi,\delta g, g_1,g_2,\psi_{p,i})
\end{align*}
where the differential operator $\delta\mathcal{F}^k$ is written schematically as 
\begin{align}
\delta\mathcal{F}^k(w)&=\psi_{p,i}\bar{A}^{jkl}_{m}\partial_l\partial_j\delta g^{3m}(w)+\psi_{p,i}\bar{B}^{jk\alpha}\partial_j\partial_\alpha\delta\phi(w)+(\psi_{p,i}\bar{C}^{k\alpha}_1+\partial_j\psi_{p,i}\bar{C}^{jk\alpha}_2)(w)\partial_\alpha\delta\phi(w)\notag\\
&\qquad +(\psi_{p,i}\bar{D}_1+\partial_j\psi_{p,i}\bar{D}^j_2)\delta\phi+(\psi_{p,i}\bar{E}^{jk}_{1m}+\partial_l\psi_{p,i}{\bar{E}}^{jkl}_{2m})(w)\partial_j \delta g^{3m}(w)\notag\\
&\qquad +(\psi_{p,i}\bar{F}^{k}_{1m}+\partial_l\psi_{p,i}{\bar{F}}^{kl}_{2m})(w)\delta g^{3m}(w),\label{delta-mathcalF}
\end{align}
 for functions $\bar{A}^{jkl}_{m}, \dots,{\bar{F}}^{kl}_{2m}$, which depend on $w\in D(r)$ and are polynomials in the unknown metric coefficients $g^{13}_1,g^{23}_1, e^{\phi_1}, e^{-\phi_1}$ and $g^{13}_2,g^{23}_2, e^{\phi_2}, e^{-\phi_2}$ and their first and second derivatives at $w$ as follows:
\begin{align*}
&\bar{A}^{jkl}_{m}, \bar{B}^{jk\alpha}\text{ contain no derivatives of the metric coefficients,}\\
&\bar{C}^{k\alpha}_1,\bar{C}^{jk\alpha}_2,\bar{E}^{jk}_{1m},{\bar{E}}^{jkl}_{2m}\text{ contain up to first derivatives, and}\\
&\bar{D}_1,\bar{D}^j_2,\bar{F}^{k}_{1m},{\bar{F}}^{kl}_{2m}\text{ contain up to second derivatives.}
\end{align*}
Moreover, the functions $\bar{A}^{jkl}_{m}, \dots,{\bar{F}}^{kl}_{2m}$, are linear in the above derivatives of the metric components and are bounded in $L^\infty(D(r))$ as follows. In the case where $g_1$ and $g_2$ are $C^3$-close to Euclidean,
%
\begin{align}
||\bar{A}^{jkl}_{m}||_{L^\infty}, ||\bar{B}^{jk\alpha}||_{L^\infty} &\le C,\label{F.function.est11}\\
||\partial \bar{A}^{jkl}_{m}||_{L^\infty}, ||\partial\bar{B}^{jk\alpha}||_{L^\infty}&\le C\epsilon_0,\\
||\bar{C}^{k\alpha}_1||_{L^\infty}, \dots, ||{\bar{E}}^{jkl}_{2m}||_{L^\infty}&\le C\epsilon_0,
\end{align}
\begin{align}
||\partial\partial \bar{A}^{jkl}_{m}||_{L^\infty}, ||\partial\partial\bar{B}^{jk\alpha}||_{L^\infty}&\le C\epsilon_0,\\
||\partial\bar{C}^{k\alpha}_1||_{L^\infty}, \dots, ||\partial{\bar{E}}^{jkl}_{2m}||_{L^\infty}&\le C\epsilon_0,\\
||\bar{D}_1||_{L^\infty},\dots, ||{\bar{F}}^{kl}_{2m}||_{L^\infty}&\le C\epsilon_0
\end{align}
\begin{align}
 ||\partial\partial\partial\bar{A}^{jkl}_{m}||_{L^\infty}, ||\partial\partial\partial\bar{B}^{jk\alpha}||_{L^\infty} &\le C\epsilon_0,\\
 ||\partial\partial\bar{C}^{k\alpha}_1||_{L^\infty}, \dots, ||\partial\partial{\bar{E}}^{jkl}_{2m}||_{L^\infty}&\le C\epsilon_0,\\
  ||\partial\bar{D}_1||_{L^\infty},\dots, ||\partial{\bar{F}}^{kl}_{2m}||_{L^\infty}&\le C\epsilon_0,\label{F.function.est21}
\end{align}
and
\begin{align}
 ||\partial\partial\bar{D}_1||_{L^\infty},\dots, ||\partial\partial{\bar{F}}^{kl}_{2m}||_{L^\infty}&\le C\label{F.function.est31}
\end{align}
for a uniform constant $C>0$ independent of $\epsilon_0$.

In the case where $g_1$ and $g_2$ are $(K,\epsilon_0,\delta_0)$-thin, we have the estimates
\begin{align}
||\bar{A}^{jkl}_{m}||_{L^\infty}, ||\bar{B}^{jk\alpha}||_{L^\infty} &\le C,\label{F.function.est12}\\
||\partial \bar{A}^{jkl}_{m}||_{L^\infty}, ||\partial\bar{B}^{jk\alpha}||_{L^\infty}&\le C\epsilon_0^{-1},\\
||\bar{C}^{k\alpha}_1||_{L^\infty}, \dots, ||{\bar{E}}^{jkl}_{2m}||_{L^\infty}&\le C\epsilon_0^{-1},
\end{align}
\begin{align}
||\partial\partial \bar{A}^{jkl}_{m}||_{L^\infty}, ||\partial\partial\bar{B}^{jk\alpha}||_{L^\infty}&\le C\epsilon_0^{-2},\\
||\partial\bar{C}^{k\alpha}_1||_{L^\infty}, \dots, ||\partial{\bar{E}}^{jkl}_{2m}||_{L^\infty}&\le C\epsilon_0^{-2},\\
||\bar{D}_1||_{L^\infty},\dots, ||{\bar{F}}^{kl}_{2m}||_{L^\infty}&\le C\epsilon_0^{-2}
\end{align}
\begin{align}
 ||\partial\partial\partial\bar{A}^{jkl}_{m}||_{L^\infty}, ||\partial\partial\partial\bar{B}^{jk\alpha}||_{L^\infty} &\le C\epsilon_0^{-3},\\
 ||\partial\partial\bar{C}^{k\alpha}_1||_{L^\infty}, \dots, ||\partial\partial{\bar{E}}^{jkl}_{2m}||_{L^\infty}&\le C\epsilon_0^{-3},\\
  ||\partial\bar{D}_1||_{L^\infty},\dots, ||\partial{\bar{F}}^{kl}_{2m}||_{L^\infty}&\le C\epsilon_0^{-3},
\end{align}
and
\begin{align}
 ||\partial\partial\bar{D}_1||_{L^\infty},\dots, ||\partial\partial{\bar{F}}^{kl}_{2m}||_{L^\infty}&\le C\epsilon_0^{-4}\label{F.function.est32}
\end{align}
These estimates follow from Lemma \ref{conf.bounds} and the estimates for $g_1$, $g_2$, and their derivatives in the definition of either $C^3$-close to Euclidean or $(K,\epsilon_0,\delta_0)$-thin, which continue to hold in the coordinates $(x^\alpha)$, $\alpha\in\{1,2,3\}$ (see Remark \ref{est.still.hold.note}). 

Now, given (\ref{dox-again}) and since $\delta \dot{x}_i^k=0$ on $\partial M$, at a point $p\in M$ we have the expression 
\begin{align}
\delta\dot x^k_i(p) &= \int_{D(r)} G(p,w)\delta\mathcal{F}^k(w)\, dw,\label{dotx_equation}
\end{align}
where $G(p,w)$ is the Dirichlet Green's function on the disc $D(r)$:
\begin{align*}
\Delta_{g_\EE}G(p,w)&= \delta(p-w)&&\text{for }x(p)\in D(r)\\
G(p,w)&=0 &&\text{for }x(p)\in \partial D(r).
\end{align*}

Since $\delta\mathcal{F}^k$ vanishes on $\partial D(r)$, using (\ref{delta-mathcalF}) and integrating by parts, (\ref{dotx_equation}) becomes 
\begin{align*}
\delta\dot{x}^k_i(p)&=\int_{D(r)} G(p,w)\delta\mathcal{F}^k(w)\, dw \\
&\overset{IBP}{=} \mathcal{K}_{i,1}^k(\delta g^{31})(p)+\mathcal{K}_{i,2}^k(\delta g^{32})(p)+ \mathcal{L}_{i,1}^k(\delta\phi)(p) +\mathcal{L}_{i,2}^k(\partial_3\delta\phi)(p),
%
\end{align*}
where
\begin{align*}
\mathcal{K}_{i,j}^k(f)(p)&:=\int_{D(r)} K_{i,j}^k(p,w)f(w)\,dw,\\
\mathcal{L}_{i,j}^k(f)(p)&:=\int_{D(r)} L_{i,j}^k(p,w)f(w)\,dw
\end{align*}
are integral operators with kernels defined as
\begin{align}
K_{i,1}^k(p,w)&:= \partial_l\partial_j[G(p,w)(\psi_{p,i}\bar{A}^{jkl}_{1})(w)]-\partial_j[G(p,w)(\psi_{p,i}\bar{E}^{jk}_{11}+\partial_l\psi_{p,i}{\bar{E}}^{jkl}_{21})(w)]\notag\\
&\qquad+G(p,w)(\psi_{p,i}\bar{F}^{k}_{11}+\partial_l\psi_{p,i}{\bar{F}}^{kl}_{21}),\label{K-kernel-1}\\
K_{i,2}^k(p,w)&:= \partial_l\partial_j[G(p,w)(\psi_{p,i}\bar{A}^{jkl}_{2})(w)]-\partial_j[G(p,w)(\psi_{p,i}\bar{E}^{jk}_{12}+\partial_l\psi_{p,i}{\bar{E}}^{jkl}_{22})(w)]\notag\\
&\qquad+G(p,w)(\psi_{p,i}\bar{F}^{k}_{12}+\partial_l\psi_{p,i}{\bar{F}}^{kl}_{22}),\label{K-kernel-2}\\
L_{i,1}^k(p,w)&:=\partial_j\partial_l[G(p,w)(\psi_{p,i}\bar{B}^{jkl})(w)]-\partial_l[G(p,w)(\psi_{p,i}\bar{C}^{kl}_1+\partial_j\psi_{p,i}\bar{C}^{jkl}_2)(w)]\notag\\
&\qquad +G(p,w)(\psi_{p,i}\bar{D}_1+\partial_j\psi_{p,i}\bar{D}^j_2)(w),\label{L-kernel-1}\\
L_{i,2}^k(p,w)&:=\partial_j[G(p,w)(\psi_{p,i}\bar{B}^{jk3})(w)]+G(p,w)(\psi_{p,i}\bar{C}^{k3}_1+\partial_j\psi_{p,i}\bar{C}^{jk3}_2)(w),\label{L-kernel-2}
\end{align}
which have singularities of order $-1$, $-1$,$-1$, and $-2$ respectively. Above we denote $\partial_i:=\partial_{w^i}$, and the indices take values $\alpha,\beta\in\{1,2,3\}$ and $i,j,k,l,m\in\{1,2\}$.


To solve for $\delta g^{3k}$, $k=1,2$ in each of the settings where the metrics area either $C^3$-close to Euclidean or $(K,\epsilon_0,\delta_0)$-thin, we require that the operators  $\mathcal{K}_{i,j}^k$, and $\mathcal{L}_{i,1}^k$, $i,j=1,2$ be bounded in the relevant spaces
\begin{align}
\mathcal{K}_{i,j}^k&:L^2(D(r))\to H^1(D(r)), && ||\mathcal{K}_{i,j}^k(f)||_{H^1}\le C||f||_{L^2},\label{k-bounds-1}\\
\mathcal{K}_{i,j}^k&:H^2(D(r))\to H^2(D(r)), && ||\mathcal{K}_{i,j}^k(f)||_{H^2}\le C||f||_{H^2},\label{k-bounds-2}\\
\mathcal{L}_{i,1}^k&:L^2(D(r))\to H^1(D(r)), && ||\mathcal{L}_{i,1}^k(f)||_{H^1}\le C||f||_{L^2},\label{l-bounds-h1}\\
\mathcal{L}_{i,2}^k&:L^2(D(r))\to H^2(D(r)), && ||\mathcal{L}_{i,2}^k(f)||_{H^2}\le C||f||_{L^2}\label{l-bounds-1},
\end{align}
where $C>0$ is a constant whose dependence on the parameters of the problem will be discussed below.

To see these bounds hold, first consider the following integrand appearing in $\mathcal{K}_{i,j}^k$: $$\partial_l\partial_j[G(p,w)(\psi_{p,i}\bar{A}^{jkl}_{m})(w)].$$ Expanding,
\begin{align*}
\partial_l\partial_j[G(p,w)(\psi_{p,i}\bar{A}^{jkl}_{m})(w)]&=\partial_l\partial_jG(p,w)\psi_{p,i}(w)\bar{A}^{jkl}_{m}(w)+\partial_lG(p,w)\partial_j\psi_{p,i}(w)\bar{A}^{jkl}_{m}(w)\\[2pt]
&\qquad +\partial_jG(p,w)\partial_l\psi_{p,i}(w)\bar{A}^{jkl}_{m}(w)+\partial_lG(p,w)\psi_{p,i}(w)\partial_j\bar{A}^{jkl}_{m}(w)\\[2pt]
&\qquad +\partial_jG(p,w)\psi_{p,i}(w)\partial_l\bar{A}^{jkl}_{m}(w)+G(p,w)\partial_j\psi_{p,i}(w)\partial_l\bar{A}^{jkl}_{m}(w)\\[2pt]
&\qquad+G(p,w)\partial_l\psi_{p,i}(w)\partial_j\bar{A}^{jkl}_{m}(w)+G(p,w)\partial_l\partial_j\psi_{p,i}(w)\bar{A}^{jkl}_{m}(w)\\[2pt]
&\qquad +G(p,w)\psi_{p,i}(w)\partial_l\partial_j\bar{A}^{jkl}_{m}(w),
\end{align*}
where above $\partial_j:=\partial_{w^j}$. Notice that the terms above are schematically of the form of a function $\bar{A}$ (which depends on the metric coefficients of $g_1$ and $g_2$ and their first and second derivatives) multiplied by either $\partial_{w^k}\partial_{w^j}G(p,w)\psi_{p,i}(w)$, $\partial_{w^k}G(p,w)\partial_{w^j}\psi_{p,i}(w)$, or $G(p,w)\partial_{w^k}\partial_{w^j}\psi_{p,i}(w)$.  Also, observe that since the Green's function on the disc is symmetric, that is, $G(p,w)=G(w,p)$, the derivatives in $w$ and $p$ satisfy $$|\partial_{w^j}G(p,w)|=|\partial_{p^j}G(p,w)|\text{ and } |\partial_{w^k}\partial_{w^j}G(p,w)|=|\partial_{p^k}\partial_{p^j}G(p,w)|.$$ Thus, to prove that the operators $\mathcal{K}_{i,j}^k$,$\mathcal{L}_{i,1}^k$, and $\mathcal{L}_{i,2}^k$, $i,j,k=1,2$ satisfy the bounds of the form (\ref{k-bounds-1}), (\ref{l-bounds-h1}), and (\ref{l-bounds-1}) respectively, it will suffice to show that for any $f\in L^2(D(r))$, the function
\begin{align*}
T(f)(p)&:=\int_{D(r)}G(p,w)\psi_{p,i}(w)f(w)\,dw\\
\end{align*}
lies in $H^3(D(r))$, and $||T(f)||_{H^3(D(r))}\le C||f||_{L^2(D(r))}$ for some universal constant $C>0$ (depending on $\delta_0$ in Case 2). In particular, 
\begin{lemma}\label{T-L2-H3-bound} For $f\in L^2(D(r))$, the operator $T$ given by
\begin{align*}
T(f)(p)&:=\int_{D(r)}G(p,w)\psi_{p,i}(w)f(w)\,dw\\
\end{align*}
maps $L^2(D(r))\to H^3(D(r))$. Moreover, $T$ satisfies 
\begin{align}
||T(f)||_{H^j(D(r))}&\le Cr^{3-j}||f||_{L^2(D(r))}\label{T-est-hj}
\end{align}
 for $j\in\{0,1,2\}$ and some universal constant $C>0$ independent of $r$ and $\epsilon_0$.
\end{lemma}

\begin{proof}
Notice that by Lemma \ref{psi-lem-thin} the functions $\psi_{p,i}$ vanish at the point $p$, so
\begin{align}
\Delta_{g_\EE}T(f)(p) 
&=\psi_{p,i}(p)f(p) +\int 2\delta^{jk}\partial_{p^j}G(p,w)\partial_{p^k}\psi_{p,i}(w)f(w)\,dw\notag\\
&\qquad +\int G(p,w)\Delta_{g_\EE}\psi_{p,i}(w)f(w)\,dw\notag\\
&= 2T_1(f)(p)+ T_2(f)(p),\label{T-function-delta}
\end{align}
where
\begin{align}
T_1(f)(p)&=\int \delta^{jk}\partial_{p^j}G(p,w)\partial_{p^k}\psi_{p,i}(w)f(w)\,dw,\label{T-function-delta1}\\
T_2(f)(p)&=\int G(p,w)\Delta_{g_\EE}\psi_{p,i}(w)f(w)\,dw,\label{T-function-delta2}
\end{align}
and all the integrals here and below are computed over $D(r)$. Further, since $G(p,w)=0$ on $\partial D(r)$, so too $T(f)(p)=0$ on $\partial D(r)$. It thus will suffice for us to show that the functions $T_1(f)$ and $T_2(f)$
\begin{align}
T_1(f)(p)&=\int \delta^{jk}\partial_{p^j}G(p,w)\partial_{p^k}\psi_{p,i}(w)f(w)\,dw,\label{T-function-delta1}\\
T_2(f)(p)&=\int G(p,w)\Delta_{g_\EE}\psi_{p,i}(w)f(w)\,dw,\label{T-function-delta2}
\end{align}
are uniformly bounded in $H^{1}(D(r))$ in terms of the $L^2(D(r))$ norm of $f$, and with a constant whose dependence on the various parameters is as above. \\

We first show $||T_2(f)||_{L^2}\le Cr^3\log r||f||_{L^2}$ and $||T_2(f)||_{H^1}\le Cr^2(1+\log r)||f||_{L^2}$ for $f\in L^2(D(r))$. Observe
\begin{align*}
||T_2(f)||_{L^2} 
&\le C\sup_p||\Delta_{g_\EE}\psi_{p,i}(w)||_{C^0_w}\left|\left|\log||w||_{g_\EE}\right|\right|_{L^1}||f(w)||_{L^2}\\
&\le Cr^2\log r\sup_p||\Delta_{g_\EE}\psi_{p,i}(w)||_{C^0_w}||f||_{L^2},
\end{align*}
where $C>0$ is a uniform constant. From the estimates for $\psi_{p,i}$ in Lemma \ref{psi-lem-thin}, $||\Delta_{g_\EE}\psi_{p,i}(w)||_{C^0_w}\le Cr$ for some constant $C$ (which depends on $\delta_0$ in the second case of Theorem \ref{global_metric_thm}). Thus $T_2:L^2(D(r))\to L^2(D(r))$, and $||T_2(f)||_{L^2}\le Cr^3\log r||f||_{L^2}$. Additionally, for $j=1,2$
\begin{align*}
\partial_{p^j}T_2(f)(p)&=\int\partial_{p^j}G(p,w)\Delta_{g_\EE}\psi_{p,i}(w)f(w)+G(p,w)\partial_{p^j}\Delta_{g_\EE}\psi_{p,i}(w)f(w)\,dw.
\end{align*}
We estimate
\begin{align*}
\left|\left|\int\partial_{p^j}G(p,w)\Delta_{g_\EE}\psi_{p,i}(w)f(w)\, dw\right|\right|_{L^2}
&\le C\sup_p||\Delta_{g_\EE}\psi_{p,i}(w)||_{C^0_w}\left|\left|\frac{1}{||w||_{g_\EE}}\right|\right|_{L^1}||f(w)||_{L^2}\\
&\le Cr^2||f||_{L^2},
\end{align*}
and similarly
\begin{align*}
\left|\left|\int G(p,w)\partial_{p^j}\Delta_{g_\EE}\psi_{p,i}(w)f(w)\, dw\right|\right|_{L^2}
&\le C\sup_p||\partial_{p^j}\Delta_{g_\EE}\psi_{p,i}(w)||_{C^0_w}\left|\left|\log||w||_{g_\EE}\right|\right|_{L^1}||f(w)||_{L^2}\\
&\le Cr^2\log r||f||_{L^2},
\end{align*}
where we used $||\partial_{p^j}\Delta_{g_\EE}\psi_{p,i}(w)||_{C^0_w}\le C$ for some constant $C$ (see Lemma \ref{psi-lem-thin}). 
In summary, $T_2:L^2(D(r))\to H^1(D(r))$ and $||T_2(f)||_{H^1}\le Cr^2(1+\log r)||f||_{L^2}$.

As above, we have $T_1:L^2(D(r))\to L^2(D(r))$ and $||T_1(f)||_{L^2}\le Cr||f||_{L^2}$.

Next, we prove $\partial_{p^l}T_1:L^2(D(r))\to L^2(D(r))$, $l=1,2$, and $\partial_{p^l}||T_1(f)||_{L^2}\le C||f||_{L^2}$. We have 
\begin{align*}
\partial_{p^l}T_1(f)(p)&=\int\delta^{jk}\partial_{p^j}\partial_{p^l}G(p,w)\partial_{p^k}\psi_{p,i}(w)f(w)+\delta^{jk}\partial_{p^j}G(p,w)\partial_{p^k}\partial_{p^l}\psi_{p,i}(w)f(w)\,dw\\
&=: I_1(f)(p)+I_2(f)(p).
\end{align*}
First, in view of Lemma \ref{psi-lem-thin}, the Taylor expansion of $\partial_{p^j}\psi_{p,i}$ about $p=w$ yields 
$$ \partial_{p^j}\psi_{p,i}(w)= -\delta_{ij}+ \nabla_{p}\partial_{p^j}\psi_{p,i}(p)\cdot(w-p)+(w-p)\cdot\nabla_p\nabla_p\partial_{p^j}\psi_{p,i}(p)\cdot(w-p)+\BO(||w-p||_{g_\EE}^3),$$
 with $||\nabla_{p}\partial_{p^j}\psi_{p,i}(p)||_{g_\EE}\le Cr$ and $||\nabla_{p}\nabla_{p}\partial_{p^j}\psi_{p,i}(p)||_{g_\EE}\le C$ for some constants (which depend on $\delta_0$ in the Class 2  case of Theorem \ref{global_metric_thm}). Then,
\begin{align*}
I_1(f)(p)&=-\int\delta^{j}_i\partial_{p^j}\partial_{p^l}G(p,w)f(w)\,dw\\
&\quad +\int\delta^{jk}\partial_{p^j}\partial_{p^l}G(p,w)\nabla_{p^a}\nabla_{p^b}\psi_{p,i}(p)[\delta^a_k(w-p)^b+\delta^b_k(w-p)^a]f(w)\, dw\\
&\qquad +\int\delta^{jk}\partial_{p^j}\partial_{p^l}G(p,w)\cdot\BO(||w-p||_{g_\EE}^2)f(w)\,dw\\
&=: J_1(f)(p)+J_2(f)(p)+J_3(f)(p).
\end{align*}
 
Since $G(p,w)$ is the Dirichlet Green's function over $D(r)$, there is a uniform constant such that $$||J_1(f)||_{L^2}\le C||f||_{L^2}.$$ 
The operators $J_2$, $J_3$ and $I_2$ are weakly singular, and bounded on $L^2(D(r))$ by estimates similar to those for $T_2$ above. Hence $\partial_{p^l}T_1:L^2(D(r))\to L^2(D(r))$ and $||\partial_{p^l}T_1(f)||_{L^2}\le C||f||_{L^2}$.

Therefore, we have $$||\Delta_{g_\EE}T(f)(p)||_{H^1}= ||2T_1(f)(p)+T_2(p)(f)(p)||_{H^1}\le C||f||_{L^2}.$$ By standard elliptic regularity, $T:L^2(D(r))\to H^{3}(D(r))$ with $||T(f)||_{H^3(D(r))}\le C||f||_{L^2(D(r))}$ for some universal constant $C>0$ independent of $r$. 

From the $L^2(D(r))$ estimates for $T_1$ and $T_2$ we have
$$||\Delta_{g_\EE}T(f)(p)||_{L^2}= ||2T_1(f)(p)+T_2(p)(f)(p)||_{L^2}\le Cr[1+r^2\log r]||f||_{L^2}.$$
We now use this estimate to prove (\ref{T-est-hj}).

The inverse $G$ of the Dirichlet Laplacian on $D(r)$ satisfies
\begin{align*}
||G(f)||_{H^j(D(r))}&\le Cr^{2-j}||f||_{L^2(D(r))}
\end{align*}
for $j\in\{0,1,2\}$. Thus,
\begin{align*}
||T(f)||_{H^j(D(r))} &= ||G\Delta_{g_\EE}T(f)||_{H^j(D(r))}\le Cr^{2-j}||G\Delta_{g_\EE}T(f)||_{L^2(D(r))}\le Cr^{3-j}||f||_{L^2(D(r))},
\end{align*}
which proves estimate (\ref{T-est-hj}).

\end{proof}


Together with Lemma \ref{T-L2-H3-bound} above, the $L^\infty(D(r))$ estimates (\ref{F.function.est11})--(\ref{F.function.est31}) or (\ref{F.function.est12})--(\ref{F.function.est32}), for the functions $\bar{A}^{jkl}_{m}, \dots,{\bar{F}}^{kl}_{2m}$, $i,j,k,l,m\in\{1,2\}$ and their first and second derivatives allow us to prove the estimates (\ref{k-bounds-1}), (\ref{l-bounds-h1}), and (\ref{l-bounds-1}) for the operators $\mathcal{K}_{i,j}^k$ and $\mathcal{L}_{i,1}^k$,  $\mathcal{L}_{i,2}^k$, $i,j,k\in\{0,1,2\}$ respectively. For example, a term in the operator  $\mathcal{K}_{i,j}^k$ has a kernel of the form 
\begin{align}
U(p,w) &= \partial_{w}^j (G(p,w) \psi_p(w)) \partial_{w}^l \bar{A}(w),\label{kernel-form-1}
\end{align}
with $j+l\le2$. Consequently, given the above bounds for the operator $T$ holds, the norm of the operator $U$ with kernel $U(p,w)$ from $L^2(D(r))$ into $H^{3-j}(D(r))$ satisfies
\begin{align}
||U(f)||_{H^{3-j}}&\le ||\partial_{w}^l \bar{A}||_{L^\infty}||\partial^j_pT||_{L^2\to H^{3-j}}||f||_{L^2}.\label{kernel-estimate-1}
\end{align}
Proceeding term by term, Lemma \ref{T-L2-H3-bound} and estimates of the form (\ref{kernel-estimate-1}) prove (\ref{k-bounds-1}), (\ref{l-bounds-h1}), and (\ref{l-bounds-1}). To show (\ref{k-bounds-2}), for $h\in H^2(D(r))$, we note that the kernels of the terms in $\partial_p^2\mathcal{K}_{i,j}^k(h)(p)$ are of the form  
\begin{align*}
\partial_{p}^2\partial_{w}^j (G(p,w) \psi_p(w)) \partial^l\bar{A}(w).
\end{align*}
for $j+l=2$. When $l\ne0$, Lemma \ref{T-L2-H3-bound} and estimates of the form (\ref{kernel-estimate-1}) give the desired bound. Thus, to prove  (\ref{k-bounds-2}), when the kernel has a term with $l=0$ it suffices to show that for $i,j\in{1,2}$, the operator
\begin{align}
U(h)(p) &:= \int_{D(r)}\partial_{w^i}\partial_{w^j} (G(p,w) \psi_p(w))\bar{A}(w)h(w)\, dw
\end{align}
maps $H^2(D(r))\to H^2(D(r))$, 
 bounded. 

 
 %
\begin{lemma} We have 
\begin{align}
||U(h)||_{H^2(D(r))}\le C||h||_{H^2(D(r))}.\label{U-estimate}
\end{align}
\end{lemma}

\begin{proof}
Integration by parts twice yields
\begin{align}
U(h)(p) &= \int_{D(r)} G(p,w) \psi_p(w) \partial_{w^i}\partial_{w^j} [\bar{A}(w)h(w)]\,dw \\
&\qquad + \int_{\partial D(r)}\nu^i\partial_{w^j}[G(p,w) \psi_p(w) ]\bar{A}(w)h(w)\, dS\\
&:=  U_1(h)(p) + U_2(h)(p).
\end{align}
Above $i,j\in\{1,2\}$, and each of the functions $\nu^i(w)$ are equal to either ${\rm cos}(\theta)$ or ${\rm sin}(\theta)$. 
Notice that we have only one nontrivial boundary integral $U_2(h)(p)$ as $G(p,w)=0$ for 
$w\in\partial D(r)$.

Consider the operator $U_1(h)$. We have
\begin{align}
U_1(h)(p) = T(\partial_{w^i}\partial_{w^j} [\bar{A}(w)h(w)])
\end{align} 
where $T:L^2\to H^3$ is defined in Lemma \ref{T-L2-H3-bound}. Therefore, inequality 
(\ref{U-estimate}) for $U_1$ follows from Lemma \ref{T-L2-H3-bound} and the estimates 
(\ref{F.function.est11})--(\ref{F.function.est31}) or respectively (\ref{F.function.est12})--
(\ref{F.function.est32}), for the function $\bar{A}(w)$ and its derivatives.

Now consider $U_2(h)(p)$. Since $G(p,w)$ and its tangential derivative are zero for $w\in\partial D(r)$, we can rewrite $U_2(h)$ as
\begin{align*}
U_2(h)&=\int_{\partial D(r)}\nu^i(w)\nu^j(w)\frac{\partial G(p,w)}{\partial \nu(w)} \psi_p(w) [\bar{A}(w)h(w)]z(\theta)\, dS(w)
\end{align*}
Let $v_{ij}(w)$ be the harmonic extension of $\nu^i(w)\nu^j(w) = \frac{w^iw^j}{r^2}$ to $D(r)$. Then, using Green's identity we have
\begin{align*}
 U_2(h)&= v_{ij}(p)\psi_p(p)\bar{A}(p)h(p) -  \int_{D(r)} G(p,w)\Delta_{g_\EE}[  v_{ij}(w)\psi_p(w) \bar{A}(w)h(w)]\, dw.
\end{align*}

The first term vanishes since $\psi_p(p)=0$. The integral term is bounded on $H^2(D(r))$ uniformly in $r$ by arguments similar to those in Lemma \ref{T-L2-H3-bound}. Here we use $|v_{ij}(w)|\le 1$, $|\nabla v_{ij}(w)|\le \frac1r$, as well as the pointwise estimates on $ \psi_p(w)$ (see Lemma \ref{psi-lem-thin}), $ \bar{A}$ (see (\ref{F.function.est11})--(\ref{F.function.est31}) or (\ref{F.function.est12})--(\ref{F.function.est32})), and their derivatives up to second order.

\end{proof}


Employing this strategy to each term in appearing in the kernels (\ref{K-kernel-1})--(\ref{L-kernel-2}), in the close to Euclidean case we have:
\begin{align}
\mathcal{K}_{i,j}^k&:L^2(D(r))\to H^1(D(r)), && ||\mathcal{K}_{i,j}^k(f)||_{H^1}\le C\epsilon_0||f||_{L^2},\label{k-bounds-1-3e}\\
\mathcal{K}_{i,j}^k&:H^2(D(r))\to H^2(D(r)), && ||\mathcal{K}_{i,j}^k(f)||_{H^2}\le C\epsilon_0||f||_{H^2},\label{k-bounds-2-3e}\\
\mathcal{L}_{i,1}^k&:L^2(D(r))\to H^1(D(r)), && ||\mathcal{L}_{i,1}^k(f)||_{H^1}\le C\epsilon_0||f||_{L^2},\label{l-bounds-1-3e}\\
\mathcal{L}_{i,2}^k&:L^2(D(r))\to H^1(D(r)), && ||\mathcal{L}_{i,2}^k(f)||_{H^1}\le C\epsilon_0||f||_{L^2},\\
\mathcal{L}_{i,2}^k&:L^2(D(r))\to H^2(D(r)), && ||\mathcal{L}_{i,2}^k(f)||_{H^2}\le C\epsilon_0||f||_{L^2}\label{l-bounds-2-3e},
\end{align}
where $C>0$ is a uniform constant independent of $\epsilon_0$.

In the $(K,\epsilon_0,\delta_0)$-thin case, from the estimates (\ref{F.function.est12})--(\ref{F.function.est32}) the differentiated terms $\partial_{w}^\alpha\bar{A}^{jkl}_{m}$, $\dots$,$\partial_{w}^\alpha{\bar{F}}^{kl}_{2m}$, $i,j,k,l,m\in\{1,2\}$, $\alpha\in\{1,2,3\}$ in each of the kernels (\ref{K-kernel-1})--(\ref{L-kernel-2}) will satisfy $L^\infty(D(r))$ bounds.

 Then by Lemma \ref{T-L2-H3-bound}, we bound each of the operators $\mathcal{K}_{i,j}^k$ and $\mathcal{L}_{i,j}^k$ term by term to obtain 
\begin{align}
\mathcal{K}_{i,j}^k&:L^2(D(r))\to L^2(D(r)), && ||\mathcal{K}_{i,j}^k(f)||_{L^2}\le C\epsilon_0^{-2}r^3||f||_{L^2},\label{k-bounds-1-thin}\\
\mathcal{K}_{i,j}^k&:L^2(D(r))\to H^1(D(r)), && ||\mathcal{K}_{i,j}^k(f)||_{H^1}\le C||f||_{L^2},\\
\mathcal{K}_{i,j}^k&:H^2(D(r))\to H^2(D(r)), && ||\mathcal{K}_{i,j}^k(f)||_{H^2}\le C\epsilon_0^{-2}r^3||f||_{H^2},
\end{align}
\begin{align}
\mathcal{L}_{i,1}^k&:L^2(D(r))\to L^2(D(r)), && ||\mathcal{L}_{i,1}^k(f)||_{L^2}\le C\epsilon_0^{-2}r^3||f||_{L^2},\label{l-bounds-1-thin}\\
\mathcal{L}_{i,1}^k&:L^2(D(r))\to H^1(D(r)), && ||\mathcal{L}_{i,1}^k(f)||_{H^1}\le C||f||_{L^2},
\end{align}
\begin{align}
\mathcal{L}_{i,2}^k&:L^2(D(r))\to L^2(D(r)), && ||\mathcal{L}_{i,2}^k(f)||_{L^2}\le C\epsilon_0^{-1}r^3|| f||_{L^2},\\
\mathcal{L}_{i,2}^k&:L^2(D(r))\to H^1(D(r)), && ||\mathcal{L}_{i,2}^k(f)||_{H^1}\le C\epsilon_0^{-1}r^2|| f||_{L^2},\\
\mathcal{L}_{i,2}^k&:L^2(D(r))\to H^2(D(r)), && ||\mathcal{L}_{i,2}^k(f)||_{H^2}\le C||f||_{L^2},\label{l-bounds-2-thin}
\end{align}
where the constant $C>0$ is independent of $r, \epsilon_0$ and in the second class depends only on $\delta_0$ as in Definition \ref{thin-defn}.

With the above estimates in hand, we may define the relevant inverse operators to solve for $\delta g^{31}$ and $\delta g^{32}$ in terms of $\delta\phi$, $\partial_3 \delta\phi$.  Indeed, the equations (\ref{dgEQ1}), (\ref{dgEQ2}), which describe $\delta g^{31}$ and $\delta g^{32}$, can be written in terms of the operators $\mathcal{K}_{i,j}^k$, and $\mathcal{L}_{i,1}^k$, $i,j,k=1,2$, via the system
\begin{align}
[I-\mathcal{K}](\delta g^{31},\delta g^{32}) &= \mathcal{L}(\delta\phi,\partial_3 \delta\phi)\label{matrixdg}
\end{align}
where  
\begin{align*}
I&= \begin{pmatrix} 1 &0\\ 0&1\end{pmatrix},\\
\varepsilon_k&:=\partial_k||\nabla x^3||_{g_1},\\
\mathcal{K} &= \begin{pmatrix} \varepsilon_k\mathcal{K}_{1,1}^k & \varepsilon_k\mathcal{K}_{1,2}^k\\
\varepsilon_k\mathcal{K}_{2,1}^k & \varepsilon_k^k\mathcal{K}_{2,2}^k\end{pmatrix},\\
\mathcal{L} &= \begin{pmatrix} \varepsilon_k\mathcal{L}_{1,1}^k & \varepsilon_k\mathcal{L}_{1,2}^k\\
\varepsilon_k\mathcal{L}_{2,1}^k & \varepsilon_k\mathcal{L}_{2,2}^k\end{pmatrix},
\end{align*}
$k\in\{1,2\}$.

To solve this system for $\delta g^{31}$, $\delta g^{32}$ in terms of $\delta\phi$, $\partial_3 \delta\phi$, we need to invert $I-\mathcal{K}.$ It will suffice to show that the operators $\varepsilon_k\mathcal{K}_{i,j}^k$, $i,j,k=1,2$, have small norms as operators $L^2(D(r))\to L^2(D(r))$. We show the necessary smallness requirement in each case of Theorem \ref{global_metric_thm}. Further, for $i,j=1,2$ and $k\in\{1,2\}$, we will derive $H^1(D(r))$ estimates for $\varepsilon_k\mathcal{K}_{i,j}^k$ and $\varepsilon_k\mathcal{L}_{i,1}^k$, as well as $H^2(D(r))$ estimates for $\varepsilon_k\mathcal{L}_{i,2}^k$. The estimates for $\varepsilon_k\mathcal{K}_{i,j}^k$ we use to invert the system (\ref{matrixdg}), and also together with the estimates for $\varepsilon_k\mathcal{L}_{i,2}^k$ we prove (\ref{p-est})--(\ref{dq-est}).

Consider $\varepsilon_k\mathcal{K}_{i,j}^k$, and let $f\in L^2(D(r))$. Then, 
\begin{align*}
|| \varepsilon_k\mathcal{K}_{i,j}^k(f)||_{L^2} &\le 
||\varepsilon_k||_{L^\infty}|| \mathcal{K}_{i,j}^k(f)||_{L^2}
\end{align*}
and
\begin{align*}
|| \partial(\varepsilon_k\mathcal{K}_{i,j}^k)(f)||_{L^2}
&\le ||\varepsilon_k||_{L^\infty}|| \partial\mathcal{K}_{i,j}^k(f)||_{L^2}+ ||\partial\varepsilon_k||_{L^\infty}||\mathcal{K}_{i,j}^k(f)||_{L^2}.
\end{align*}
Note 
\begin{align}
|\varepsilon_k(p)|&:=|\partial_k||\nabla x^3||_{g_1}(p)|= |\partial_k\sqrt{g_1^{33}}(p)| \le C||\nabla g^{33}_1||_{L^\infty},\label{lapse-est-1}\\
|\partial \varepsilon_k(p)|& \le C||\nabla\nabla g^{33}_1||_{L^\infty},\label{lapse-est-2}
\end{align}
for some constant universal constant $C$ independent.
Hence,
\begin{align}
|| \varepsilon_k\mathcal{K}_{i,j}^k(f)||_{L^2} &\le C||\nabla g^{33}_1||_{L^\infty}||\mathcal{K}_{i,j}^k(f)||_{L^2}\label{K-est-l2}\\
|| \partial(\varepsilon_k\mathcal{K}_{i,j}^k)(f)||_{L^2}&\le C\left(||\nabla g^{33}_1||_{L^\infty}|| \partial\mathcal{K}_{i,j}^k(f)||_{L^2}+ ||\nabla\nabla g^{33}_1||_{L^\infty}|| \mathcal{K}_{i,j}^k(f)||_{L^2}\right),\label{K-est-h1}\\
|| \partial\partial(\varepsilon_k\mathcal{K}_{i,j}^k)(f)||_{L^2}&\le C\left(||\nabla g^{33}_1||_{L^\infty}|| \partial\partial\mathcal{K}_{i,j}^k(f)||_{L^2}+ 2||\nabla\nabla g^{33}_1||_{L^\infty}|| \partial\mathcal{K}_{i,j}^k(f)||_{L^2}\right.\notag\\
&\qquad\left. +||\nabla\nabla\nabla g^{33}_1||_{L^\infty}|| \mathcal{K}_{i,j}^k(f)||_{L^2}\right).\label{K-est-h2}
\end{align}
By analogous computations, we obtain estimates for the operators $\mathcal{L}_{i,j}$, $i,j=1,2$ (for some universal constant $C$): 
 \begin{align}
||\varepsilon_k\mathcal{L}_{i,j}^k(f)||_{L^2}&\le C||\nabla g^{33}_1||_{L^\infty}||\mathcal{L}_{i,j}^k(f)||_{L^2},\label{L-est-l2}\\
||\partial(\varepsilon_k\mathcal{L}_{i,j}^k)(f)||_{L^2}&\le C\left(||\nabla g^{33}_1||_{L^\infty}|| \partial\mathcal{L}_{i,j}^k(f)||_{L^2}+ |||\nabla\nabla g^{33}_1||_{L^\infty}|| \mathcal{L}_{i,j}^k(f)||_{L^2}\right).\label{L-est-h1}\\
|| \partial\partial(\varepsilon_k\mathcal{L}_{i,j}^k)(f)||_{L^2}&\le C\left(||\nabla g^{33}_1||_{L^\infty}|| \partial\partial\mathcal{L}_{i,j}^k(f)||_{L^2}+ 2||\nabla\nabla g^{33}_1||_{L^\infty}|| \partial\mathcal{L}_{i,j}^k(f)||_{L^2}\right.\notag\\
&\qquad\left. +||\nabla\nabla\nabla g^{33}_1||_{L^\infty}|| \mathcal{L}_{i,j}^k(f)||_{L^2}\right).\label{L-est-h2}
\end{align}

Now we consider the first case of Theorem \ref{global_metric_thm}. Recall in this case, we are assuming that the metric $g_1$ and $g_2$ are $C^3$-close to Euclidean. 
From (\ref{k-bounds-1-3e}) and (\ref{K-est-l2})--(\ref{K-est-h1}), the operators $\varepsilon_k\mathcal{K}_{i,j}: L^2\to L^2$ have norm controlled by $\epsilon_0$, which is small. Thus $I-\mathcal{K}: L^2\to L^2$ is invertible. 

In the second case of Theorem \ref{global_metric_thm}, we are assuming that the metric $g_1$ and $g_2$ are $(K,\epsilon_0,\delta_0)$-thin for some $K>0$ and sufficiently small $\delta_0,\epsilon_0>0$. Using (\ref{k-bounds-1-thin}) and (\ref{K-est-l2})--(\ref{K-est-h1}) together with the bounds for the operators $\varepsilon_k\mathcal{K}_{i,j}: L^2\to L^2$, obey 
\begin{align}
|| \varepsilon_k\mathcal{K}_{i,j}^k(f)||_{L^2} &\le C||\nabla g^{33}_1||_{L^\infty}\epsilon_0^{-2}r^3||(f)||_{L^2}\\
&\le C||\nabla g^{33}_1||_{L^\infty}\epsilon_0^{-2}r^3||(f)||_{L^2}.
\end{align}
Since $0<r<\epsilon_0$, we derive that $$||\nabla g^{33}_1||_{L^\infty}
\epsilon^{-2}_0r^3\le K \epsilon_0,$$ which is sufficiently small; this implies that $I-\mathcal{K}: L^2\to L^2$ is invertible. 
 
Therefore, in both cases the system (\ref{matrixdg}) is solvable in terms of $\delta\phi$ and $\partial_3\delta\phi$:
\begin{align*}
\delta g^{31} &= P^1_{-1}(\delta\phi)+Q^1_{-2}(\partial_3\delta\phi),\\
\delta g^{32} &= P^2_{-1}(\delta\phi)+Q^2_{-2}(\partial_3\delta\phi),
\end{align*}
where $P^k_{-1}$, $Q^k_{-2}$, $k=1,2$ are respectively order $-1$ and $-2$ pseudodifferential operators in the tangential directions $\partial_k$, $k=1,2$, given by the compositions
\begin{align*}
\begin{pmatrix} P^1_{-1}(\delta\phi)\\ P^2_{-1}(\delta\phi)\end{pmatrix}&:=(I-\mathcal{K})^{-1}\mathcal{L}\begin{pmatrix}\delta\phi\\0\end{pmatrix},\\
\begin{pmatrix} Q^1_{-2}(\partial_3\delta\phi)\\ Q^2_{-2}(\partial_3\delta\phi)\end{pmatrix}&:=(I-\mathcal{K})^{-1}\mathcal{L}\begin{pmatrix}0\\\partial_3\delta\phi\end{pmatrix}.
\end{align*}


The estimates (\ref{K-est-l2})--(\ref{L-est-h1}) together with (\ref{k-bounds-1-3e})--(\ref{l-bounds-2-3e}) or respectively (\ref{k-bounds-1-thin})--(\ref{l-bounds-2-thin}) give the claimed inequalities (\ref{p-est}) and (\ref{q-est}). Indeed, under the $\epsilon_0$-$C^3$-close to Euclidean assumption, from (\ref{L-est-l2}), (\ref{k-bounds-1-3e})--(\ref{l-bounds-2-3e}), and the $L^2$ bound on $(I-\mathcal{K})^{-1}$, we have 
\begin{align*}
|| P^k_{-1}(\delta\phi)||_{L^2}&\le C\epsilon_0||\nabla g^{33}_1||_{L^\infty}||\delta\phi||_{L^2}\\
|| Q^k_{-2}(\partial_3\delta\phi)||_{L^2}&\le C\epsilon_0||\nabla g^{33}_1||_{L^\infty}||\partial_3\delta\phi||_{L^2}.
\end{align*}
By (\ref{L-est-l2}) and (\ref{L-est-h1}), the operator $\partial P^k_{-1}:L^2\to L^2$, $k=1,2$, satisfies
\begin{align*}
||\partial P^k_{-1}(\delta\phi)||_{L^2}&\le ||(I-\mathcal{K})^{-1}||_{L^2\to L^2}\left[||\partial(\varepsilon_j\mathcal{L}_{2,1}^j)(\delta\phi)||_{L^2}+||\partial(\varepsilon_j\mathcal{L}_{1,1}^j)(\delta\phi)||_{L^2}\right]\\
&\qquad+||\partial(I-\mathcal{K})^{-1}||_{L^2\to L^2}\left[||\varepsilon_j\mathcal{L}_{2,1}^j(\delta\phi)||_{L^2}+||\varepsilon_j\mathcal{L}_{1,1}^j(\delta\phi)||_{L^2}\right]\\
%
&\le C\epsilon_0\left(||\nabla\nabla g^{33}_1||_{L^\infty} + ||\nabla g^{33}_1||_{L^\infty}\right)||\delta\phi||_{L^2}\\
& \qquad+ C\epsilon_0\left(||\nabla\nabla g^{33}_1||_{L^\infty} + ||\nabla g^{33}_1||_{L^\infty}\right)||\nabla g^{33}_1||_{L^\infty}||\delta\phi||_{L^2},\\
&\le C\epsilon_0^2||\delta\phi||_{L^2},
\end{align*}
since 
\begin{align}
\partial(I-\mathcal{K})^{-1} &= (I-\mathcal{K})^{-1}\partial \mathcal{K}(I-\mathcal{K})^{-1};\label{KK-invert}
\end{align}
which is bounded from $L^2\to L^2$ since $(I-\mathcal{K})^{-1}$ and $\partial\mathcal{K}$ are bounded from $L^2\to L^2$.

Similarly, inequalities (\ref{L-est-l2}) and (\ref{L-est-h1}) imply for $\partial Q^k_{-2}:L^2\to L^2$, $k=1,2$, 
\begin{align*}
|| \partial Q^k_{-2}(\partial_3\delta\phi)||_{L^2}&\le||(I-\mathcal{K})^{-1}||_{L^2\to L^2}\left[||\partial(\varepsilon_j\mathcal{L}_{1,2}^j)(\partial_3\delta\phi)||_{L^2}+||\partial(\varepsilon_j\mathcal{L}_{2,2}^j)(\partial_3\delta\phi)||_{L^2}\right]\\
&\qquad+||\partial(I-\mathcal{K})^{-1}||_{L^2\to L^2}\left[||\varepsilon_j\mathcal{L}_{1,2}^j(\partial_3\delta\phi)||_{L^2}+||\varepsilon_j\mathcal{L}_{2,2}^j(\partial_3\delta\phi)||_{L^2}\right]\\
&\le C\epsilon_0\left(||\nabla\nabla g^{33}_1||_{L^\infty} + ||\nabla g^{33}_1||_{L^\infty}\right)||\delta\phi||_{L^2}\\
& \qquad+ C\epsilon_0\left(||\nabla\nabla g^{33}_1||_{L^\infty} + ||\nabla g^{33}_1||_{L^\infty}\right)||\nabla g^{33}_1||_{L^\infty}||\delta\phi||_{L^2},\\
&\le C\epsilon_0^2||\delta\phi||_{L^2},
\end{align*}
Further,
\begin{align}
||\partial\partial Q^k_{-2}(\partial_3\delta\phi)||_{L^2}&\le ||(I-\mathcal{K})^{-1}||_{L^2\to L^2}\left[||\partial\partial(\varepsilon_j\mathcal{L}_{1,2}^j)(\partial_3\delta\phi)||_{L^2}+||\partial\partial(\varepsilon_j\mathcal{L}_{2,2}^j)(\partial_3\delta\phi)||_{L^2}\right]\notag\\
&\quad+2||\partial(I-\mathcal{K})^{-1}||_{L^2\to L^2}\left[||\partial(\varepsilon_j\mathcal{L}_{1,2}^j)(\partial_3\delta\phi)||_{L^2}+||\partial(\varepsilon_j\mathcal{L}_{2,2}^j)(\partial_3\delta\phi)||_{L^2}\right]\notag\\
&\quad+||\partial\partial(I-\mathcal{K})^{-1}||_{H^2\to L^2}\left[||\varepsilon_j\mathcal{L}_{1,2}^j(\partial_3\delta\phi)||_{H^2}+||\varepsilon_j\mathcal{L}_{2,2}^j(\partial_3\delta\phi)||_{H^2}\right]\label{ddQ}
\end{align}
is an estimate for $\partial\partial Q^k_{-2}$, $k=1,2$. We claim that when the metrics are of Class 1,
\begin{align*}
|| \partial\partial Q^k_{-2}(\partial_3\delta\phi)||_{L^2}&\le C\epsilon_0||\partial_3\delta\phi||_{L^2}
\end{align*}
for a universal constant $C$.

To see this, recall $\mathcal{K}$ is comprised of the operators $\varepsilon_k\mathcal{K}_{i,j}^k$, which obey estimates (\ref{K-est-l2})--(\ref{K-est-h2}). This together with
\begin{align}
\partial\partial(I-\mathcal{K})^{-1} &= (I-\mathcal{K})^{-1}\partial\partial \mathcal{K}(I-\mathcal{K})^{-1}-2(I-\mathcal{K})^{-1}\partial \mathcal{K}(I-\mathcal{K})^{-1}\partial \mathcal{K}(I-\mathcal{K})^{-1}
\end{align}
 implies $\partial\partial(I-\mathcal{K})^{-1}:H^2\to L^2$ is bounded. From the argument below (\ref{KK-invert}), we know $(I-\mathcal{K})^{-1},\partial(I-\mathcal{K})^{-1}:L^2\to L^2$ are uniformly bounded. Then (\ref{K-est-l2})--(\ref{L-est-h2}) proves the claim.

This proves the inequalities (\ref{p-est1}), (\ref{q-est1}), and (\ref{dq-est1}). 
\newline

Now we show the inequalities (\ref{p-est})--(\ref{dq-est}) for the second case of Theorem \ref{global_metric_thm}.
From (\ref{L-est-l2}) together with (\ref{k-bounds-1-thin})--(\ref{l-bounds-2-thin}), plus the bound on $(I-\mathcal{K})^{-1}$ we have 
\begin{align*}
|| P^k_{-1}(\delta\phi)||_{L^2}&\le C||(I-\mathcal{K})^{-1}||_{L^2\to L^2}||\nabla g^{33}_1||_{L^\infty}\epsilon_0^{-2}r^3||\delta\phi||_{L^2}\\
&\le CK \epsilon_0^{-2}r^3 ||\delta\phi||_{L^2}\\
|| Q^k_{-2}(\partial_3\delta\phi)||_{L^2} &\le C||(I-\mathcal{K})^{-1}||_{L^2\to L^2}||\nabla g^{33}_1||_{L^\infty}\epsilon_0^{-1}r^3||\partial_3\delta\phi||_{L^2}\\
&\le CK\epsilon_0^{-1}r^3||\partial_3\delta\phi||_{L^2},
\end{align*}
where $C$ is a universal constant. By the same computation as above, using (\ref{K-est-l2})--(\ref{K-est-h1}), (\ref{L-est-l2})--(\ref{L-est-h1}), and (\ref{k-bounds-1-thin})--(\ref{l-bounds-2-thin}), $\partial P^k_{-1}:L^2\to L^2$, $k=1,2$, satisfies
\begin{align*}
||\partial P^k_{-1}(\delta\phi)||_{L^2}&\le ||(I-\mathcal{K})^{-1}||_{L^2\to L^2}[||\partial(\varepsilon_j\mathcal{L}_{2,1}^j)(\delta\phi)||_{L^2}+||\partial(\varepsilon_j\mathcal{L}_{1,1}^j)(\delta\phi)||_{L^2}]\\
&\qquad+||\partial(I-\mathcal{K})^{-1}||_{L^2\to L^2}[||\varepsilon_j\mathcal{L}_{2,1}^j(\delta\phi)||_{L^2}+||\varepsilon_j\mathcal{L}_{1,1}^j(\delta\phi)||_{L^2}]\\
&\le C\left(||\nabla\nabla g^{33}_1||_{L^\infty}\epsilon_0^{-2}r^3 + ||\nabla g^{33}_1||_{L^\infty}\right)||\delta\phi||_{L^2}\\
& \qquad+ C\left(||\nabla\nabla g^{33}_1||_{L^\infty}\epsilon_0^{-2}r^3 + ||\nabla g^{33}_1||_{L^\infty}\right)||\nabla g^{33}_1||_{L^\infty} \epsilon_0^{-2}r^3||\delta\phi||_{L^2},
%
%
\end{align*}
where $C$ is a universal constant and having used $0<r<\epsilon_0$ and the estimates for $||\nabla^kg^{33}_1||_{L^\infty}$ in Defintion \ref{thin-defn}. This shows (\ref{p-est}) in the $(K,\delta_0,\epsilon_0)$-thin case of Theorem \ref{global_metric_thm}.
 
Now, to prove (\ref{q-est}), we calculate as before
 \begin{align*}
|| \partial Q^k_{-2}(\partial_3\delta\phi)||_{L^2}&\le||(I-\mathcal{K})^{-1}||_{L^2\to L^2}[||\partial(\varepsilon_j\mathcal{L}_{1,2}^j)(\partial_3\delta\phi)||_{L^2}+||\partial(\varepsilon_j\mathcal{L}_{2,2}^j)(\partial_3\delta\phi)||_{L^2}]\\
&\qquad+||\partial(I-\mathcal{K})^{-1}||_{L^2\to L^2}[||\varepsilon_j\mathcal{L}_{1,2}^j(\partial_3\delta\phi)||_{L^2}+||\varepsilon_j\mathcal{L}_{2,2}^j(\partial_3\delta\phi)||_{L^2}],\\
&\le C\left(||\nabla\nabla g^{33}_1||_{L^\infty}\epsilon_0^{-1}r^3 + ||\nabla g^{33}_1||_{L^\infty}\epsilon_0^{-1}r^2\right)||\delta\phi||_{L^2}\\
& \qquad+ C\left(||\nabla\nabla g^{33}_1||_{L^\infty}\epsilon_0^{-2}r^3 + ||\nabla g^{33}_1||_{L^\infty}\right)||\nabla g^{33}_1||_{L^\infty} \epsilon_0^{-1}r^3||\delta\phi||_{L^2},
\end{align*}
where we used estimates (\ref{K-est-l2})--(\ref{K-est-h1}), (\ref{L-est-l2})--(\ref{L-est-h1}), and (\ref{k-bounds-1-thin})--(\ref{l-bounds-2-thin}). 

Additionally, to show (\ref{dq-est}), we use estimates (\ref{K-est-l2})--(\ref{K-est-h1}), (\ref{L-est-l2})--(\ref{L-est-h1}), and (\ref{k-bounds-1-thin})--(\ref{l-bounds-2-thin}) in (\ref{ddQ}) to obtain
\begin{align*}
|| \partial\partial Q^1_{-2}(\partial_3\delta\phi)||_{L^2}&\le C\left[||\nabla g^{33}_1||_{L^\infty} + 2||\nabla\nabla g^{33}_1||_{L^\infty}r^2\epsilon_0^{-1}+ ||\nabla\nabla\nabla g^{33}_1||_{L^\infty}r^3\epsilon_0^{-1}\right.\\
&\quad +2\left(||\nabla\nabla g^{33}_1||_{L^\infty}r^3\epsilon_0^{-2} + ||\nabla g^{33}_1||_{L^\infty}\right)\left( ||\nabla g^{33}_1||_{L^\infty}r^2\epsilon_0^{-1}+||\nabla\nabla g^{33}_1||_{L^\infty}r^3\epsilon_0^{-1}\right)\\
&\quad + \left(||\nabla g^{33}_1||_{L^\infty}r^3\epsilon_0^{-2} + 2||\nabla\nabla g^{33}_1||_{L^\infty}r^3\epsilon_0^{-1}+ ||\nabla\nabla\nabla g^{33}_1||_{L^\infty}r^3\right)\left(||\nabla g^{33}_1||_{L^\infty} \right.\\
&\qquad\left.+ 2||\nabla\nabla g^{33}_1||_{L^\infty}r^2\epsilon_0^{-1}+ ||\nabla\nabla\nabla g^{33}_1||_{L^\infty}r^3\epsilon_0^{-1}\right)\\
&\quad + \left(||\nabla g^{33}_1||_{L^\infty}r^3\epsilon_0^{-1} + ||\nabla\nabla g^{33}_1||_{L^\infty}r^3\epsilon_0^{-2}\right)^2\left(||\nabla g^{33}_1||_{L^\infty} \right.\\
&\qquad\left.\left.+ 2||\nabla\nabla g^{33}_1||_{L^\infty}r^2\epsilon_0^{-1}+ ||\nabla\nabla\nabla g^{33}_1||_{L^\infty}r^3\epsilon_0^{-1}\right)\right]||\partial_3\delta\phi||_{L^2}.
\end{align*}
\newline

Summarizing in the second case of Theorem \ref{global_metric_thm}, from the bounds for $||\nabla^kg^{33}_1||_{L^\infty}$ in the definition of $(K,\delta_0,\epsilon_0)$-thin, and since $0<r< \epsilon_0$, we have for some universal constant $C$
\begin{align}
||\partial P^k_{-1}(\delta\phi)||_{L^2} &\le CK||\delta\phi||_{L^2},\notag\\
||\partial Q^k_{-2}(\delta\phi)||_{L^2} &\le CK\epsilon_0||\partial_3\delta\phi||_{L^2} \label{Q-est-1}\\
||\partial\partial Q^k_{-2}(\delta\phi)||_{L^2} &\le CK||\partial_3\delta\phi||_{L^2}
%
\end{align}
which completes the proof.

\end{proof}

\subsubsection{Reduction to uniqueness for $\delta\phi$}
From Proposition \ref{psido_lem}, we readily obtain the following uniqueness result:

\begin{lemma}
If $\delta\phi\equiv0$ on $M$, then 
\begin{align*}
g^{11}_1&=g^{11}_2,\quad g^{22}_1=g^{22}_2,\quad g^{31}_1=g^{31}_2,\\
g^{32}_1 &= g^{32}_2,\quad g^{12}_1= g^{12}_2,\quad g^{21}_1= g^{21}_2,
\end{align*}
on $M$.
\end{lemma}

\begin{proof}
From Lemma \ref{g33-equiv}, $\delta g^{33}:=g^{33}_1-g^{33}_2=0$ on $M$.

Substituting $\delta\phi\equiv0$ into the pseudodifferential expressions (\ref{psido01}) and (\ref{psido02}), gives $\delta g^{31}=0$ and $\delta g^{32}=0$ on $M$. From simple algebraic equations for the other metric components, $\delta\phi=\delta g^{31}=\delta g^{32}=\delta g^{33}=0$ implies $g_1^{\alpha \beta}=g_2^{\alpha \beta}$, on $M$ for ${\alpha, \beta}=1,2,3$.

\end{proof}



In light of the above Lemma, to conclude $g_1=g_2$ in our chosen coordinates, it only remains to prove $\delta\phi\equiv0$ on $M$. Below we show that from the pseudodifferential expressions (\ref{psido01}) and (\ref{psido02}), the equation (\ref{meancurve3}) for $\delta\phi$ may be expressed as a hyperbolic Cauchy problem for $\delta\phi$ with initial data $\delta\phi=0$ on $Y(0)$. Then, using a standard energy estimate we prove $\delta\phi=0$ on $M$ as desired.

\subsubsection{A hyperbolic Cauchy problem for $\delta\phi$:}
Substituting the expressions (\ref{psido01}), (\ref{psido02}) into equation (\ref{meancurve3}) gives us the following evolution equation for $\delta\phi$ on $M$:
\begin{align}
0&= g^{33}_1\partial_3\delta\phi + g^{31}_1\partial_1(\delta\phi)+g^{32}_1\partial_2(\delta\phi)+\left(\partial_k\phi_2-\frac{1}{2}\partial_k\log(g^{33}_1)\right) P^k_{-1}(\delta\phi)\notag\\
&\qquad+\left(\partial_k\phi_2-\frac{1}{2}\partial_k\log(g^{33}_1)\right) Q^k_{-2}(\partial_3\delta\phi)+\frac{1}{2}\partial_k P^k_{-1}(\delta\phi)+\frac{1}{2}\partial_k Q^k_{-2}(\partial_3\delta\phi).\label{phiEQalmost}
\end{align}
Now since $P^k_{-1},Q^k_{-2}$ are pseudodifferential operators of order $-1$ and $-2$ respectively, $\partial_kP^k_{-1}(\delta\phi)$, $\partial_kQ^k_{-2}(\partial_3\delta\phi)$ are respectively order $0$  and $-1$ pseudodifferential operators in the tangential directions $\partial_k$.
So, equation (\ref{phiEQalmost}) takes the form
\begin{align}
(I-Q_0)\partial_3\delta\phi + Q_1(\delta\phi)&=0,\label{psido_c1}
\end{align}
where 
\begin{align}
Q_1 &= \frac{g^{31}_1}{g^{33}_1}\partial_1(\delta\phi)+\frac{g^{32}_1}{g^{33}_1}\partial_2+\frac{1}{g^{33}_1}\left(\partial_k\phi_2-\frac{1}{2}\partial_k\log(g^{33}_1)\right) P^k_{-1} +\frac{1}{2g^{33}_1}\partial_k P^k_{-1}
\end{align} 
is an order $1$ pseudodifferential operator and 
\begin{align}
Q_0 &= \frac{1}{g^{33}_1}\left(\partial_k\phi_2-\frac{1}{2}\partial_k\log(g^{33}_1)\right) Q^k_{-2} +\frac{1}{2g^{33}_1}\partial_k Q^k_{-2}
\end{align} 
is a pseudodifferential operator of order $0$, both of which act only in the tangential directions $\partial_1$ and $\partial_2$. Now, by Lemma \ref{conf.bounds} and Sobolev embedding, $\phi_2$ and $\partial_k\phi_2$ are uniformly bounded in $L^\infty(D(r))$. This together with  the bounds given in Lemma \ref{psido_lem} shows that $Q_1:L^2(D(r))\to L^2(D(r))$ is uniformly bounded. We now argue that $I-Q_0:L^2(D(r))\to L^2(D(r))$ is invertible.

From (\ref{dq-est}) in Proposition \ref{psido_lem},
\begin{align*}
|| \partial Q^k_{-2}(\partial_3\delta\phi)||_{L^2(D(r))}&\le CK\epsilon_0||\partial_3\delta\phi||_{L^2(D(r))},
\end{align*}
for $k=1,2$. Therefore, 
\begin{align*}
||(I-Q_0)\partial_3\delta\phi||_{L^2(D(r))}&\ge ||[I-CK\epsilon_0]\partial_3\delta\phi||_{L^2(D(r))}.
\end{align*}
Since $\epsilon_0>0$ is small, $(I-Q_0)$ is invertible.

Inverting $(I-Q_0)$, we derive a hyperbolic Cauchy problem for $\delta\phi$ of the form
\begin{align}
\partial_t\delta\phi + \tilde{Q}_1(\delta\phi)&=0 \quad\text{ on }M\label{hyperbolicEQ}\\
\delta\phi(t)&=0 \quad\text{ on }\partial M,\nonumber\\
\lim_{t\to-1}||\delta\phi(t)||_{H^1(D(r))}&=0 \nonumber
\end{align}
where $\tilde{Q}_1=(I-Q_0)^{-1}Q_1$ is an order 1 pseudodifferential operator in the tangential directions. Recall $x^3=t$.

\begin{lemma}
For $\delta\phi\in C^3(M)$, the Cauchy problem (\ref{hyperbolicEQ}) has a unique solution $\delta\phi\equiv0$.
\end{lemma}

\begin{proof}
We use a standard energy argument. First we show that $\tilde{Q}_1+\tilde{Q}_1^*:L^2(D(r))\to L^2(D(r))$ is bounded.
 Expanding,
\begin{align*}
\tilde{Q}_1+\tilde{Q}_1^* &= (I-Q_0)^{-1}Q_1+Q_1^*(I-Q_0^*)^{-1}\\
&= (I-Q_0)^{-1}[Q_1+Q_1^*] - (I-Q_0)^{-1}Q_1^*+Q_1^*(I-Q_0^*)^{-1}.
\end{align*}
Now, $Q_1= \frac{g^{31}_1}{g_1^{33}}\partial_1 + \frac{g^{32}_1}{g_1^{33}}\partial_2 + \text{l.o.t.}$, thus $Q_1+Q_1^*$ is a pseudodifferential operator of order 0 and is uniformly bounded from $L^2(D(r))\to L^2(D(r))$. Using Proposition \ref{psido_lem}, $(I-Q_0)^{-1}$ and $Q_1^*$  are uniformly bounded operators from $L^2(D(r))\to L^2(D(r))$, thus $(I-Q_0)^{-1}Q_1^*:L^2(D(r))\to L^2(D(r))$ is uniformly bounded.

It remains to show that $Q_1^*(I-Q_0^*)^{-1}:L^2(D(r))\to L^2(D(r))$ is uniformly bounded.

Since the leading order term in $Q_1^*$ is  a regular first order differential operator and
$$ \partial_j(I-Q_0^*)^{-1} = (I-Q_0^*)^{-1}\partial_j Q_0^* (I-Q_0^*)^{-1}, $$
it suffices to bound $\partial_j Q_0^*:L^2(D(r))\to L^2(D(r))$. By definition, $$\partial_j Q_0^* := \partial_j \left[\frac{1}{g^{33}_1}\left(\partial_k\phi_2-\frac{1}{2}\partial_k\log(g^{33}_1)\right) Q^k_{-2} +\frac{1}{2g^{33}_1}\partial_k Q^k_{-2}\right]^*,$$ where $j,k\in\{1,2\}$. By inspection of $\partial_j Q_0^*$ above, it suffices to uniformly bound $\partial_j Q^k_{-2}$ and $\partial_j\partial_k Q^k_{-2}$ from $L^2(D(r))\to L^2(D(r))$. These bounds are achieved by estimate $\ref{q-est}$ in Proposition \ref{psido_lem} for $Q^j_{-2}:L^2(D(r))\to H^2(D(r))$.

Now, let $(\delta\phi,\delta\phi):=||\delta\phi||_2^2:= \int_{D(r)} |\delta\phi(x)|^2\,d x^1dx^2$. Recall that $r=r(t)$. Differentiating with respect to $t$,
\begin{align*}
\partial_t||\delta\phi||_2^2 
&= \int_{\partial D(r(t))} |\delta\phi(x)|^2\,dS+ ([\tilde{Q}_1+\tilde{Q}_1^*](\delta\phi), \delta\phi)\\
&\le C(t)||\delta\phi||_2^2
\end{align*}
where we used Cauchy-Schwarz and the fact $\delta\phi=0$ on $\partial D(r(t))$ for each $t\in(-1,1)$. 

By  Gronwall's inequality,
\begin{align*}
\partial_t||\delta\phi(t)||_2^2 &\le C||\delta\phi(-1)||_2^2=0,
\end{align*}
therefore $\delta\phi\equiv0$ on $\bigcup_{t\in(-1,1)}D(r(t))$.

\end{proof}
With the above lemma in hand, we conclude $g_1=F^*(g_2)$ as desired. 
\end{proof}

\subsection{Proofs of Theorems \ref{local_metric_thm} and \ref{onion_thm}}

We now show a purely local result near a point on the boundary. 
\begin{figure}[h]
\centering
 \includegraphics[width=0.43\textwidth]{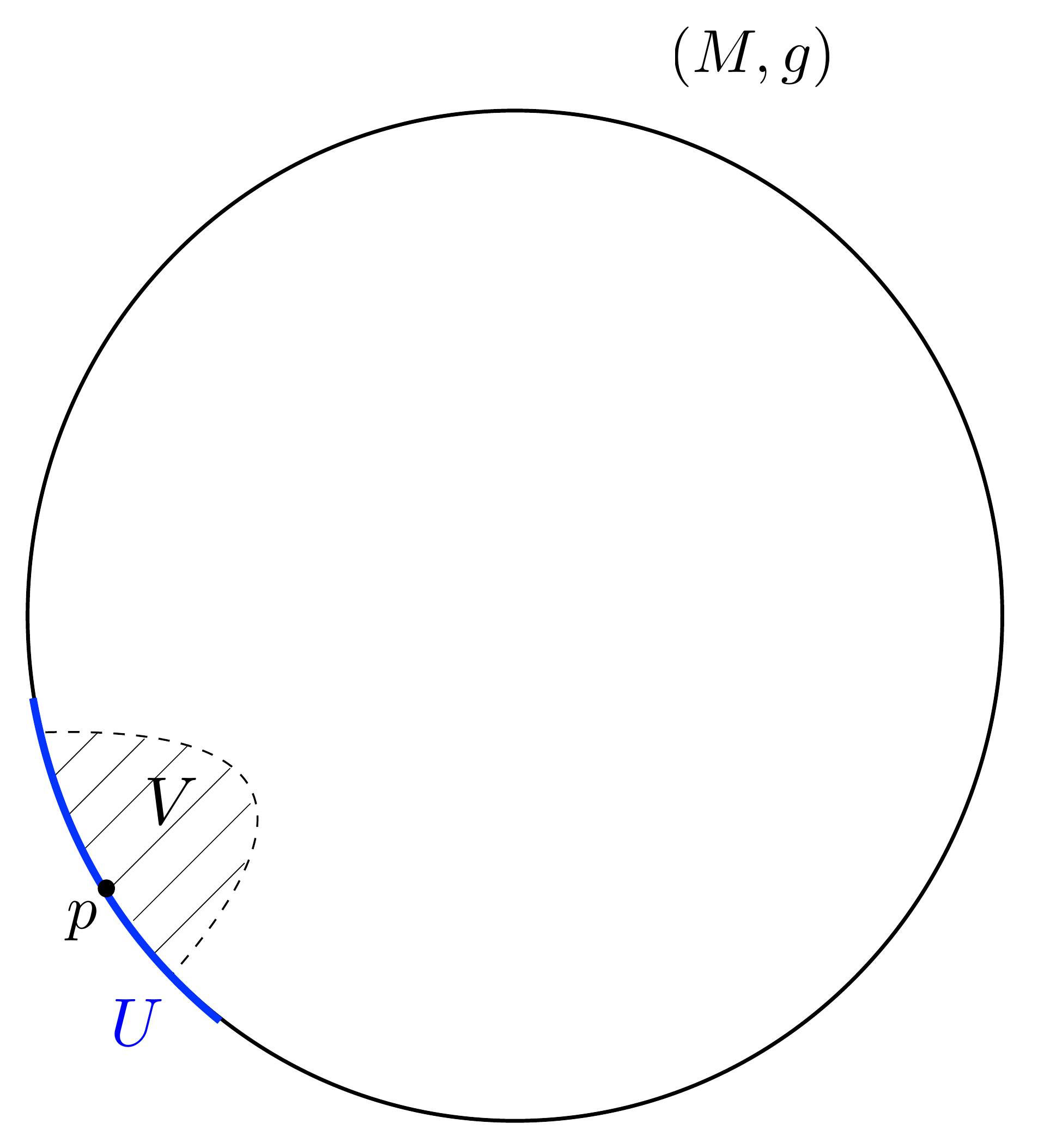}
  \caption{Neighbourhood near a point on the manifold $(M,g)$.}
\end{figure}

\begin{reptheorem}{local_metric_thm}
Let $(M,g)$ be a $C^4$-smooth, Riemannian manifold with boundary $\partial M$. Assume that $\partial M$ is both $C^4$-smooth and mean convex at $p\in \partial M$. Let $U\subset \partial M$ be a neighbourhood of $p$, and let $\{\gamma(t)\,:\,t\in(-1,1)\}$ be a foliation of $U$ by simple, closed curves which satisfy the estimates in Definition \ref{class1and2}. Suppose that $g|_{U}$ is known, and for each $\gamma(t)$ and any nearby perturbation $\gamma(s,t)\subset U$, we know the area of the properly embedded surface $Y(s,t)$ which solves the least-area problem for $\gamma(s,t)$. Then, there exists a neighbourhood $V\subset M$ of $p$ such that, up to isometries which fix the boundary, $g$ is uniquely determined on $V$.
\end{reptheorem}

\begin{proof}
Let $V\subset M$ be a neighbourhood near $p\in M$ for which we know the above area information. Further, we may choose $V$ sufficiently small so that $(V,g|_{V})$ is a $(K,\epsilon_0,\delta_0)$-thin manifold for some parameters $K,\epsilon_0,\delta_0>0$. Since $\partial M$ is mean convex at $p$, we may further choose $V$ so that $V\cap \partial M$ is mean convex. In this case, the induced foliation $Y(t)\subset V$ by properly embedded, area-minimizing curves is non-degenerate.

Applying Theorem \ref{global_metric_thm} to $(V,g|_{V})$, we may recover the metric 
$g$ near $p\in M$. 
\end{proof}

 Finally we sketch the proof of Theorem \ref{onion_thm}, which we restate below.
 
 \begin{reptheorem}{onion_thm} 
Let $(M,g)$ be a $C^4$-smooth Riemannian manifold which admits foliations from all directions, 
and let $g|_{\partial M}$ be given. Suppose that for all $p\in M$ and for each $\gamma(t,p)$ 
as in Definition \ref{foliations-all-directions}, and any nearby perturbation 
$\gamma(s,t,p)\subset \partial M$, we know the area of the properly embedded surface
 $Y(s,t,p)$ which solves the least-area problem for $\gamma(s,t,p)$.  \\

Then the knowledge of these areas uniquely determines the metric $g$ (up to isometries 
which fix the boundary).
\end{reptheorem}

\begin{proof}[Sketch of proof.] 
 
Consider two metrics $g_1, g_2$, as in the assumption of our Theorem. 
 Applying Theorem \ref{local_metric_thm} we derive that $g_1$
 restricted to 
 \[
 \bigcup_{r\in [0,\epsilon)}N(r),
 \]
 for some small $\epsilon$
is isometric to $g_2$ restricted on an open set of $M$ which has nonempty intersection with $\partial M$. 

We will then show that this is true for all $r_0<1$. This will conclude our argument.

 To do this, we can apply an open-closed argument: Let $R\subset [0,1)$ be the largest 
 connected set of  values for which $g_1$ is isometric to  a portion of $g_2$  over
 \[
 \bigcup_{r\in R}N(r). 
 \]
 We wish to show that $R=[0,1)$. The continuity of the metrics implies that $R$ is relatively 
 closed; it 
 remains to show it is open. We assume that $R=[0,r_0]$ with $r_0<1$, and  we will
  reach a contradiction. 
 
To derive the contradiction we need to show that for any point $p\in N(r_0)$ there is a 
small neighbourhood $\Omega$ 
 of $p$
 for which $g_1|_\Omega$ is isometric to a region of $(M,g_2)$ containing $p$. 
 
 Consider the minimal surface $Y(t_0,p)$ (for $g_1$) 
 for which $p\in Y(t_0,p)$ and $Y(t_0,p)$ is tangent 
 to $N(r_0)$ at $p$. 
We can employ the same strategy as in the previous theorems, considering ``tilts'' 
$Y(s,t,p)\subset M$
of the surfaces $Y(t,p)$ with $t-t_0$ suitably small 
 (see Lemma \ref{sep_lin_lem}). One difference  
are the bounds on the ``tilt'' functions $\psi_{p,i}$, $i=1,2$ (see Lemma \ref{psi-lem-thin}) 
can now have very large norms in $W^{3,p}$, in the absence 
of a thinness condition. Accordingly $g^{33}$, $|Rm|$, $|A|$ and their derivatives can be large in the relevant
norms. However 
restricting attention to $t\in [t_0, t_0+\zeta]$ for some $\zeta>0$ small enough, 
the corresponding functions $\delta\phi, \delta g^{3k}$ appearing in the analogous system 
(\ref{dgEQ1})--(\ref{meancurve3}) are now supported on 
discs of size\footnote{Here ``size'' is measured by the corresponding area.} 
$\eta>0$, where $\eta>0$ can be chosen to be arbitrarily small. 

  This (employing the argument in Lemma \ref{psido_lem}) 
  provides the invertibility of the operators $I-\mathcal{K}$
in equation (\ref{matrixdg}) and $(I-Q_0)$ in equation (\ref{psido_c1}) over 
  $Y(s,t,p)$ for $t\in [t_0, t_0+\zeta]$, with $\zeta>0$ and $|s|$ small 
  enough, yielding that $g_1, g_2$ are isometric over 
  
  \[
\left(\bigcup_{v\in [t_0,t_0+\eta)} Y(v,p)\right)\bigcap \left(\bigcup_{r\in [r_0,1)}
N(r)\right). 
 \]
This can be done for all $p\in N(r_0)$ (and the set of these points form a compact set),
 thus proving that $g_1,g_2$ are isometric up to 
\[
 \bigcup_{r\in [1,r_0+\eta_0)}N(r). 
 \]

\end{proof} 
 


\section{Appendix}\label{appendix}

   	

\subsection{Proof of Proposition \ref{C1.sees}} \label{alg-comp-g}


\begin{proof} The ``only if'' part of the statement is easy and follows by a standard perturbation argument for the minimal surface equation. Let us show that if $A(t)$ 
is $C^1$ then the area minimizing surfaces $Y(t)$ are unique for each $t\in (0,1)$. 

We argue by contradiction and assume that this is not the case. Then, there exists a $T<1$ for which two things hold: 

\begin{itemize}
\item The family of surfaces $Y(t), t\in (0,T)$ are the unique area-minimizing fill-ins 
of $\gamma(t)$ for $t<T$, 
and the surfaces $Y(t)$ converge to an area-minimizing surface $Y_{-}(T)$ which fills in 
$\gamma(T)$.

\item There exists a sequence $t_i, t_i>T$ with $t_i\to T^+$ so that some area-minimizing 
surfaces $Y(t_i)$ filling in $\gamma(t_i)$ converge to an area-minimizer 
$Y_+(T)$ filling in $\gamma(T)$ with $Y_+(T)\ne Y_{-}(T)$. 
\end{itemize}

 To see this, we observe that for some small $\epsilon>0$ the mean-convexity of the boundary  
 the fill-ins $Y(t), t<\epsilon$ are indeed unique, and thus define a foliation. We choose 
 $T$ to be the maximum value of $\tau\ge \epsilon$ for which 
 the family of surfaces $Y(t), t\in (0,\tau)$ are the unique area minimizers, and thus 
 define a $C^1$ foliation.
 
 It follows that the left limit of these surfaces $Y(t), t<\tau$ must be an area-minimizer;
 we denote this limit by $Y_{-}(T)$. The right limit is denoted similarly by $Y_{+}(T)$. Now, since $T$ was chosen to be maximal, 
 there must exist a sequence $t_i$ as described in the second requirement above. 
\medskip

Therefore, to derive our claim we need to show that 
\[
Y_{+}(T)\bigcap Y_{-}(T)=\emptyset.
\] 
This will hold if: 

\begin{equation}
\label{cont}
Y_{+}(T)\bigcap (\bigcup_{t<T}Y(t))=\emptyset.
\end{equation}
Equation \eqref{cont} follows by the maximum principle, using the fact that $Y(t)$, $t<T$ is a foliation. 
 \medskip
 
 Now, let $\nu_{-}$ be the inward pointing unit normal vector field along $\gamma(T)$ that is tangent to $Y_{-}(T)$, and $\nu_{+}$ the inward pointing unit normal vector field along $\gamma(T)$ that is tangent to $Y_{+}(T)$.
 Using the Hopf maximum principle, 
 we know that $\nu_{+}$ is transverse everywhere to $Y_{-}(T)$ and points above 
 $Y_{-}(T)$. 
 
 Let us now derive  a contradiction to the assumption that $A(t)$ is $C^1$-smooth in $t$. 
 On the one hand, using the first variation of area formula for minimal-surfaces, the left derivative of $A(t)$ at $t=T$ must equal:
 \[
 \int_{\gamma(T)} \nu_{-}\cdot\dot{\gamma}(T).
 \]
 Here we note that the vector $\dot{\gamma}(T)$ is tangent to $\partial M$, and captures the first variation in the foliation $\{\gamma(t)\,:\, t\in(-1,1)\}$ at $t=T$. This follows from the first property of $\gamma(T)$.
 From the second property of $\gamma(T)$ we find that the lim-sup of the 
right derivative of $A(t)$ at $T$ must be bounded below 
by  
  \[
 \int_{\gamma(T)} \nu_{+}\cdot\dot{\gamma}(T).
 \]
 Since the integral
 \[
 \int_{\gamma(T)}  (\nu_{-}-\nu_{+})\cdot\dot{\gamma}(T)
 \]
 is never zero, we have derived that 
 $A(t)$ cannot be differentiable at $t=T$ and we have derived our contradiction.  
 \end{proof}

\subsection{Algebraic Relationships Between the Components of $g$ and $g^{-1}$} \label{alg-comp-g}

Let $(x^\alpha)$, $\alpha =1,2,3$ be a local coordinate system on a Riemannian manifold $(M,g)$ such that in these coordinates the metric $g$ takes the form
\begin{align*}
g&= \begin{pmatrix} e^{2\phi} & 0 & g_{13}\\
                                   0 & e^{2\phi} & g_{23}\\
                                   g_{31} & g_{32} & g_{33}
                                   \end{pmatrix},
 \end{align*}
 where the functions $g_{13}=g_{31}$  and $g_{23}=g_{32}$.                                
Then, simple cofactor expansion gives
\begin{align*}
g^{-1}&:= \begin{pmatrix} g^{11} & g^{12} &g^{13}\\
                                   g^{21} & g^{22} & g^{23}\\
                                   g^{31} & g^{32} & g^{33}
                                   \end{pmatrix}\\
          &= -\det(g^{-1})\begin{pmatrix} \frac{e^{2\phi}g_{33} -(g_{32})^2}{e^{2\phi}} & \frac{g_{31}g_{32}}{e^{2\phi}} & -g_{13}\\
                                   \frac{g_{31}g_{32}}{e^{2\phi}} & \frac{e^{2\phi}g_{33} -(g_{31})^2}{e^{2\phi}} & -g_{23}\\
                                   -g_{31} & -g_{32} & e^{2\phi}
                                   \end{pmatrix}.
 \end{align*} 
Thus we have the following relationships:
\begin{align*}
\det(g^{-1})&= -\frac{e^{-2\phi}}{g^{33}},\\
g_{31} &= \frac{g^{31}}{\det(g^{-1})} = -\frac{g^{31}}{g^{33}}e^{2\phi},\\
g_{32} &= \frac{g^{32}}{\det(g^{-1})} = -\frac{g^{32}}{g^{33}}e^{2\phi}.
\end{align*}
Now, the determinant of $g^{-1}$ is
\begin{align*}
\det(g^{-1}) &= \det(g)^{-1}\\
&= \left[e^{2\phi}g_{33}+(g_{31})^2+(g_{32})^2\right]^{-1}.
\end{align*}
So 
\begin{align*}
g^{33}&= -\det(g^{-1})e^{-2\phi}\\
&=\frac{e^{-2\phi}}{\left[e^{2\phi}g_{33}+(g_{31})^2+(g_{32})^2\right]},
\end{align*}
and manipulating the above we obtain an expression for $g_{33}$ in terms of $g^{31},g^{32}$ and $\phi$:
\begin{align*}
g_{33} &= -\frac{1}{g^{33}} -\frac{(g_{31})^2+(g_{32})^2}{(g^{33})^2}e^{2\phi}.
\end{align*}

Therefore, we derive the following expressions for the components of $g^{-1}$ in the $\partial_1,\partial_2$ directions in terms of the functions $g^{33},g^{31},g^{32}$ and $\phi$:
\begin{align*}
g^{12} &=g^{21}= \frac{g^{31}g^{32}}{g^{33}},\\ 
g^{11}&=-e^{-2\phi}-\frac{(g^{13})^2+2(g^{23})^2}{g^{33}},\\
g^{22}&=-e^{-2\phi}-\frac{2(g^{13})^2+(g^{23})^2}{g^{33}}.
\end{align*}

Lastly, we may compute
\begin{align*}
e^{-2\phi}\left[g_{31}\partial_1g^{33}+g_{32}\partial_2g^{33}\right]&= e^{-2\phi}\left[-\frac{g^{31}}{g^{33}}e^{2\phi}\partial_1g^{33}-\frac{g^{32}}{g^{33}}e^{2\phi}\partial_2g^{33}\right]\\
&= -g^{31}\partial_1\log(g^{33})-g^{32}\partial_2\log(g^{33}).
\end{align*}

%
%




\addcontentsline{toc}{section}{Bibliography}




\end{document}